\documentclass[10pt]{amsart}
\usepackage{amsmath,amsfonts,amscd,amssymb,graphicx,mathrsfs,eufrak}
\usepackage[dvips,all,arc,curve,color,frame]{xy}

\newtheorem{defin}{Definition}[section]
\newtheorem{thm}[defin]{Theorem}
\newtheorem{prop}[defin]{Proposition}
\newtheorem{lemma}[defin]{Lemma}
\newtheorem{cor}[defin]{Corollary}

\newtheorem{example}[defin]{Example}
\newtheorem{remark}[defin]{Remark}

\newcommand{\Mod}[1]{\mathfrak{Mod}#1}

\newcommand{\stq}[1]{\underline{\Mod{#1}}}

\newcommand{\Sstq}[1]{\textsf{St}(\stq{R})}

\newcommand{\C}[1]{\mathbb{C}^{#1}}

\newcommand{\s}[1]{\mathscr{#1}}

\newcommand{\ve}{\varepsilon}

\newcommand{\w}[1]{\widetilde{#1}}
\newcommand{\e}{\varepsilon}

\renewcommand{\t}[1]{\textnormal{#1}}

\def\minus{\hbox{-}}
\def\plus{\hbox{+}}

\def\CM{\mathop{\rm CM}\nolimits}
\def\SCM{\mathop{\rm SCM}\nolimits}
\def\OCM{\mathop{\Omega{\rm CM}}\nolimits}
\def\depth{\mathop{\rm depth}\nolimits}

\def\mod{\mathop{\rm mod}\nolimits}

\def\pd{\mathop{\rm pd}\nolimits}
\def\Hom{\mathop{\rm Hom}\nolimits}
\def\End{\mathop{\rm End}\nolimits}
\def\Ext{\mathop{\rm Ext}\nolimits}
\def\Tr{\mathop{\rm Tr}\nolimits}
\def\add{\mathop{\rm add}\nolimits}

\def\Ker{\mathop{\rm Ker}\nolimits}
\def\Im{\mathop{\rm Im}\nolimits}
\def\gl{\mathop{\rm gl.dim}\nolimits}

\def\CC{\mathop{\mathcal{C}}\nolimits}

\def\FF{\mathop{\mathcal{F}}\nolimits}
\def\XX{\mathop{\mathcal{X}}\nolimits}

\CompileMatrices
\MakeOutlines

\begin{document}
\title{\textsc{The Classification of Special Cohen-Macaulay Modules}}
\author{Osamu Iyama}
\address{Osamu Iyama\\ Graduate School of Mathematics\\ Nagoya University\\ Chikusa-ku, Nagoya, 464-8602, Japan}
\email{iyama@math.nagoya-u.ac.jp}
\author{Michael Wemyss}
\address{Michael Wemyss\\ Graduate School of Mathematics\\ Nagoya University\\ Chikusa-ku, Nagoya, 464-8602, Japan}
\email{wemyss.m@googlemail.com}
\thanks{The second author was supported by the Cecil King Travel Scholarship, and would like to thank both the London Mathematical Society and the Cecil King Foundation.}
\begin{abstract}
In this paper we completely classify all the special Cohen-Macaulay (=CM) modules corresponding to the exceptional curves in the dual graph of the minimal resolutions of all two dimensional quotient singularities.  In every case we exhibit the specials explicitly in a combinatorial way.   Our result relies on realizing the specials as those CM modules whose first Ext group vanishes against the ring $R$, thus reducing the problem to combinatorics on the AR quiver;  such possible AR quivers were classified by Auslander and Reiten.  We also give some general homological properties of the special CM modules and their corresponding reconstruction algebras.
\end{abstract}
\maketitle
\parindent 20pt
\parskip 0pt

\tableofcontents

\section{Introduction}
For a finite subgroup $G\leq SL(2,\C{})$ the McKay Correspondence
\cite{McKay_original} gives a 1-1 correspondence between the
non-trivial representations of $G$ and the exceptional curves on the
minimal resolution of $\C{2}/G$, thus linking the geometry of the
variety $\C{2}/ G$ with the representation theory of $G$.  However
when $G\leq GL(2,\C{})$ it is no longer true that the geometry of
$\C{2}/G$ and the representation theory of $G$ are linked in such a
simple manner since there are now more representations than
exceptional curves.  Put more coarsely the representation theory is
too `big' for the geometry, and to regain a 1-1 correspondence we
need to throw away some representations.

This problem led Wunram \cite{Wunram_generalpaper} to develop the
idea of a special representation so that after passing to the
non-trivial special representations the 1-1 correspondence with the
exceptional curves is recovered.  However the definition of a
special representation is \emph{homological} since it is defined by
the vanishing of cohomology of the dual of a certain vector bundle on the minimal
resolution.  To be able to explicitly say what the non-trivial
special representations are for any non-cyclic subgroup of
$GL(2,\C{})$ has been a hard open question; without knowing what the
special representations are it is certainly difficult (though not impossible) to describe
their structure.

The representation theory of CM modules was initiated by
Auslander and Reiten. They developed a powerful theory based on
homological methods which reveals the hidden structure of the category
of CM modules in terms of Auslander-Reiten(=AR) duality and almost
split sequences, enabling us to visualize the category by the combinatorial
structure of AR quivers.  Auslander classified the indecomposable CM modules over quotient singularities in terms of the irreducible representations of the corresponding group and furthermore showed that when the group is small the AR and McKay quivers coincide \cite{Auslander_rational}.

On the other hand geometric methods in the representation theory of CM modules, initiated by
Artin and Verdier \cite{ArtVer}, often provides us with certain important classes of
CM modules directly from the minimal resolutions of
singularities, and the geometric structure of exceptional curves
on the minimal resolutions is transferred into the categorical structure
of certain CM modules.
For Gorenstein quotient surface singularities the geometric methods
fit quite nicely with the homological methods since they provide
us with all CM modules.
However for non-Gorenstein quotient surface singularities the geometric methods
provide us with only special CM modules and their meaning was much less understood from a homological viewpoint.
In this paper, we shall give several homological characterization of
special CM modules, then give a complete classification of them.

The problem is how to deduce the vanishing of the higher cohomology
of the dual of a certain vector bundle on a space we don't really
understand, and in this paper we solve this via two simple counting
arguments on a noncommutative ring.  The first counting argument uses a new
characterization of the specials in terms of the syzygy functor.  By counting on the AR quiver we can easily compute syzygies, and so this forms one method to deduce if a module is special or not.  Alternatively, the second counting argument relies on the new homological
characterizations of the specials as those CM modules whose first
Ext group vanishes against the ring of invariants.  By AR duality this means we have reduced the problem to counting homomorphisms in the stable category of CM modules, which again is  easy to compute.

In fact our new characterizations of the specials work in greater generality, namely for all rational normal surfaces.  What is somewhat remarkable is that although these may have infinitely many isomorphism classes of indecomposable CM modules, there are only ever finitely many indecomposable objects arising as first syzygies of CM modules, and so they are `syzygy finite'.

In \cite{Wemyss_reconstruct_A} (and subsequent work
\cite{Wemyss_reconstruct_D(i)},\cite{Wemyss_reconstruct_D(ii)}) the
main object of study is the endomorphism ring of the special CM
modules, the so called \emph{reconstruction algebra}.  It was
discovered that the reconstruction algebra is intimately related to
the geometry and gives a correspondence with the dual graph of the
minimal resolution complete with self-intersection numbers via its
underlying quiver.  In this paper we show, via a modified argument of Auslander \cite{A}, that for any rational normal surface $X$ the
global dimension of the corresponding reconstruction algebra is always 2 or 3. Furthermore the value is 2 precisely when $X$ is Gorenstein, i.e. a rational double point.  
This proof not only generalises  \cite{Wemyss_reconstruct_A} but is
also philosophically better since the definition of special CM
module is homological so we should not have to pass down to
generators and relations to prove homological properties.

Since the geometry is unaffected by factoring out by
pseudoreflections, in this paper we can (and will) assume our groups
to be small, thus we can make use of the classification of such
groups by Brieskorn \cite{Brieskorn}.

We now describe the structure of this paper in more detail - in
Section 2 we give the new homological characterisations of the
specials and use them to prove that the global dimension of the
corresponding reconstruction algebras is either two or three. 
In Section 3 we improve some of the results in Section 2 by using a geometrical argument and in
Section 4 we describe the two main counting arguments.  In the remainder of the paper we classify the specials for all small finite subgroups of $GL(2,\C{})$.

We remark that the special CM modules are also known (via different methods) for type $\mathbb{A}$ by Wunram \cite{Wunram_cyclicBook} and Type $\mathbb{D}$ by the
PhD thesis of Nolla de Celis \cite{Alvaro}.

\medskip\noindent{\bf Acknowledgment }The authors would like to thank Tokuji Araya, Ryo Takahashi and Alvaro Nolla de Celis for stimulating discussions.   They also thank the anonymous referee for many valuable comments.

\medskip\noindent{\bf Conventions }
All modules are usually right modules, and the
composition $fg$ of morphisms means first $g$, then $f$.
We denote by $\mod(R)$ the category of finitely generated
$R$-modules, by $J_R$ the Jacobson radical of $R$.
For $M\in\mod(R)$, we denote by $\add M$ the subcategory of
$\mod(R)$ consisting of direct summands of finite direct sums of
copies of $M$. For example $\add R$ is the category of finitely
generated projective $R$-modules.
For an additive category $\CC$, we denote by $J_{\CC}$ the Jacobson radical
of $\CC$. For a full subcategory $\CC'$ of $\CC$, we denote by $[\CC']$
the ideal of $\CC$ consisting of morphisms which factor through objects
in $\CC'$.

\section{Homological properties of special Cohen-Macaulay modules}

Let $R$ be a commutative noetherian ring.
We have a duality $(-)^*:=\Hom_R(-,R):\add R\to\add R$.
For any $X\in\mod(R)$, we take a projective resolution
\[P_1\xrightarrow{g}P_0\xrightarrow{f}X\to0.\]
Define $\Tr X\in\mod(R)$ by an exact sequence
\begin{equation}\label{PRES2}
0\to X^*\xrightarrow{f^*}P_0^*\xrightarrow{g^*}P_1^*\to\Tr X\to0.
\end{equation}
We denote by $\underline{\mod}(R):=(\mod(R))/[\add R]$ the stable category
of $R$ \cite{AB}. Then we have a duality
\[\Tr:\underline{\mod}(R)\xrightarrow{\sim}\underline{\mod}(R)\]
called the \emph{Auslander-Bridger transpose} \cite{AB}\cite{Y}.
We also have the syzygy functor
\[\Omega:\underline{\mod}(R)\rightarrow\underline{\mod}(R).\]

\begin{defin}
Let $n\ge1$.
We put
\[\XX_n:=\{X\in\mod(R)\ |\ \Ext^i_R(X,R)=0\ (0<i\le n)\}.\]
We call $X\in\mod(R)$ \emph{$n$-torsionfree} \cite{AB} if $\Tr X\in\XX_n$.
We denote by $\FF_n$ the category of $n$-torsionfree $R$-modules.
\end{defin}

It is easily shown that $X\in\mod(R)$ is $n$-torsionfree if and only if
there exists an exact sequence $0\to X\to P_0\to\cdots\to P_{n-1}$ such that
$P_{n-1}^*\to\cdots\to P_0^*\to X^*\to0$ is exact \cite{AB}.
Thus any $n$-torsionfree module is an $n$-th syzygy of an $R$-module.

The following result is well-known \cite{AB}.

\begin{lemma}\label{AB sequence}
For any $X,Y\in\mod(R)$, we have an exact sequence
\[0\to\Ext^1_R(\Tr X,Y)\to X\otimes_RY\xrightarrow{\alpha_{X,Y}}
\Hom_R(X^*,Y)\to\Ext^2_R(\Tr X,Y)\to0,\]
where $\alpha_{X,Y}$ is defined by $\alpha_{X,Y}(x\otimes y)(f)=yf(x)$
for $x\in X$, $y\in Y$ and $f\in X^*$.
\end{lemma}

We note that $\alpha_{X,R}$ is the natural map $X\to X^{**}$ and so putting $Y=R$ in Lemma \ref{AB sequence} we have an exact sequence
\[0\to\Ext^1_R(\Tr X,R)\to X\to X^{**}\to\Ext^2_R(\Tr X,R)\to0.\]
Thus $X$ is $1$-torsionfree if and only if it is torsionless,
and $X$ is $2$-torsionfree if and only if it is reflexive.
For a full subcategory $\CC$ of $\mod(R)$, we denote by $\underline{\CC}$
the corresponding full subcategory of $\underline{\mod}(R)$.
Clearly we have the following result.

\begin{lemma}\label{commutative}
We have the following commutative diagram whose
rows are equivalences and columns are dualities:
\[\begin{array}{ccccccccc}
\underline{\XX_n}&\stackrel{\Omega}{\longrightarrow}&\underline{\XX_{n-1}\cap\FF_1}&\stackrel{\Omega}{\longrightarrow}&\cdots
&\stackrel{\Omega}{\longrightarrow}&\underline{\XX_1\cap\FF_{n-1}}&\stackrel{\Omega}{\longrightarrow}&\underline{\FF_n}\\
\downarrow^{\Tr}&&\downarrow^{\Tr}&&&&\downarrow^{\Tr}&&\downarrow^{\Tr}\\
\underline{\FF_n}&\stackrel{\Omega}{\longleftarrow}&\underline{\XX_1\cap\FF_{n-1}}&\stackrel{\Omega}{\longleftarrow}&\cdots&\stackrel{\Omega}{\longleftarrow}&\underline{\XX_{n-1}\cap\FF_1}&\stackrel{\Omega}{\longleftarrow}&\underline{\XX_n}
\end{array}\]
\end{lemma}

Let $R$ be a complete local ring of dimension
$d$.  For $X\in\mod(R)$, we put
\[\depth X:=\min\{i\ge0\ |\ \Ext^i_R(R/J_R,X)\neq0\}.\]
We call $X$ \emph{maximal Cohen-Macaulay} (=\emph{CM}) if $\depth X=d$.
We denote by $\CM(R)$ the category of CM $R$-modules.
We call $R$ a \emph{CM ring} if $R\in\CM(R)$.
Clearly the category $\CM(R)$ is closed under extensions.
In the rest of this section we assume that $R$ is a CM ring with canonical module $\omega$. We often use the equality
\begin{equation}\label{depth}
\depth X=d-\sup\{i\ge0\ |\ \Ext^i_R(X,\omega)\neq0\}.
\end{equation}

We denote by $\OCM(R)$ the subcategory of $\mod(R)$ consisting of $X\in\mod(R)$
such that there exists an exact sequence $0\to X\to P\to Y\to0$ with $Y\in\CM(R)$
and $P\in\add R$. We have the following relationship between CM
modules and $n$-torsionfree modules.

\begin{prop}\label{CM=FD}
Let $R$ be a CM isolated singularity of dimension $d$. Then we have
$\CM(R)=\FF_d$ and $\OCM(R)=\FF_{d+1}$.
\end{prop}

\begin{proof}
It is a well-known result due to Auslander \cite{A2,EG} that $\CM(R)=\FF_d$ holds.

We shall show $\OCM(R)=\FF_{d+1}$.
Since any $(d+1)$-tosionfree module is a syzygy of a $d$-torsionfree module,
we have $\FF_{d+1}\subset\OCM(R)$.
On the other hand, for any $X\in\OCM(R)$, take an exact sequence
$0\to X\to P\to Y\to0$ with $Y\in\CM(R)$ and $P\in\add R$.
Take a morphism $f:X\to Q$ with $Q\in\add R$ such that $f^*:Q^*\to X^*$ is surjective.
We have a commutative diagram
\[\begin{array}{ccccccc}
0\longrightarrow&X&\longrightarrow&P&\longrightarrow&Y&\longrightarrow0\\
&\parallel&&\uparrow&&\uparrow\\
0\longrightarrow&X&\xrightarrow{f}&Q&\longrightarrow&Z&\longrightarrow0
\end{array}\]
of exact sequences.
Taking a mapping cone, we have an exact sequence $0\to Q\to P\oplus Z\to Y\to0$.
Since $\CM(R)$ is closed under extensions, we have $Z\in\CM(R)=\FF_d$.
Thus we have $X\in\FF_{d+1}$.
\end{proof}

The following well-known property \cite{A2} is useful.

\begin{lemma}\label{finite length}
Let $X\in\CM(R)$ and $Y\in\mod(R)$.
If $R$ is an isolated singularity, then $\Ext^i_R(\Tr X,Y)$ is a finite length $R$-module for any $i>0$.
\end{lemma}

\begin{proof}
We give a proof for the convenience of the reader.
For any non-maximal prime ideal $\mathfrak{p}$ of $R$, we have that $X_\mathfrak{p}$ is a projective $R_\mathfrak{p}$-module. Thus we have
\[\Ext^i_R(\Tr X,Y)_\mathfrak{p}=\Ext^i_{R_\mathfrak{p}}(\Tr X_\mathfrak{p},Y_\mathfrak{p})=0\]
for any $i>0$.
Thus $\Ext^i_R(\Tr X,Y)$ is a finite length $R$-module for any $i>0$.
\end{proof}

In the rest of this section we assume that $R$ is a complete local normal domain of dimension two.
Then CM $R$-modules are exactly the reflexive $R$-modules by Proposition \ref{CM=FD}.
Thus we have two dualities
\[(-)^*=\Hom_R(-,R):\CM(R)\to\CM(R)\ \ \mbox{ and }
\ \ \Hom_R(-,\omega):\CM(R)\to\CM(R).\]
For $X\in\mod(R)$, we denote by ${\bf T}(X)$ the torsion submodule of $X$,
which is equal to the kernel of the natural map $X\to X^{**}$.
In the rest of this section we study the following class of CM $R$-modules.
\begin{defin}
Following Wunram \cite{Wunram_generalpaper}, we call $X\in\CM(R)$ \emph{special} if
\[
(X\otimes_R\omega)/{\bf T}(X\otimes_R\omega)\in\CM(R).
\]
We denote by $\SCM(R)$ the category of special CM $R$-modules.
\end{defin}

Let us start with giving several homological characterizations of special CM modules.

\begin{thm}\label{characterization of SCM}
For $X\in\CM(R)$, the following conditions are equivalent.
\begin{itemize}
\item[(a)] $X\in\SCM(R)$.
\item[(b)] $\Ext^2_R(\Tr X,\omega)=0$.
\item[(c)] $\Omega\Tr X\in\CM(R)$.
\item[(d)] $\Ext^1_R(X,R)=0$.
\item[(e)] $X^*\in\OCM(R)$.
\end{itemize}
\end{thm}

\begin{proof}
(a)$\Leftrightarrow$(b) 
By Lemma \ref{AB sequence}, we have an exact sequence
\[0\to\Ext^1_R(\Tr X,\omega)\to X\otimes_R\omega
\xrightarrow{\alpha_{X,\omega}}\Hom_R(X^*,\omega)\to\Ext^2_R(\Tr X,\omega)\to0.\]
By Lemma \ref{finite length} we have that $\Ext^i_R(\Tr X,\omega)$ is a finite length $R$-module for $i=1,2$.
Since $\Hom_R(X^*,\omega)\in\CM(R)$, we have $(X\otimes_R\omega)/{\bf T}(X\otimes_R\omega)=
\Im\alpha_{X,\omega}$.
Thus $X$ is special if and only if $\Im\alpha_{X,\omega}\in\CM(R)$
if and only if $\Ext^2_R(\Tr X,\omega)=0$.

(b)$\Leftrightarrow$(c) Clearly we have $\depth(\Omega\Tr X)\ge1$.
By \eqref{depth}, we have that (c) is equivalent to
$\Ext^1_R(\Omega\Tr X,\omega)=0$, which is clearly equivalent to (b).


(c)$\Rightarrow$(d) By \eqref{PRES2}, we have an exact sequence
\begin{equation}\label{OmegaTr}
0\to X^*\xrightarrow{f^*}P_0^*\to\Omega\Tr X\to0
\end{equation}
from which we obtain an exact sequence
\begin{equation}\label{OmegaTr2}
0\to(\Omega\Tr X)^*\to P_0\xrightarrow{f}X\to0
\end{equation}
by applying $(-)^*$ to \eqref{OmegaTr}.
Applying $(-)^*$ to \eqref{OmegaTr2}, we have \eqref{OmegaTr} since each term is reflexive by (c).
This implies $\Ext^1_R(X,R)=0$.

(d)$\Rightarrow$(e) Take a projective resolution $P_2\to P_1\to P_0\to X\to0$.
Applying $(-)^*$, we have an exact sequence
$0\to X^*\to P_0^*\to P_1^*\to P_2^*$. Thus $X^*\in\OCM(R)$.

(e)$\Rightarrow$(c) Take an exact sequence $0\to X^*\to P\to Y\to0$ with $Y\in\CM(R)$ and $P\in\add R$.
We use the exact sequence \eqref{OmegaTr}.
Since $P_0=\Hom_R(P_0^*,R)\to X=\Hom_R(X^*,R)$ is surjective,
we have a commutative diagram
\[\begin{array}{ccccccccc}
0&\to&X^*&\to&P&\to&Y&\to&0\\
&&\parallel&&\uparrow&&\uparrow&&\\
0&\to&X^*&\to&P_0^*&\to&\Omega\Tr X&\to&0
\end{array}\]
of exact sequences. Taking a mapping cone, we have an exact sequence
$0\to P_0^*\to P\oplus\Omega\Tr X\to Y\to0$.
This implies $\Omega\Tr X\in\CM(R)$.
\end{proof}

We have the following description of categories in terms of $n$-torsionfreeness.

\begin{cor}\label{description}
$\CM(R)=\FF_2$,\ \ \ $\OCM(R)=\FF_3$\ \ and\ \ $\SCM(R)=\XX_1\cap\FF_2$.
\end{cor}

\begin{proof}
Immediate from Proposition \ref{CM=FD} and
Theorem \ref{characterization of SCM}(a)$\Leftrightarrow$(d).
\end{proof}

We have the following equivalences.

\begin{cor}\label{equivalence}
\begin{itemize}
\item[(a)] We have a duality $(-)^*:\SCM(R)\stackrel{\sim}{\longleftrightarrow}\OCM(R)$.
\item[(b)] We have an equivalence $\Omega:\underline{\SCM}(R)\stackrel{\sim}{\longrightarrow}\underline{\OCM}(R)$.
\item[(c)] We have dualities
\[(\Omega-)^*:\underline{\SCM}(R)\stackrel{\sim}{\longrightarrow}\underline{\SCM}(R)\ \mbox{ and }\ 
\Omega(-)^*:\underline{\OCM}(R)\stackrel{\sim}{\longrightarrow}\underline{\OCM}(R)\]
such that $((\Omega-)^*)^2\simeq 1_{\underline{\SCM}(R)}$ and $(\Omega(-)^*)^2\simeq 1_{\underline{\OCM}(R)}$.
\end{itemize}
\end{cor}

\begin{proof}
(a) Immediate from Theorem \ref{characterization of SCM}(a)$\Leftrightarrow$(e).

(b) We have an equivalence $\Omega:\underline{\XX_1\cap\FF_2}\stackrel{\sim}{\longrightarrow}\underline{\FF_3}$ by Lemma \ref{commutative}.
Thus the assertion follows from Corollary \ref{description}.

(c) By (a) and (b), we have the desired dualities.
For any $X\in\SCM(R)$, take a projective resolution $0\to\Omega X\to P\to X\to0$.
Applying $(-)^*$, we have an exact sequence $0\to X^*\to P^*\to(\Omega X)^*\to0$.
Thus $\Omega(\Omega X)^*\simeq X^*$ holds, and we have
$((\Omega-)^*)^2\simeq 1_{\underline{\SCM}(R)}$.
Similarly, one can show $(\Omega(-)^*)^2\simeq 1_{\underline{\OCM}(R)}$.
\end{proof}

Later in Section 3 we will improve Theorem~\ref{characterization of SCM}
and Corollary \ref{equivalence} for rational singularities by using a geometric argument.
In the rest of this section we assume that $R$ is \emph{syzygy finite} in the sense that
there are only finitely many isoclasses of indecomposable objects in
$\OCM(R)$ (and hence $\SCM(R)$). We study homological properties of
the endomorphism algebras of additive generators in $\SCM(R)$,
which are called the \emph{reconstruction algebras}.
We have the following result.

\begin{thm}\label{End of OCM}
Assume $\OCM(R)=\add M$ and put $\Lambda:=\End_R(M)$.
\begin{itemize}
\item[(a)] If $R$ is Gorenstein, then $\gl\Lambda=2$. All simple $\Lambda$ and $\Lambda^{\rm op}$-modules have projective dimension 2.
\item[(b)] If $R$ is not Gorenstein, then $\gl\Lambda=3$.
All simple $\Lambda$-modules have projective dimension 2
except $\Hom_R(M,R)/J_{\CM(R)}(M,R)$, which has projective dimension 3.
\end{itemize}
\end{thm}

We need the observation below.
This kind of result was used in the study of Auslander's representation dimension \cite{A,EHIS}.

\begin{prop}\label{End 3 criterion}
Let $M\in\CM(R)$ be a generator and $\Lambda:=\End_R(M)$.
For $n\ge0$, the following conditions are equivalent.
\begin{itemize}
\item[(a)] ${\rm gl.dim} \Lambda\le n+2$.
\item[(b)] For any $X\in\CM(R)$, there exists an exact sequence
\[0\to M_n\to\cdots\to M_0\stackrel{}{\to}X\to0\]
with $M_i\in\add M$
such that the following sequence is exact.
\[0\to\Hom_R(M,M_n)\to\cdots\to\Hom_R(M,M_0)\to\Hom_R(M,X)\to0.\]
\end{itemize}
\end{prop}

\begin{proof}
We have an equivalence
$\Hom_R(M,-):\add M_R\to\add\Lambda_\Lambda$ of categories.

(b)$\Rightarrow$(a) For any $Y\in\mod(\Lambda)$, take a projective resolution
$P_1\stackrel{f}{\to} P_0\to Y\to0$. Take a morphism
$M_1\stackrel{g}{\to} M_0$ in $\add M_R$ such that $f=\Hom_R(M,g)$.
Put $X:=\Ker g$. Then $X\in\CM(R)$. By (b), there exists an exact sequence
$0\to M_{n+2}\to\cdots\to M_2\stackrel{}{\to}X\to0$ with $M_i\in\add M_R$
such that
\[0\to\Hom_R(M,M_{n+2})\to\cdots\to\Hom_R(M,M_2)\to\Hom_R(M,X)\to0\]
is exact. Then we have a projective resolution
\[0\to\Hom_R(M,M_{n+2})\to\cdots\to\Hom_R(M,M_2)\to\Hom_R(M,M_1)\to\Hom_R(M,M_0)\to Y\to0.\]
Thus we have $\pd Y\le n+2$.

(a)$\Rightarrow$(b) For any $X\in\CM(R)$, there exists an exact sequence
$0\to X\to F_1\to F_0$ with $F_i\in\add R_R$.
Put $Y:=\Hom_R(M,X)$. Since we have an exact sequence $0\to Y\to\Hom_R(M,F_1)\to\Hom_R(M,F_0)$ with $\Hom_R(M,F_i)\in\add\Lambda_\Lambda$, we have
$\pd Y\le n$.
Take a projective resolution
\begin{equation}\label{resolution of Y}
0\to P_n\to\cdots\to P_0\to Y\to0.
\end{equation}
Then there exists a complex
\begin{equation}\label{approximation of X}
0\to M_n\to\cdots\to M_0\stackrel{}{\to}X\to0
\end{equation}
with $M_i\in\add M_R$ such that the image of \eqref{approximation of X}
under the functor $\Hom_R(M,-)$ is \eqref{resolution of Y}.
Since $M$ is a generator, \eqref{approximation of X} is exact.
This is the desired sequence.
\end{proof}

We need the following easy observation.

\begin{lemma}\label{at least two}
For any non-zero $M\in\CM(R)$, we put $\Lambda:=\End_R(M)$.
Then any simple $\Lambda$-module $S$ has projective dimension at least $2$.
\end{lemma}

\begin{proof}
Assume that there exists a projective resolution $0\to P_1\to P_0\to S\to0$ of the $\Lambda$-module $S$.  Since projective $\Lambda$-modules are CM $R$-modules (since $\t{dim}R=2$) we have that $\depth P_i\ge2$ for $i=0,1$ and so $\depth S\ge1$, a contradiction.
\end{proof}

Immediately we have the following result by putting $n=0$ in Proposition \ref{End 3 criterion}.

\begin{prop}\label{End 2 criterion}
Let $M\in\CM(R)$ be a generator.
Then $\CM(R)=\add M$ holds if and only if $\gl\End_R(M)=2$ holds.
\end{prop}


We also need the following easy observation.

\begin{lemma}\label{syzygy exact}
If $0\to Z\to Y\to X\to0$ is an exact sequence with $X\in\CM(R)$ and $Y\in\OCM(R)$, then we have $Z\in\OCM(R)$.
\end{lemma}

\begin{proof}
Since $Y\in\OCM(R)$, there exists an exact sequence
$0\to Y\to P\to W\to 0$ with $W\in\CM(R)$ and $P\in\add R$.
Then we have a commutative diagram
\[\begin{array}{ccccccccc}
&&&&0&&0&&\\
&&&&\downarrow&&\downarrow&&\\
0&\to&Z&\to&Y&\to&X&\to&0\\
&&\parallel&&\downarrow&&\downarrow&&\\
0&\to&Z&\to&P&\to&V&\to&0\\
&&&&\downarrow&&\downarrow&&\\
&&&&W&=&W&&\\
&&&&\downarrow&&\downarrow&&\\
&&&&0&&0&&
\end{array}\]
of exact sequences. We have $V\in\CM(R)$ by the right vertical sequence,
and the middle horizontal sequence shows that $Z$ is a syzygy of $V\in\CM(R)$.
\end{proof}

Now we prove Theorem \ref{End of OCM}.

(i) First we show $\gl\Lambda\le3$.
We only have to show that $M$ in Theorem \ref{End of OCM} satisfies
the condition Proposition \ref{End 3 criterion}(b) for $n=1$.
For any $X\in\CM(R)$, take an exact sequence
\[0\to Y\to M_0\stackrel{f}{\to}X\]
with $M_0\in\OCM(R)$ such that
$\Hom_R(M,M_0)\xrightarrow{\Hom(M,f)}\Hom_R(M,X)$ is surjective.
Since $M$ is a generator, $f$ is surjective.
By Lemma \ref{syzygy exact}, we have $Y\in\OCM(R)$.
Thus $\gl\Lambda\le3$ holds.

(ii) We decide the precise value of $\gl\Lambda$.
If $R$ is Gorenstein, then $\OCM(R)=\CM(R)$.
Thus we have $\gl\Lambda=2$ by Proposition \ref{End 2 criterion}.

If $R$ is not Gorenstein, then $\omega\notin\OCM(R)$.
Thus $\OCM(R)$ is strictly smaller than $\CM(R)$.
We have $\gl\Lambda=3$ by Proposition \ref{End 2 criterion}.

(iii) Let $S=\Hom_R(M,X)/J_{\CM(R)}(M,X)$ be a simple $\Lambda$-module with indecomposable non-free $X\in\OCM(R)$.
Take an exact sequence
\[0\to Y\to M_0\stackrel{f}{\to}X\]
with $M_0\in\OCM(R)$ such that $\Hom_R(M,M_0)\xrightarrow{\Hom(M,f)}J_{\CM(R)}(M,X)$ is surjective.
Take a surjection $P\stackrel{g}{\to}X\to0$ with $P\in\add R$.
Since $X$ is non-free, we have $g\in J_{\CM(R)}$ and that $g$ factors through $f$.
Hence $f$ is surjective, so by Lemma \ref{syzygy exact} we have $Y\in\OCM(R)$.
Thus we have a projective resolution
\[0\to\Hom_R(M,Y)\to\Hom_R(M,M_0)\to\Hom_R(M,X)\to S\to0.\]
We have $\pd S=2$ by Lemma \ref{at least two}.
\qed

\medskip
Immediately we have the following result.

\begin{cor}\label{End of SCM}
Assume $\SCM(R)=\add N$ and put $\Lambda:=\End_R(N)$.
\begin{itemize}
\item[(a)] If $R$ is Gorenstein, then $\gl\Lambda=2$. All simple $\Lambda$ and $\Lambda^{\rm op}$-modules have projective dimension 2.
\item[(b)] If $R$ is not Gorenstein, then $\gl\Lambda=3$.
All simple $\Lambda^{\rm op}$-modules have projective dimension 2
except $\Hom_R(R,N)/J_{\CM(R)}(R,N)$, which has projective dimension 3.
\end{itemize}
\end{cor}

\begin{proof}
We have $\add N^*=\OCM(R)$ by Corollary \ref{equivalence}.
Since $\End_R(N)=\End_R(N^*)^{\rm op}$, the assertion follows from Theorem \ref{End of OCM}.
\end{proof}

\section{Geometric aspects of special Cohen-Macaulay modules}
Let $R$ be a rational normal surface singularity.
In this section we use the geometry of the minimal resolution of $\t{Spec}R$ to improve some of the algebraic results in Section 2;
in particular we obtain the rather surprising result that all rational normal surfaces are syzygy finite.

To do this we use results of Wunram \cite{Wunram_generalpaper}, and so we first need to introduce some notation.
For a rational normal surface $X=\t{Spec}R$ denote the minimal resolution by $\pi:\w{X}\rightarrow \t{Spec}R$ and the exceptional curves by $\{ E_i\}$.
Also,  for a given CM module $M$ of $R$, denote by $\w{M}:=\pi^*M/{\rm torsion}$ the corresponding sheaf on $\w{X}$. 

\begin{defin}
Given the exceptional curves  $\{ E_i\}$ we define the labelled dual graph of the minimal resolution as follows: for every exceptional curve $E_i$ draw a dot, and join two dots if the corresponding curves intersect.  Additionally, decorate each vertex with the self-intersection number corresponding to the curve at that vertex. 
\end{defin}
\begin{defin} \cite{Art}
For a given labelled dual graph, define the fundamental cycle $Z_f=\sum_{\t{vertices } i}r_i E_i$ (with each $r_i\geq 1$) to be the unique smallest element such that $Z_f\cdot E_i\leq 0$ for all vertices $i$.
\end{defin}
There is an easy algorithm to find $Z_f$ given by Laufer \cite{Lau72}, which we illustrate in two examples below.  
\begin{example}
\t{Firstly, consider the dual graph
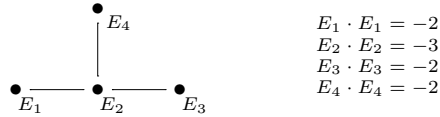
\[
\begin{array}{cc}
\begin{array}{c}
\xymatrix@C=20pt@R=20pt{ &\bullet\ar@{-}[d]^<{E_4}&&\\
\bullet\ar@{-}[r]_<{E_1} & \bullet\ar@{-}[r]_<{E_2}
& \bullet\ar@{}[r]_<{E_3}&}
\end{array}&
{\scriptsize{\begin{array}{c}
E_1\cdot E_1=-2\\
E_2 \cdot E_2=-3\\
E_3 \cdot E_3=-2\\
E_4\cdot E_4=-2
\end{array}}}
\end{array}
\]
We shall denote this by $\begin{array}{c}\tiny{\xymatrix@C=-2pt@R=-2pt{&2&\\2&3&2}}\end{array}$.  To calculate $Z_f$, first try the smallest element $Z_r=E_1+E_2+E_3+E_4$:{\scriptsize{
\begin{eqnarray*}
Z_r\cdot E_1&=&E_1\cdot E_1+E_2\cdot E_1+E_3\cdot E_1+E_4\cdot E_1=(-2)+1+0+0=-1\leq 0\\
Z_r\cdot E_2&=&E_1\cdot E_2+E_2\cdot E_2+E_3\cdot E_2+E_4\cdot E_2=1+(-3)+1+1=0\leq 0\\
Z_r\cdot E_3&=&E_1\cdot E_3+E_2\cdot E_3+E_3\cdot E_3+E_4\cdot E_3=0+1+(-2)+0=-1\leq 0\\
Z_r\cdot E_4&=&E_1\cdot E_4+E_2\cdot E_4+E_3\cdot E_4+E_4\cdot E_4=0+1+0+(-2)=-1\leq 0
\end{eqnarray*}}}
Since $Z_r\cdot E_i\leq 0$ for all exceptional curves $E_i$ we conclude that $Z_f=Z_r=E_1+E_2+E_3+E_4$.  In this paper we shall denote this by $Z_f=\begin{array}{c}\tiny{\xymatrix@C=-2pt@R=-2pt{&1&\\1&1&1}}\end{array}$. }
\end{example}
\begin{example}
\t{Now if we change the above example slightly and consider the dual graph $\begin{array}{c}\tiny{\xymatrix@C=-2pt@R=-2pt{&2&\\2&2&2}}\end{array}$ then the above fails since now $Z_r\cdot E_2=1\nleq 0$.  But $Z^\prime=E_1+2E_2+E_3+E_4$ satisfies $Z^\prime\cdot E_i\leq 0$ for all exceptional $E_i$ and so we deduce that  $Z_f=\begin{array}{c}\tiny{\xymatrix@C=-2pt@R=-2pt{&1&\\1&2&1}}\end{array}$. }
\end{example}
The main result obtained by Wunram was the following, which recovers
the results of Artin-Verdier \cite{ArtVer} as a special case
\begin{thm}\cite[1.2]{Wunram_generalpaper}\label{Wurnam_main_result}
\begin{itemize}
\item[(a)] For every irreducible curve $E_i$ ($1\leq i\leq k$) in the
exceptional divisor of the minimal resolution there is exactly one
indecomposable CM module $M_i$ (up to isomorphism) with
\[
H^1(\w{M_i}^\vee)=0
\]
and
\[
c_1(\w{M_i})\cdot E_j=\delta_{ij} \mbox{ for all } 1\leq i,j\leq k .
\]
The rank of $M_i$ equals $r_i=c_1(\w{M_i})\cdot Z_f$ where $Z_f=\sum
r_i E_i$ is the fundamental cycle.\\
\item[(b)] $M\in\CM(R)$ satisfies $H^1({\w{M}}^\vee)=0$ if and only if $M\in\SCM(R)$.
\end{itemize}
\end{thm}
Thus the fundamental cycle dictates the ranks of the special CM modules.

We now use the above to improve our results in Section 2 as follows: note that (b) below also generalizes \cite[Th. 3]{MS} to the non-Gorenstein case.
\begin{thm}
Let $R$ be a rational normal surface singularity.
\begin{itemize}
\item[(a)] $\SCM(R)$ and $\OCM(R)$ contain only finitely many isoclasses of indecomposable objects.
\item[(b)] $X\in\CM(R)$ belongs to $\SCM(R)$ if and only if $(\Omega X)^*\simeq X$ up to a free summand.
\item[(c)] $X\in\CM(R)$ belongs to $\OCM(R)$ if and only if $\Omega(Y^*)\simeq Y$ up to a free summand.
\end{itemize}
\end{thm}

\begin{proof}
(a) The assertion for $\SCM(R)$ follows from Theorem~\ref{Wurnam_main_result} above, since there are only finitely many exceptional curves.  The assertion for $\OCM(R)$ follows from Corollary \ref{equivalence}.

(b) The `if' part follows by Theorem~\ref{characterization of SCM}.

We shall show the `only if' part.   Suppose $M$ is a special CM module.  Since $M$ is CM, by Artin-Verdier \cite[1.2]{ArtVer} we have the following exact sequence
\[
\xymatrix@C=20pt@R=20pt{0\ar[r]&{\s{O}^{r}}\ar[r]&{\w{M}}\ar[r]&{\s{O}_{D}}\ar[r]&0
}
\]
where $r$ is the rank of $M$.  After dualizing the above we get
\[
\begin{array}{c}
\xymatrix@C=20pt@R=20pt{0\ar[r]&{\w{M}^\vee}\ar[r]&{\s{O}^{r}}\ar[r]&{\s{O}_{D}}\ar[r]&0
}
\end{array}
\]
Taking the appropriate pullback gives us a diagram
\[
\begin{array}{c}
\xymatrix@C=20pt@R=20pt{&&0\ar[d]&0\ar[d]&\\
&&{\w{M}^\vee}\ar[d]\ar@{=}[r]&{\w{M}^\vee}\ar[d]&\\
0\ar[r]&{\s{O}^{r}}\ar@{=}[d]\ar[r]&{\s{E}}\ar[r]\ar[d]&{\s{O}^{r}}\ar[r]\ar[d]&0\\
0\ar[r]&{\s{O}^{r}}\ar[r]&{\w{M}}\ar[r]\ar[d]&{\s{O}_{D}}\ar[r]\ar[d]&0\\
&&0&0&\\
}
\end{array}
\]
Since the singularity is rational the middle horizontal sequence splits giving $\s{E}=\s{O}^{2r}$, and so we have a short exact sequence
\[
\begin{array}{c}
\xymatrix@C=20pt@R=20pt{0\ar[r]&{\w{M}^\vee}\ar[r]&{\s{O}^{2r}}\ar[r]&{\w{M}}\ar[r]&0
}
\end{array}.
\]
But now $M$ is special and so by Theorem~\ref{Wurnam_main_result} above $H^1(\w{M}^\vee)=0$, so taking global sections of this sequence yields
\[
\begin{array}{c}
\xymatrix@C=20pt@R=20pt{0\ar[r]&{{M}^*}\ar[r]&{R^{2r}}\ar[r]&{M}\ar[r]&0
}
\end{array}
\]
as required.

(c) The `if' part is clear and the `only if' part follows from (b) and Corollary \ref{equivalence}(a).
\end{proof}

In this remainder of this paper we consider the surface quotient singularities and classify the special CM modules in all these cases.  We use the Brieskorn \cite{Brieskorn} classification
of finite small subgroups of $GL(2,\C{})$, but with the notation
from Riemenschneider \cite{Riemenschneider_invarianten}.  The
classification can be stated as follows:
\[
\begin{array}{ccc}
\t{Type} & \t{Notation} & \t{Conditions}\\
\mathbb{A} & \mathbb{A}_{r,a}:=\frac{1}{r}(1,a):=\left\langle \left(\begin{smallmatrix} \ve_r & 0\\ 0& \ve_r^a
\end{smallmatrix}\right)\right\rangle & 1<a<r, (r,a)=1 \\
\mathbb{D}& \begin{array}{cc}
\mathbb{D}_{n,q}:=\left\{ \begin{array}{cc} \langle \psi_{2q},
\tau, \varphi_{2(n-q)} \rangle& \mbox{if } n-q\equiv 1 \mbox{ mod
}2\\\langle \psi_{2q}, \tau\varphi_{4(n-q)} \rangle& \mbox{if }
n-q\equiv 0 \mbox{ mod }2
\end{array}\right.
\end{array}& 1<q<n, (n,q)=1\\
\mathbb{T}
&\mathbb{T}_m:=\left\{ \begin{array}{cc} \langle \psi_{4},
\tau, \eta, \varphi_{2m} \rangle& \mbox{if } m\equiv 1,5 \mbox{ mod
}6\\\langle \psi_{4},\tau,\eta\varphi_{6m} \rangle& \mbox{if }
m\equiv 3 \mbox{ mod }6
\end{array}\right. & m\equiv 1,3,5 \mbox{ mod }6\\
\mathbb{O}&
\mathbb{O}_m:=\langle \psi_{8},\tau,\eta,
\varphi_{2m} \rangle
& m\equiv 1,5,7,11 \mbox{ mod }12
\\
\mathbb{I}&
\mathbb{I}_m:=\langle \left(\begin{smallmatrix}0&-1\\1&0\end{smallmatrix}\right),\omega,\iota, \varphi_{2m} \rangle
&  \begin{array}{rl}m\equiv& 1, 7,11,13,17,19,\\ &23,29 \mbox{ mod }30\end{array}
\end{array}
\]
with the matrices{\scriptsize{
\[
\begin{array}{c}
\begin{array}{ccccc}
\psi_k= \begin{pmatrix}\e_k & 0\\ 0& \e_k^{-1}
\end{pmatrix} &\tau = \begin{pmatrix}0 & \e_4\\ \e_4& 0
\end{pmatrix}&\varphi_k= \begin{pmatrix}\e_k & 0\\ 0& \e_k
\end{pmatrix}&\eta=\frac{1}{\sqrt{2}} \begin{pmatrix}\e_8 & e_8^3\\ \e_8& \e_8^7 \end{pmatrix}&\omega= \begin{pmatrix}\e_5^3 & 0\\ 0& \e_5^2
\end{pmatrix}
\end{array}\\
\iota=\frac{1}{\sqrt{5}} \begin{pmatrix}\e_5^4-\e_5 & \e_5^2-\e_5^3\\ \e_5^2-\e_5^3& \e_5-\e_5^4
\end{pmatrix}
\end{array}
\]
}}where $\e_t$ is a primitive $t^{th}$ root of unity.  Note that in
this notation $E_6=\mathbb{T}_1$, $E_7=\mathbb{O}_1$ and
$E_8=\mathbb{I}_1$.

For a given group $G$ in the above classification, by \cite{AR_McKayGraphs} the universal cover of the AR quiver of $\C{}[[x,y]]^G$ is
\[
\begin{array}{cc}
\t{Type} & \t{Universal Cover}\\
\mathbb{A}_{r,a} & \mathbb{Z}A_\infty^\infty\\
\mathbb{D}_{n,q}& \mathbb{Z}\tilde{D}_{q+2}\\
\mathbb{T}& \mathbb{Z}\tilde{E}_{6}\\
\mathbb{O}& \mathbb{Z}\tilde{E}_{7}\\
\mathbb{I}& \mathbb{Z}\tilde{E}_{8}
\end{array}
\]
where we give more precise information in later sections.

Notice that the three families of type $\mathbb{T}$, $\mathbb{O}$
and $\mathbb{I}$ are one-parameter families which naturally split
into subfamilies depending on the conditions in the right hand side
of the table.  Each subfamily depends on one parameter, and in each
subfamily there is precisely one value of that parameter for which
the fundamental cycle $Z_f$ is not reduced; for all other values it
is.   Although we do not use fundamental cycles this observation explains why the proof of each subfamily splits into two - compare for example Lemma~\ref{T5} and Lemma~\ref{T5b}.

\section{Combinatorics on Auslander-Reiten quivers}

Throughout this section, let $k$ be an algebraically closed field.
Let $R$ be a complete local normal domain of dimension two with
$k=R/J_R$, and let $\omega$ be the canonical module of $R$.
Let $\CM(R)$ be the category of maximal CM $R$-modules.
We denote by $\underline{\CM}(R):=(\CM(R))/[\add R]$ and
$\overline{\CM}(R):=(\CM(R))/[\add\omega]$ the stable categories.
We denote by $\Omega:\underline{\CM}(R)\to\underline{\CM}(R)$
and $\Omega^-:\overline{\CM}(R)\to\overline{\CM}(R)$ the syzygy
and the cosyzygy functors respectively.
Composing dualities, we have mutually quasi-inverse equivalences
\begin{eqnarray*}
\tau&:&\CM(R)\xrightarrow{(-)^*}\CM(R)\xrightarrow{\Hom_R(-,\omega)}\CM(R),\\
\tau^-&:&\CM(R)\xrightarrow{\Hom_R(-,\omega)}\CM(R)\xrightarrow{(-)^*}\CM(R)
\end{eqnarray*}
called \emph{AR translations}.
Clearly $\tau$ gives a bijection from the set of isoclasses of indecomposable objects in $\CM(R)$ to itself.
Moreover $\tau R=\omega$ holds.

Let us recall the following classical results \cite{A2,Y},
where we denote by $D=\Ext^2_R(-,\omega)$ the Matlis duality.

\begin{thm}\label{classical}
\begin{itemize}
\item[(a)] There exists a functorial isomorphism (called \emph{AR duality})
\[\underline{\Hom}_R(\tau^-Y,X)\simeq D\Ext^1_R(X,Y)\simeq\overline{\Hom}_R(Y,\tau X)\]
for any $X,Y\in\CM(R)$.
\item[(b)] For any indecomposable non-projective object $X\in\CM(R)$,
there exists an exact sequence (called an \emph{almost split sequence})
\[0\to\tau X\to\theta X\to X\to0\]
such that the following sequences are exact on $\CM(R)$.
\begin{eqnarray*}
&0\to\Hom_R(-,\tau X)\to\Hom_R(-,\theta X)\to J_{\CM(R)}(-,X)\to0,&\\
&0\to\Hom_R(X,-)\to\Hom_R(\theta X,-)\to J_{\CM(R)}(\tau X,-)\to0.&
\end{eqnarray*}
\item[(c)] There exists an exact sequence
(called a \emph{fundamental sequence}) 
\[0\to\tau R\to\theta R\to R\to k\to0\]
such that the following sequences are exact on $\CM(R)$.
\begin{eqnarray*}
&0\to\Hom_R(-,\tau R)\to\Hom_R(-,\theta R)\to J_{\CM(R)}(-,R)\to0,&\\
&0\to\Hom_R(R,-)\to\Hom_R(\theta R,-)\to J_{\CM(R)}(\tau R,-)\to0.&
\end{eqnarray*}
\end{itemize}
\end{thm}

Recall that the \emph{AR quiver} of $\CM(R)$ is defined as follows.
\begin{itemize}
\item Vertices are isoclasses of indecomposable objects in $\CM(R)$.
\item For indecomposable objects $X,Y\in\CM(R)$, draw $d_{XY}$ arrows from $X$ to $Y$ for $d_{XY}:=\dim_k(J_{\CM(R)}/J_{\CM(R)}^2)(X,Y)$.
\item For any indecomposable object $X\in\CM(R)$, draw a dotted arrow from $X$ to $\tau X$.
\end{itemize}
It is easily shown that $d_{XY}$ coincides with the multiplicity
of $X$ in $\theta Y$, and with that of $Y$ in $\theta\tau^-X$.

In the rest of this section, we shall give methods to calculate the following data for $X,Y\in\CM(R)$ by using the AR quiver of $\CM(R)$.
\begin{itemize}
\item[(A)] $\dim_k\Ext^1_R(X,Y)$, or equivalently (by AR duality) $\dim_k\underline{\Hom}_R(\tau^-Y,X)$,
\item[(B)] The position of each summand of $\Omega X$ in the AR quiver.
\end{itemize}
For this, we have to consider more general class of categories including $\CM(R)$, $\underline{\CM}(R)$ and $\overline{\CM}(R)$.

\begin{defin}
We call an additive category $\CC$ a \emph{$\tau$-category} \cite{I1}
if the following conditions are satisfied.
\begin{itemize}
\item[(a)] $\CC$ is \emph{Krull-Schmidt}, i.e. any object in $\CC$ is isomorphic to a finite direct sum of objects whose endomorphism rings are local.
\item[(b)] For any object $X\in\CC$, there exists a complex
\begin{equation}\label{right t}
\tau X\xrightarrow{\nu_X}\theta X\xrightarrow{\mu_X}X
\end{equation}
with right minimal morphisms $\mu_X$ and $\nu_X$ contained in $J_{\CC}$ such that the following sequences are exact.
\begin{eqnarray*}
&\CC(-,\tau X)\xrightarrow{\nu_X}\CC(-,\theta X)\xrightarrow{\mu_X}J_{\CC}(-,X)\to0,&\\
&\CC(\theta X,-)\xrightarrow{\nu_X}J_{\CC}(\tau X,-)\to0.&
\end{eqnarray*}
\item[(c)] For any object $X\in\CC$, there exists a complex
\begin{equation}\label{left t}
X\xrightarrow{\mu^-_X}\theta^-X\xrightarrow{\nu^-_X}\tau^-X
\end{equation}
with left minimal morphisms $\mu^-_X$ and $\nu^-_X$ contained in $J_{\CC}$ such that the following sequences are exact.
\begin{eqnarray*}
&\CC(\tau^-X,-)\xrightarrow{\nu^-_X}\CC(\theta^-X,-)\xrightarrow{\mu^-_X}J_{\CC}(X,-)\to0,&\\
&\CC(-,\theta^-X)\xrightarrow{\nu^-_X}J_{\CC}(-,\tau^-X)\to0.&
\end{eqnarray*}
\end{itemize}
We call the complex \eqref{right t} (respectively, \eqref{left t}) a \emph{right $\tau$-sequence} (respectively, \emph{left $\tau$-sequence}).
\end{defin}

The following fact is shown in \cite{I1}.
\begin{itemize}
\item If $X\in\CC$ is indecomposable, then either $\tau X=0$ (respectively, $\tau^-X=0$) holds or $\tau X$ (respectively, $\tau^-X$) is also indecomposable.
\end{itemize}
We assume that $\CC$ is $k$-linear and $\dim_k(\CC/J_{\CC})(X,Y)<\infty$
for any $X,Y\in\CC$. We define the \emph{AR quiver} of $\CC$
by replacing $\CM(R)$ in the above
definition of the AR quiver of $\CM(R)$ by $\CC$.

By Theorem \ref{classical}, the category $\CM(R)$ is a $\tau$-category.
By the following easy observation \cite[1.4]{I2}, the stable categories
$\underline{\CM}(R)$ and $\overline{\CM}(R)$ are also $\tau$-categories.

\begin{prop}\label{tau1}
Let $\CC$ be a $\tau$-category and $\CC'$ a full subcategory of $\CC$.
Then the factor category $\CC/[\CC']$ is a $\tau$-category, and its
AR quiver is given by removing from the AR quiver of $\CC$
all vertices corresponding to indecomposable objects in $\CC'$
and all dotted arrows from $X$ to $\tau X$ satisfying $\theta X\in\CC'$.
\end{prop}

Let us recall a method to calculate $\dim_k\CC(X,Y)$ for each
$X,Y\in\CC$ following \cite{I1}.
One of the key results is the existence theorem of ladders (a) below
\cite[Th. 3.3, 4.1]{I1},
which was introduced by Igusa-Todorov for some cases \cite{IT1}.
For $X\in\CC$ and indecomposable $Y\in\CC$, we denote by $m_Y(X)$ the
multiplicity of $Y$ in $X$.

\begin{thm}\label{ladder}
Let $\CC$ be a $\tau$-category and $X\in\CC$.
\begin{itemize}
\item[(a)] There exist a commutative diagram (called a \emph{left ladder} of $X$)
\[\begin{array}{cccccccc}
X=Y_0&\xrightarrow{f_0}&Y_1&\xrightarrow{f_1}&Y_2&\xrightarrow{f_2}&Y_3&\xrightarrow{f_3}\cdots\\
\downarrow^{b_0}&&\downarrow^{b_1}&&\downarrow^{b_2}&&\downarrow^{b_3}&\\
0=Z_0&\xrightarrow{g_0}&Z_1&\xrightarrow{g_1}&Z_2&\xrightarrow{g_2}&Z_3&\xrightarrow{g_3}\cdots,
\end{array}\]
and objects $U_{n+1}\in\CC$ and a morphism $h_n\in\CC(Z_n,U_{n+1})$ such that
\begin{equation}\label{cone}
Y_n\xrightarrow{{b_n\choose -f_n}}Z_n\oplus Y_{n+1}\xrightarrow{{g_n\ b_{n+1}\choose h_{n}\ 0}}Z_{n+1}\oplus U_{n+1}
\end{equation}
is a left $\tau$-sequence for any $n\ge0$.
\item[(b)] For any $n\ge0$, we have an isomorphism
$(J_{\CC}^n/J_{\CC}^{n+1})(X,-)\simeq(\CC/J_{\CC})(Y_n,-)$
of functors on $\CC$.
In particular, if $\bigcap_{i\ge0}J_{\CC}^i=0$, then
\[\dim_k\CC(X,Y)=\sum_{n\ge0}m_Y(Y_n)\]
holds for any indecomposable $Y\in\CC$.
\end{itemize}
\end{thm}

We know $\dim_k\CC(X,Y)$ by (b) if we calculate the terms $Y_n$ explicitly.
By \eqref{cone}, we have
\[Z_{n}\oplus Y_{n+1}\simeq\theta^-Y_{n}\ \mbox{ and }\ 
Z_{n+1}\oplus U_{n+1}\simeq\tau^-Y_n.\]
We denote by $K_0(\CC)$ the Grothendieck group of the additive category $\CC$.
Thus $K_0(\CC)$ is the free abelian group generated by isoclasses of
indecomposable objects in $\CC$ by Krull-Schmidt property.
Any $X\in K_0(\CC)$ can be written uniquely as $X=X_+-X_-$ for $X_+,X_-\in\CC$
such that $X_+$ and $X_-$ have no non-zero common direct summand.
We have an equality
\begin{equation}\label{Y and U}
Y_n=\theta^-Y_{n-1}-Z_{n-1}=\theta^-Y_{n-1}-\tau^-Y_{n-2}+U_{n-1}
\end{equation}
in $K_0(\CC)$ for $n\ge2$.
It is shown in \cite{I1} that $Y_n$ and $U_{n-1}$ have no non-zero
common direct summand for any $n\ge1$. Immediately we have the following
recursion formula \cite[Th. 7.1]{I1} from \eqref{Y and U}.

\begin{thm}\label{formula}
In Theorem \ref{ladder}, we have the following
equalities in $K_0(\CC)$.
\begin{eqnarray*}
&Y_0=X,\ Y_1=\theta^-X,\ Y_n=(\theta^-Y_{n-1}-\tau^-Y_{n-2})_+\ (n\ge2),&\\
&Z_n=\theta^-Y_n-Y_{n+1},\ U_n=(\theta^-Y_n-\tau^-Y_{n-1})_-.&
\end{eqnarray*}
\end{thm}

We can apply the above observation to calculate $\dim_k\Ext^1_R(-,R)=\dim_k\underline{\Hom}_R(\tau^-R,-)$.  We remark that this kind of counting argument first appeared in the work of Gabriel \cite{Gabriel}.

\begin{example}\label{D52(i)}\t{For the group $\mathbb{D}_{5,2}$ the AR quiver is
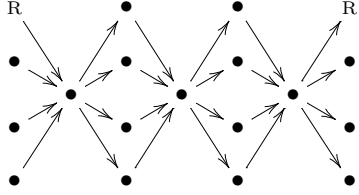
\[
\xymatrix@C=10pt@R=2pt{{\txt{\scriptsize R}}\ar[3,1]&&\bullet\ar[3,1]&&\bullet\ar[3,1]&&{\txt{\scriptsize R}}\\
&&&&&&\\
\bullet\ar[1,1]&&\bullet\ar[1,1]&&\bullet\ar[1,1]&&\bullet\\
&\bullet\ar[-3,1]\ar[-1,1]\ar[1,1]\ar[3,1]&&\bullet\ar[-3,1]\ar[-1,1]\ar[1,1]\ar[3,1]&&\bullet\ar[-3,1]\ar[-1,1]\ar[1,1]\ar[3,1]&\\
\bullet\ar[-1,1]&&\bullet\ar[-1,1]&&\bullet\ar[-1,1]&&\bullet\\
&&&&&&\\
\bullet\ar[-3,1]&&\bullet\ar[-3,1]&&\bullet\ar[-3,1]&&\bullet}
\]
where the left and right hand sides are identified and the AR translation shifts everything one place to the left. The counting argument begins as follows:
\[
\begin{array}{cccc}
\begin{array}{c}
\xymatrix@C=0pt@R=1pt{{{}\save[]*\txt{\scriptsize{R}}\restore}&&{{}\save[]*\txt{\scriptsize{1}}\restore}&&.&&{{}\save[]*\txt{\scriptsize{R}}\restore}\\
&&&&&&\\
.&&.&&.&&.\\
&.&&.&&.&\\
.&&.&&.&&.\\
&&&&&&\\
.&&.&&.&&.}\end{array}
&
\begin{array}{c}
\xymatrix@C=0pt@R=1pt{{}\save[]*\txt{\scriptsize R}\restore&&{}\save[]*\txt{\scriptsize 1}\restore&&.&&{}\save[]*\txt{\scriptsize R}\restore\\
&&&&&&\\
.&&.&&.&&.\\
&.&&{}\save[]*\txt{\scriptsize 1}\restore&&.&\\
.&&.&&.&&.\\
&&&&&&\\
.&&.&&.&&.}\end{array}
&
\begin{array}{c}
\xymatrix@C=0pt@R=1pt{{}\save[]*\txt{\scriptsize R}\restore&&{}\save[]*\txt{\scriptsize 1}\restore&&{}\save[]*\txt{\scriptsize 0}\restore&&{}\save[]*\txt{\scriptsize R}\restore\\
&&&&&&\\
.&&.&&{}\save[]*\txt{\scriptsize 1}\restore&&.\\
&.&&{}\save[]*\txt{\scriptsize 1}\restore&&.&\\
.&&.&&{}\save[]*\txt{\scriptsize 1}\restore&&.\\
&&&&&&\\
.&&.&&{}\save[]*\txt{\scriptsize 1}\restore&&.}\end{array}
&
\begin{array}{c}
\xymatrix@C=0pt@R=1pt{{}\save[]*\txt{\scriptsize R}\restore&&{}\save[]*\txt{\scriptsize 1}\restore&&{}\save[]*\txt{\scriptsize 0}\restore&&{}\save[]*\txt{\scriptsize R}\restore\\
&&&&&&\\
.&&.&&{}\save[]*\txt{\scriptsize 1}\restore&&.\\
&.&&{}\save[]*\txt{\scriptsize 1}\restore&&{}\save[]*\txt{\scriptsize 2}\restore&\\
.&&.&&{}\save[]*\txt{\scriptsize 1}\restore&&.\\
&&&&&&\\
.&&.&&{}\save[]*\txt{\scriptsize 1}\restore&&.}\end{array}\\
&&&\\
\mbox{\scriptsize Step 1: $Y_0=\tau^-R$}&\mbox{\scriptsize Step 2: $Y_1$}&\mbox{\scriptsize Step 3: $Y_2$}&\mbox{\scriptsize Step 4: $Y_3$}
\end{array}
\]
Continuing in this fashion we see
\[
\begin{array}{c}
\xymatrix@C=2pt@R=1pt{{}\save[]*\txt{\scriptsize R}\restore&&{}\save[]*\txt{\scriptsize 1}\restore&&{}\save[]*\txt{\scriptsize 0}\restore&&{}\save[]*\txt{\scriptsize R}\restore&&{}\save[]*\txt{\scriptsize 1}\restore&&{}\save[]*\txt{\scriptsize 0}\restore\\
&&&&&&&&&&\\
.&&.&&{}\save[]*\txt{\scriptsize 1}\restore&&{}\save[]*\txt{\scriptsize 1}\restore&&{}\save[]*\txt{\scriptsize 0}\restore&&{}\save[]*\txt{\scriptsize 0}\restore\\
&.&&{}\save[]*\txt{\scriptsize 1}\restore&&{}\save[]*\txt{\scriptsize 2}\restore&&{}\save[]*\txt{\scriptsize 1}\restore&&{}\save[]*\txt{\scriptsize 0}\restore&\\
.&&.&&{}\save[]*\txt{\scriptsize 1}\restore&&{}\save[]*\txt{\scriptsize 1}\restore&&{}\save[]*\txt{\scriptsize 0}\restore&&{}\save[]*\txt{\scriptsize 0}\restore\\
&&&&&&\\
.&&.&&{}\save[]*\txt{\scriptsize 1}\restore&&{}\save[]*\txt{\scriptsize 1}\restore&&{}\save[]*\txt{\scriptsize 0}\restore&&{}\save[]*\txt{\scriptsize 0}\restore}
\end{array}
\]
which after identifying CM modules gives us the following picture:
\[
\begin{array}{c}
\xymatrix@C=2pt@R=1pt{{}\save[]*\txt{\scriptsize R}\restore&&{}\save[]*\txt{\scriptsize 2}\restore&&{}\save[]*\txt{\scriptsize 0}\restore&&{}\save[]*\txt{\scriptsize R}\restore\\
&&&&&&\\
{}\save[]*\txt{\scriptsize 1}\restore&&{}\save[]*\txt{\scriptsize 0}\restore&&{}\save[]*\txt{\scriptsize 1}\restore&&{}\save[]*\txt{\scriptsize 1}\restore\\
&{}\save[]*\txt{\scriptsize 1}\restore&&{}\save[]*\txt{\scriptsize 1}\restore&&{}\save[]*\txt{\scriptsize 2}\restore&\\
{}\save[]*\txt{\scriptsize 1}\restore&&{}\save[]*\txt{\scriptsize 0}\restore&&{}\save[]*\txt{\scriptsize 1}\restore&&{}\save[]*\txt{\scriptsize 1}\restore\\
&&&&&&\\
{}\save[]*\txt{\scriptsize 1}\restore&&{}\save[]*\txt{\scriptsize 0}\restore&&{}\save[]*\txt{\scriptsize 1}\restore&&{}\save[]*\txt{\scriptsize 1}\restore}\\
\\
{ \scriptstyle{\dim_k\Ext^1_R(-,R)=\dim_k\underline{\Hom}_R(\tau^-R,-)}}
\end{array}
\]
From this we read off that the specials are precisely those which sit in the following positions in the AR quiver:}
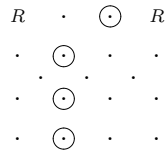
\[
\begin{array}{c}
\xymatrix@C=0pt@R=1pt{{}\save[]*\txt{\scriptsize R}\restore&&.&&{{}\drop\xycircle<4pt,4pt>{}.}&&{}\save[]*\txt{\scriptsize R}\restore\\
&&&&&&\\
.&&{{}\drop\xycircle<4pt,4pt>{}.}&&.&&.\\
&.&&.&&.&\\
.&&{{}\drop\xycircle<4pt,4pt>{}.}&&.&&.\\
&&&&&&\\
.&&{{}\drop\xycircle<4pt,4pt>{}.}&&.&&.}\end{array}
\]
\end{example}
Associated to the left ladder in Theorem \ref{ladder}, we call a commutative diagram 
\[\begin{array}{llllllll}
X=Y_0&\xrightarrow{f_0}&Y_1&\xrightarrow{f_1}&Y_2&\xrightarrow{f_2}&Y_3&\to\cdots\\
\ \downarrow{\scriptstyle b_0}&&\ \downarrow{\scriptstyle b_1\oplus 0}&&
\ \downarrow{\scriptstyle b_2\oplus 0\oplus 0}&&
\ \downarrow{\scriptstyle b_3\oplus 0\oplus 0\oplus 0}&\\
0=Z_0&\xrightarrow{{g_0\choose h_0}}&Z_1\oplus U_1&\xrightarrow{{g_1\choose h_1}\oplus 1_{U_1}}
&Z_2\oplus U_2\oplus U_1&\xrightarrow{{g_2\choose h_2}\oplus 1_{U_2}\oplus 1_{U_1}}&
Z_3\oplus U_3\oplus U_2\oplus U_1&\to\cdots
\end{array}\]
an \emph{extended left ladder} of $X$ in $\CC$.

\begin{thm}\label{cosy}
Let $R$ be a two-dimensional quotient singularity and $\CC=\overline{\CM}(R)$.
For any $X\in\CM(R)$, define $U_n\in\CC$ by Theorem \ref{formula}.
Then $\Omega^-X\simeq\bigoplus_{n\ge0}U_n$ in $\CC$.
\end{thm}
\begin{proof}
(i) We shall construct a commutative diagram
\begin{equation}\label{preladder}
\begin{array}{cccccccc}
X=A_0&\xrightarrow{f_0}&A_1&\xrightarrow{f_1}&A_2&\xrightarrow{f_2}&A_3&\xrightarrow{f_3}\cdots\\
\downarrow^{a_0}&&\downarrow^{a_1}&&\downarrow^{a_2}&&\downarrow^{a_3}&\\
0=B_0&\xrightarrow{g_0}&B_1&\xrightarrow{g_1}&B_2&\xrightarrow{g_2}&B_3&\xrightarrow{g_3}\cdots
\end{array}\end{equation}
in $\CM(R)$ with $a_n\in J_{\CM(R)}$ as follows.

When we have a morphism $a_n\in J_{\CM(R)}$, we write $a_n\simeq
(b_n\ c_n):A_n\simeq C_n\oplus I_n\to B_n$,
where $I_n$ is a maximal summand of $A_n$ contained in $\add\omega$.
Since $b_n\in J_{\CM(R)}$, we have a commutative diagram
\[\begin{array}{lllll}
C_n&\xrightarrow{\mu^-_{C_n}}&\theta^-C_n&\xrightarrow{\nu^-_{C_n}}&\tau^-C_n\\
\ \downarrow{\scriptstyle {1\choose 0}}&&\ \downarrow{\scriptstyle d_n}\\
A_n\simeq C_n\oplus I_n&\xrightarrow{(b_n\ c_n)}&B_n
\end{array}\]
This gives a commutative diagram 
\[\begin{array}{lll}
A_n\simeq C_n\oplus I_n&\xrightarrow{{\mu^-_{C_n}\ 0\choose\ 0\ \ \ 1}}&\theta^-C_n\oplus I_n\\
\ \downarrow{\scriptstyle a_n\simeq(b_n\ c_n)}&&\ \downarrow{{\nu^-_{C_n}\ 0\choose\ d_n\ \ c_n}}\\
B_n&\xrightarrow{{0\choose 1}}&\tau^-C_n\oplus B_n
\end{array}\]
Let $a_{n+1}:A_{n+1}\to B_{n+1}$ be a maximal direct summand of the two-termed complex
${\nu^-_{C_n}\ 0\choose\ d_n\ \ c_n}:\theta^-C_n\oplus I_n\to\tau^-C_n\oplus B_n$ contained in $J_{\CM(R)}$.
Then we have a commutative diagram
\[\begin{array}{lll}
A_n&\xrightarrow{f_n}&A_{n+1}\\
\ \downarrow{\scriptstyle a_n}&&\ \downarrow{\scriptstyle a_{n+1}}\\
B_n&\xrightarrow{g_n}&B_{n+1}.
\end{array}\]

(ii) From our construction, the sequence
\[0\to A_n\xrightarrow{{a_n\choose -f_n}}B_n\oplus A_{n+1}
\xrightarrow{(g_n\ a_{n+1})}B_{n+1}\to0\]
is isomorphic to a direct sum of an almost split sequence of $C_n$ and a complex
\[0\to I_n\xrightarrow{{1\choose 0}}I_n\oplus X_n\xrightarrow{(0\ 1)}X_n\to0\]
for some $X_n\in\CM(R)$.

(iii) By (ii), the image of the commutative diagram \eqref{preladder}
under the functor $\CM(R)\to\overline{\CM}(R)=\CC$ is an extended left ladder of $X$ in the $\tau$-category $\CC$.
Thus we have
\[A_n\simeq Y_n\ \mbox{ and }\ B_n\simeq Z_n\oplus(\bigoplus_{i=1}^nU_i)\]
in $\CC$ for any $n$.
On the other hand, using (ii) and the commutative diagram \eqref{preladder}, one can inductively show that
\[0\to X\xrightarrow{f_{n-1}\cdots f_1f_0}A_n\xrightarrow{a_n}B_n\to0\]
is an exact sequence for any $n\ge0$.

Since $R$ is representation-finite, we have $J_{\CC}^m=0$ for sufficiently large $m$. Then $A_m\simeq Y_m=0$ and $Z_m=0$ hold in $\CC$.
Hence we have $A_m\in\add\omega$ and $B_m=\Omega^-X$ in $\CM(R)$.
Consequently, we have $\Omega^-X=B_m\simeq\bigoplus_{i=1}^mU_i$ in $\CC$.
\end{proof}

In practice in this paper we use the dual of the above result.  To do this is standard, but we must first set up notation.  Firstly, we have the following dual version of Theorem~\ref{ladder} and Theorem~\ref{formula}.
\begin{thm}\label{rightladder}
Let $\CC$ be a $\tau$-category and $X\in\CC$.
\begin{itemize}
\item[(a)] There exists a commutative diagram (called a \emph{right ladder} of $X$)
\[\begin{array}{ccccccccc}
\cdots\xrightarrow{g^{\prime}_3}& Z^{\prime}_3&\xrightarrow{g^{\prime}_2}&Z^{\prime}_2&\xrightarrow{g^{\prime}_1}&Z^{\prime}_1&\xrightarrow{g^{\prime}_0}&Z^{\prime}_0=0\\
&\downarrow^{b^{\prime}_3}&&\downarrow^{b^{\prime}_2}&&\downarrow^{b^{\prime}_1}&&\downarrow^{b^{\prime}_0}&\\
\cdots\xrightarrow{f^{\prime}_3}& Y^{\prime}_3&\xrightarrow{f^{\prime}_2}&Y^{\prime}_2&\xrightarrow{f^{\prime}_1}&Y^{\prime}_1&\xrightarrow{f^{\prime}_0}&Y^{\prime}_0=X,
\end{array}\]
and objects $U^{\prime}_{n+1}\in\CC$ and a morphism $h^{\prime}_n\in\CC(U^{\prime}_{n+1},Z^{\prime}_{n})$ such that
\[
Z^{\prime}_{n+1}\oplus U^{\prime}_{n+1}\xrightarrow{{g^{\prime}_n\ h^{\prime}_{n}\choose b^{\prime}_{n+1}\ 0}}
Z^{\prime}_{n}\oplus Y^{\prime}_{n+1}\xrightarrow{(b^{\prime}_n\ -f^{\prime}_n)}Y^{\prime}_{n}
\]
is a right $\tau$-sequence for any $n\ge0$.
\item[(b)] For any $n\ge0$, we have an isomorphism
$(J_{\CC}^n/J_{\CC}^{n+1})(-,X)\simeq(\CC/J_{\CC})(-,Y^{\prime}_n)$
of functors on $\CC$.
In particular, if $\bigcap_{i\ge0}J_{\CC}^i=0$, then
\[\dim_k\CC(Y,X)=\sum_{n\ge0}m_Y(Y^{\prime}_n)\]
holds for any indecomposable $Y\in\CC$.
\item[(c)] We have the following equalities in $K_0(\CC)$.
\begin{eqnarray*}
&Y^{\prime}_0=X,\ Y^{\prime}_1=\theta X,\ Y^{\prime}_{n}=(\theta Y^{\prime}_{n-1}-\tau Y^{\prime}_{n-2})_+\ (n\ge2),&\\
&Z^{\prime}_n=\theta Y^{\prime}_{n}-Y^{\prime}_{n+1},\ U^{\prime}_n=(\theta Y^{\prime}_n-\tau Y^{\prime}_{n-1})_-.&
\end{eqnarray*}
\end{itemize}
\end{thm}
This leads to the dual version of Theorem~\ref{cosy}
\begin{thm}
Let $R$ be a two-dimensional quotient singularity and $\CC=\underline{\CM}(R)$.
For any $X\in\CM(R)$, define $U^{\prime}_n\in\CC$ by Theorem \ref{rightladder}(c).
Then $\Omega X\simeq\bigoplus_{n\ge0}U^{\prime}_n$ in $\CC$.
\end{thm}
We now illustrate how to calculate the syzygy. 
\begin{example}
\t{Consider the group $\mathbb{D}_{14,9}$, then the AR quiver of $\C{}[[x,y]]^{\mathbb{D}_{14,9}}$ is
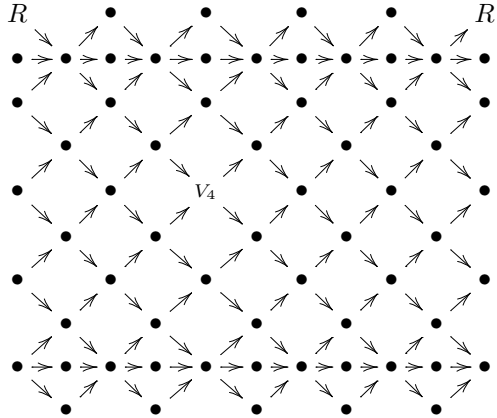
\[
\begin{array}{c}
\xymatrix@R=6pt@C=6pt{R\ar[1,1] &&\bullet\ar[1,1]&&\bullet\ar[1,1]&&\bullet\ar[1,1]&&\bullet\ar[1,1]&&R\\
\bullet\ar[r]&\bullet\ar[r]\ar[-1,1]\ar[1,1]&\bullet\ar[r]&\bullet\ar[r]\ar[-1,1]\ar[1,1]&\bullet\ar[r]&\bullet\ar[r]\ar[-1,1]\ar[1,1]&\bullet\ar[r]&\bullet\ar[r]\ar[-1,1]\ar[1,1]&\bullet\ar[r]&\bullet\ar[r]\ar[-1,1]\ar[1,1]&\bullet\\
\bullet\ar[-1,1]\ar[1,1]&&\bullet\ar[-1,1]\ar[1,1]&&\bullet\ar[-1,1]\ar[1,1]&&\bullet\ar[-1,1]\ar[1,1]&&\bullet\ar[-1,1]\ar[1,1]&&\bullet\\
&\bullet\ar[-1,1]\ar[1,1]&&\bullet\ar[-1,1]\ar[1,1]&&\bullet\ar[-1,1]\ar[1,1]&&\bullet\ar[-1,1]\ar[1,1]&&\bullet\ar[-1,1]\ar[1,1]&&\\
\bullet\ar[-1,1]\ar[1,1]&&\bullet\ar[-1,1]\ar[1,1]&&_{V_4}\ar[-1,1]\ar[1,1]&&\bullet\ar[-1,1]\ar[1,1]&&\bullet\ar[-1,1]\ar[1,1]&&\bullet\\
&\bullet\ar[1,1]\ar[-1,1]&&\bullet\ar[-1,1]\ar[1,1]&&\bullet\ar[1,1]\ar[-1,1]&&\bullet\ar[-1,1]\ar[1,1]&&\bullet\ar[-1,1]\ar[1,1]&&\\
\bullet\ar[1,1]\ar[-1,1]&&\bullet\ar[-1,1]\ar[1,1]&&\bullet\ar[1,1]\ar[-1,1]&&\bullet\ar[-1,1]\ar[1,1]&&\bullet\ar[-1,1]\ar[1,1]&&\bullet\\
&\bullet\ar[-1,1]\ar[1,1]&&\bullet\ar[-1,1]\ar[1,1]&&\bullet\ar[-1,1]\ar[1,1]&&\bullet\ar[-1,1]\ar[1,1]&&\bullet\ar[-1,1]\ar[1,1]&\\
\bullet\ar[r]\ar[-1,1]\ar[1,1]&\bullet\ar[r]&\bullet\ar[r]\ar[-1,1]\ar[1,1]&\bullet\ar[r]&\bullet\ar[r]\ar[-1,1]\ar[1,1]&\bullet\ar[r]&\bullet\ar[r]\ar[-1,1]\ar[1,1]&\bullet\ar[r]&\bullet\ar[r]\ar[-1,1]\ar[1,1]&\bullet\ar[r]&\bullet\\
&\bullet\ar[-1,1]&&\bullet\ar[-1,1]&&\bullet\ar[-1,1]&&\bullet\ar[-1,1]&&\bullet\ar[-1,1]&}
\end{array}
\]
where the left and right hand sides of the picture are identified, and where we have illustrated the module $V_4$ whose syzygy we would like to compute.  To do this, proceed as follows:
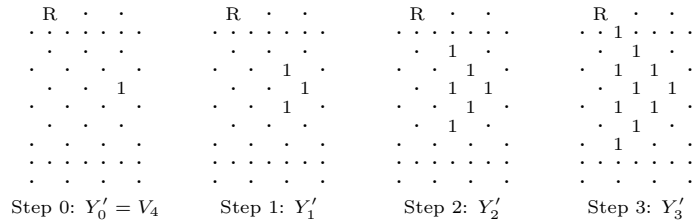
\[
\begin{array}{cccc}
\begin{array}{c}
\xymatrix@R=0pt@C=-2pt{&{{}\save[]*\txt{\scriptsize{R}}\restore}&&.&&.&\\
.&.&.&.&.&.&.\\
&.&&.&&.&\\
.&&.&&.&&.\\
&.&&.&&{{}\save[]*\txt{\scriptsize{1}}\restore}&\\
.&&.&&.&&.\\
&.&&.&&.&\\
.&&.&&.&&.\\
.&.&.&.&.&.&.\\
.&&.&&.&&.
}
\end{array} &\begin{array}{c}
\xymatrix@R=0pt@C=-2pt{&{{}\save[]*\txt{\scriptsize{R}}\restore}&&.&&.&\\
.&.&.&.&.&.&.\\
&.&&.&&.&\\
.&&.&&{{}\save[]*\txt{\scriptsize{1}}\restore}&&.\\
&.&&.&&{{}\save[]*\txt{\scriptsize{1}}\restore}&\\
.&&.&&{{}\save[]*\txt{\scriptsize{1}}\restore}&&.\\
&.&&.&&.&\\
.&&.&&.&&.\\
.&.&.&.&.&.&.\\
.&&.&&.&&.
}
\end{array}&\begin{array}{c}
\xymatrix@R=0pt@C=-2pt{&{{}\save[]*\txt{\scriptsize{R}}\restore}&&.&&.&\\
.&.&.&.&.&.&.\\
&.&&{{}\save[]*\txt{\scriptsize{1}}\restore}&&.&\\
.&&.&&{{}\save[]*\txt{\scriptsize{1}}\restore}&&.\\
&.&&{{}\save[]*\txt{\scriptsize{1}}\restore}&&{{}\save[]*\txt{\scriptsize{1}}\restore}&\\
.&&.&&{{}\save[]*\txt{\scriptsize{1}}\restore}&&.\\
&.&&{{}\save[]*\txt{\scriptsize{1}}\restore}&&.&\\
.&&.&&.&&.\\
.&.&.&.&.&.&.\\
.&&.&&.&&.
}
\end{array}&\begin{array}{c}
\xymatrix@R=0pt@C=-2pt{&{{}\save[]*\txt{\scriptsize{R}}\restore}&&.&&.&\\
.&.&{{}\save[]*\txt{\scriptsize{1}}\restore}&.&.&.&.\\
&.&&{{}\save[]*\txt{\scriptsize{1}}\restore}&&.&\\
.&&{{}\save[]*\txt{\scriptsize{1}}\restore}&&{{}\save[]*\txt{\scriptsize{1}}\restore}&&.\\
&.&&{{}\save[]*\txt{\scriptsize{1}}\restore}&&{{}\save[]*\txt{\scriptsize{1}}\restore}&\\
.&&{{}\save[]*\txt{\scriptsize{1}}\restore}&&{{}\save[]*\txt{\scriptsize{1}}\restore}&&.\\
&.&&{{}\save[]*\txt{\scriptsize{1}}\restore}&&.&\\
.&&{{}\save[]*\txt{\scriptsize{1}}\restore}&&.&&.\\
.&.&.&.&.&.&.\\
.&&.&&.&&.
}
\end{array}\\ \mbox{\scriptsize Step 0: $Y^{\prime}_0=V_4$}&\mbox{\scriptsize Step 1: $Y^{\prime}_1$}&\mbox{\scriptsize Step 2: $Y^{\prime}_2$}&\mbox{\scriptsize Step 3: $Y^{\prime}_3$}
\end{array}
\]
Now in Step 4 below $R$ absorbs a 1 (since we are working in $\underline{\CM}(R)$), and then the calculation continues
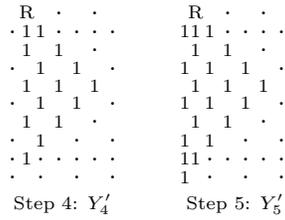
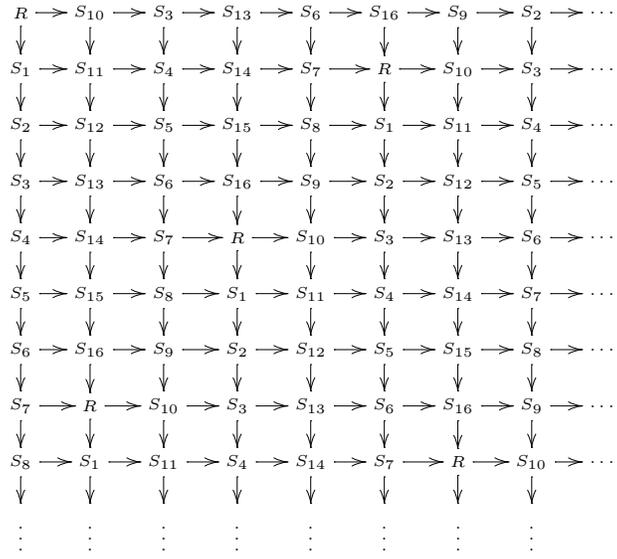
\[
\begin{array}{cc}
\begin{array}{c}
\xymatrix@R=0pt@C=-2pt{&{{}\save[]*\txt{\scriptsize{R}}\restore}&&.&&.&\\
.&{{}\save[]*\txt{\scriptsize{1}}\restore}&{{}\save[]*\txt{\scriptsize{1}}\restore}&.&.&.&.\\
&{{}\save[]*\txt{\scriptsize{1}}\restore}&&{{}\save[]*\txt{\scriptsize{1}}\restore}&&.&\\
.&&{{}\save[]*\txt{\scriptsize{1}}\restore}&&{{}\save[]*\txt{\scriptsize{1}}\restore}&&.\\
&{{}\save[]*\txt{\scriptsize{1}}\restore}&&{{}\save[]*\txt{\scriptsize{1}}\restore}&&{{}\save[]*\txt{\scriptsize{1}}\restore}&\\
.&&{{}\save[]*\txt{\scriptsize{1}}\restore}&&{{}\save[]*\txt{\scriptsize{1}}\restore}&&.\\
&{{}\save[]*\txt{\scriptsize{1}}\restore}&&{{}\save[]*\txt{\scriptsize{1}}\restore}&&.&\\
.&&{{}\save[]*\txt{\scriptsize{1}}\restore}&&.&&.\\
.&{{}\save[]*\txt{\scriptsize{1}}\restore}&.&.&.&.&.\\
.&&.&&.&&.
}
\end{array}&\begin{array}{c}
\xymatrix@R=0pt@C=-2pt{&{{}\save[]*\txt{\scriptsize{R}}\restore}&&.&&.&\\
{{}\save[]*\txt{\scriptsize{1}}\restore}&{{}\save[]*\txt{\scriptsize{1}}\restore}&{{}\save[]*\txt{\scriptsize{1}}\restore}&.&.&.&.\\
&{{}\save[]*\txt{\scriptsize{1}}\restore}&&{{}\save[]*\txt{\scriptsize{1}}\restore}&&.&\\
{{}\save[]*\txt{\scriptsize{1}}\restore}&&{{}\save[]*\txt{\scriptsize{1}}\restore}&&{{}\save[]*\txt{\scriptsize{1}}\restore}&&.\\
&{{}\save[]*\txt{\scriptsize{1}}\restore}&&{{}\save[]*\txt{\scriptsize{1}}\restore}&&{{}\save[]*\txt{\scriptsize{1}}\restore}&\\
{{}\save[]*\txt{\scriptsize{1}}\restore}&&{{}\save[]*\txt{\scriptsize{1}}\restore}&&{{}\save[]*\txt{\scriptsize{1}}\restore}&&.\\
&{{}\save[]*\txt{\scriptsize{1}}\restore}&&{{}\save[]*\txt{\scriptsize{1}}\restore}&&.&\\
{{}\save[]*\txt{\scriptsize{1}}\restore}&&{{}\save[]*\txt{\scriptsize{1}}\restore}&&.&&.\\
{{}\save[]*\txt{\scriptsize{1}}\restore}&{{}\save[]*\txt{\scriptsize{1}}\restore}&.&.&.&.&.\\
{{}\save[]*\txt{\scriptsize{1}}\restore}&&.&&.&&.
}
\end{array}\\
\mbox{\scriptsize Step 4: $Y^{\prime}_4$}&\mbox{\scriptsize Step 5: $Y^{\prime}_5$}
\end{array}
\]
Continuing in this fashion we get{\tiny{
\[
\begin{array}{c}
\xymatrix@R=-2pt@C=-2pt{.&&.&&.&&R&&0&&1&&0&&1&&R&&0&&1&&0&&1&&R&&.&&.&\\
.&.&.&.&.&0&0&1&1&1&0&1&1&1&0&1&1&2&1&1&0&1&1&1&0&1&1&1&.&.&.&.\\
.&&.&&0&&1&&1&&1&&1&&1&&2&&2&&1&&1&&1&&1&&1&&.&\\
&.&&0&&1&&1&&1&&1&&1&&2&&2&&2&&1&&1&&1&&1&&1&&.\\
.&&\minus 1&&1&&1&&1&&1&&1&&2&&2&&2&&2&&1&&1&&1&&1&&1&\\
&.&&0&&1&&1&&1&&1&&2&&2&&2&&2&&2&&1&&1&&1&&1&&.\\
.&&.&&0&&1&&1&&1&&2&&2&&2&&2&&2&&2&&1&&1&&1&&.&\\
&.&&.&&0&&1&&1&&2&&2&&2&&2&&2&&2&&2&&1&&1&&.&&.\\
.&.&.&.&.&.&0&0&1&1&2&1&2&1&2&1&2&1&2&1&2&1&2&1&2&1&1&.&.&.&.&.\\
&.&&.&&.&&0&&1&&1&&1&&1&&1&&1&&1&&1&&1&&.&&.&&.}
\end{array}
\]
}}from which we see that the only non-zero $U^{\prime}$ is $U^{\prime}_{27}$, which corresponds to the position of the $\minus 1$ in the above picture.  But we know that this is the position of $V_4^*$, so we deduce that $\Omega V_4\simeq V_4^*$.  Hence there is a short exact sequence}
\[
\begin{array}{c}
\xymatrix@C=20pt@R=20pt{0\ar[r]&V_4^*\ar[r]&{R^{4}}\ar[r]&{V_4}\ar[r]&0
}
\end{array}.
\]
\end{example}

\section{Type $\mathbb{A}$}
Consider the groups $\frac{1}{r}(1,a)$ where $1\leq a<r$ and
$(r,a)=1$.  Call the generator of the group $\sigma$.  To this
combinatorial data we associate the Hirzebruch-Jung continued
fraction expansion of $\frac{r}{a}$, namely
\[
\frac{r}{a}=\alpha_1-\frac{1}{\alpha_2 - \frac{1}{\alpha_3 -
\frac{1}{(...)}}} :=[\alpha_1,\hdots,\alpha_n]
\]
with each $\alpha_i\geq 2$.  The following combinatorics are crucial:
\begin{defin}
For $\frac{r}{a}=[\alpha_1,\hdots,\alpha_n]$ define the $i$-series as follows:
\[
\begin{array}{ccl}
i_0=r & i_1=a & i_{t}=\alpha_{t-1}i_{t-1}-i_{t-2}\mbox{ for }2\leq
t\leq n+1.
\end{array}
\]
\end{defin}
It's easy to see that
\[
\begin{array}{ccccccccccc}
i_0=r &>& i_1=a &>& i_2 &>& \hdots &>& i_n=1 &>& i_{n+1}=0.
\end{array}
\]
For the one-dimensional representation $\rho_t$ define
\[
S_{t}=\{ f\in\C{}[[x,y]] : \sigma\cdot f=\e^t f \}
\]
The following classification of the special CM modules in this case is well known:
\begin{thm}\cite{Wunram_cyclicBook}
For $G=\frac{1}{r}(1,a)$ with
$\frac{r}{a}=[\alpha_1,\hdots,\alpha_n]$, the special CM
modules are precisely those $S_{i_p}$ for $0\leq p\leq n$.
\end{thm}
There are many different proofs of this fact, so we do not go out of
our way in this paper to give another one.  We do however note that
our counting argument recovers another way to determine the specials in type $\mathbb{A}$ given by Ito \cite{Ito_special} using combinatorics of the $G$-Hilbert scheme.   Ito's result works only for cyclic groups (the toric case):
let us explain in an example why our counting argument
recovers her method.  Consider the group
$\frac{1}{17}(1,10)$.  The AR quiver is {\tiny{
\[
\xymatrix@C=10pt@R=10pt{{R}\ar[r]\ar[d]&{S_{10}}\ar[r]\ar[d]&{S_{3}}\ar[r]\ar[d]&{S_{13}}\ar[r]\ar[d]&{S_{6}}\ar[r]\ar[d]&{S_{16}}\ar[r]\ar[d]&{S_{9}}\ar[r]\ar[d]&{S_{2}}\ar[r]\ar[d]&{\hdots}\\
{S_{1}}\ar[r]\ar[d]&{S_{11}}\ar[r]\ar[d]&{S_{4}}\ar[r]\ar[d]&{S_{14}}\ar[r]\ar[d]&{S_{7}}\ar[r]\ar[d]&{R}\ar[r]\ar[d]&{S_{10}}\ar[r]\ar[d]&{S_{3}}\ar[r]\ar[d]&{\hdots}\\
{S_{2}}\ar[r]\ar[d]&{S_{12}}\ar[r]\ar[d]&{S_{5}}\ar[r]\ar[d]&{S_{15}}\ar[r]\ar[d]&{S_{8}}\ar[r]\ar[d]&{S_{1}}\ar[r]\ar[d]&{S_{11}}\ar[r]\ar[d]&{S_{4}}\ar[r]\ar[d]&{\hdots}\\
{S_{3}}\ar[r]\ar[d]&{S_{13}}\ar[r]\ar[d]&{S_{6}}\ar[r]\ar[d]&{S_{16}}\ar[r]\ar[d]&{S_{9}}\ar[r]\ar[d]&{S_{2}}\ar[r]\ar[d]&{S_{12}}\ar[r]\ar[d]&{S_{5}}\ar[r]\ar[d]&{\hdots}\\
{S_{4}}\ar[r]\ar[d]&{S_{14}}\ar[r]\ar[d]&{S_{7}}\ar[r]\ar[d]&{R}\ar[r]\ar[d]&{S_{10}}\ar[r]\ar[d]&{S_{3}}\ar[r]\ar[d]&{S_{13}}\ar[r]\ar[d]&{S_{6}}\ar[r]\ar[d]&{\hdots}\\
{S_{5}}\ar[r]\ar[d]&{S_{15}}\ar[r]\ar[d]&{S_{8}}\ar[r]\ar[d]&{S_{1}}\ar[r]\ar[d]&{S_{11}}\ar[r]\ar[d]&{S_{4}}\ar[r]\ar[d]&{S_{14}}\ar[r]\ar[d]&{S_{7}}\ar[r]\ar[d]&{\hdots}\\
{S_{6}}\ar[r]\ar[d]&{S_{16}}\ar[r]\ar[d]&{S_{9}}\ar[r]\ar[d]&{S_{2}}\ar[r]\ar[d]&{S_{12}}\ar[r]\ar[d]&{S_{5}}\ar[r]\ar[d]&{S_{15}}\ar[r]\ar[d]&{S_{8}}\ar[r]\ar[d]&{\hdots}\\
{S_{7}}\ar[r]\ar[d]&{R}\ar[r]\ar[d]&{S_{10}}\ar[r]\ar[d]&{S_{3}}\ar[r]\ar[d]&{S_{13}}\ar[r]\ar[d]&{S_{6}}\ar[r]\ar[d]&{S_{16}}\ar[r]\ar[d]&{S_{9}}\ar[r]\ar[d]&{\hdots}\\
{S_{8}}\ar[r]\ar[d]&{S_{1}}\ar[r]\ar[d]&{S_{11}}\ar[r]\ar[d]&{S_{4}}\ar[r]\ar[d]&{S_{14}}\ar[r]\ar[d]&{S_{7}}\ar[r]\ar[d]&{R}\ar[r]\ar[d]&{S_{10}}\ar[r]\ar[d]&{\hdots}\\
{\vdots}&{\vdots}&{\vdots}&{\vdots}&{\vdots}&{\vdots}&{\vdots}&{\vdots}
}
\]
}}where there is lots of identification.  We begin by placing a 1 in
the place of $\tau^{-1}R=S_{11}$ and begin counting: {\tiny
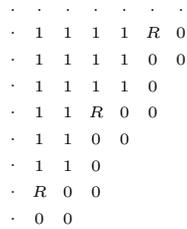
{
\[
\xymatrix@C=0.1pt@R=0.1pt{.&.&.&.&.&.&.\\
.&1&1&1&1&R&0&\\
.&1&1&1&1&0&0\\
.&1&1&1&1&0&\\
.&1&1&R&0&0\\
.&1&1&0&0&\\
.&1&1&0&\\
.&R&0&0\\
.&0&0&}
\]
}}Thus we read off that the specials are precisely those
CM modules which do not lie in the region covered by 1's.  But
this is precisely the region denoted $B(G)\backslash L(G)$ in
\cite[3.7]{Ito_special}.

\section{Type $\mathbb{D}$}
In this section we consider the groups $\mathbb{D}_{n,q}$ with
$1\leq q<n$ and $(n,q)=1$.  To this combinatorial data we again
associate the Hirzebruch-Jung continued fraction expansion of
$\frac{n}{q}$, namely
\[
\frac{n}{q}=\alpha_1-\frac{1}{\alpha_2 - \frac{1}{\alpha_3 -
\frac{1}{(...)}}} :=[\alpha_1,\hdots,\alpha_N]
\]
and the corresponding $i$-series
$i_0=n>i_1=q>\hdots>i_N=1>i_{N+1}=0$.  By \cite[2.11]{Brieskorn} the
dual graph of the minimal resolution of $\C{2}/\mathbb{D}_{n,q}$ is
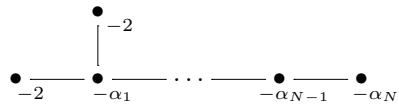
\[
\xymatrix@C=20pt@R=15pt{ &\bullet\ar@{-}[d]^<{-2}&&&\\
\bullet\ar@{-}[r]_<{-2} & \bullet\ar@{-}[r]_<{-\alpha_{1}}
&\hdots\ar@{-}[r] &\bullet\ar@{-}[r]_<{-\alpha_{N-1}} & \bullet
\ar@{}[r]_<{-\alpha_N}&}
\]
where the $\alpha$'s come from the the Hirzebruch-Jung continued
fraction expansion of $\frac{n}{q}$, and so $Z_f$ is either
\[
\begin{array}{ccc}
\begin{array}{c}\xymatrix@R=0pt@C=0pt{ &_{1}&&&\\
_{1}& _{1}&_{1}&\hdots&_{1}&_{1}}\end{array} &\mbox{or}&
\begin{array}{c}\xymatrix@R=2pt@C=2pt{ &_{1}&&&\\
_{1}& _{2}&\hdots&_{2}&_{1}&\hdots&_{1}}\end{array}\\
\t{if }\alpha_1\geq 3&& \t{if }\alpha_1=\hdots=\alpha_\nu=2\,\,\t{and }\left\{\begin{array}{cl} \alpha_{\nu+1}\geq 3&(\nu<N-1) \\ \alpha_{\nu+1} \geq 2&(\nu=N-1) \end{array}\right. .
\end{array}
\]
In either case denote the number of $2$'s in $Z_f$ by $\nu$.  By \cite{Wunram_generalpaper}  there are $N+2-\nu$
non-free rank 1 indecomposable special CM modules and $\nu$
rank 2 indecomposable special CM modules.  Thus once we exhibit these numbers of special CM modules, we have them all.

By \cite{AR_McKayGraphs} the universal cover of the AR quiver of $R=\C{}[[x,y]]^{\mathbb{D}_{n,q}}$ is always $\mathbb{Z}\widetilde{D}_{q+2}$,
but to give a more detailed description of the AR quiver of $R$ we need to split into cases.  Firstly if $n-q$ is odd then there are two cases
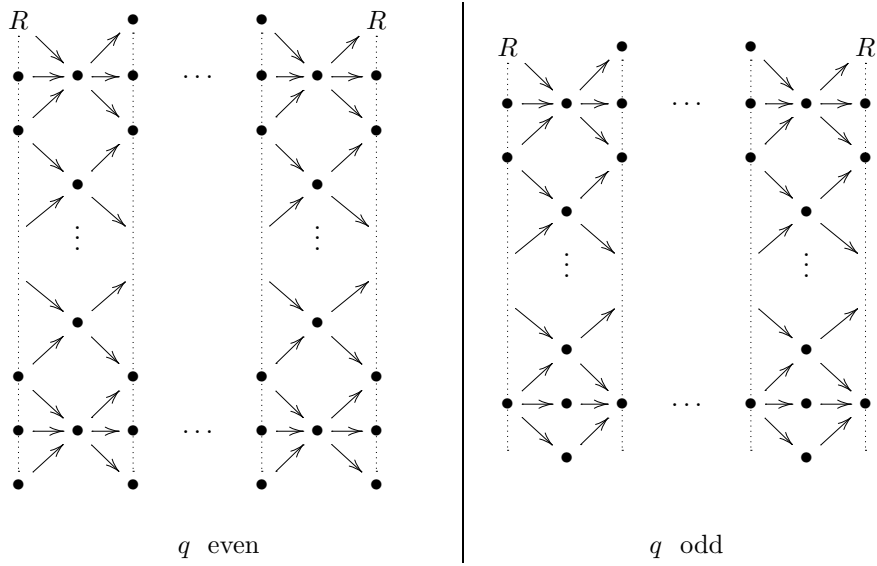
\[
\begin{array}{c|c}
\begin{array}{c}
\xymatrix@R=10pt@C=10pt{R\ar[1,1] &&\bullet&&\bullet\ar[1,1]&&R\\
\bullet\ar[r]&\bullet\ar[r]\ar[-1,1]\ar[1,1]&\bullet&\hdots&\bullet\ar[r]&\bullet\ar[r]\ar[-1,1]\ar[1,1]&\bullet\\
\bullet\ar[-1,1]\ar[1,1]&&\bullet&&\bullet\ar[-1,1]\ar[1,1]&&\bullet\\
&\bullet\ar[-1,1]\ar@{}[2,0]|{{\vdots}}\ar[1,1]&&&&\bullet\ar[-1,1]\ar@{}[2,0]|{{\vdots}}\ar[1,1]&\\
\ar[-1,1]&&&&\ar[-1,1]&&\\
\ar[1,1]&&&&\ar[1,1]&&\\
&\bullet\ar[1,1]\ar[-1,1]&&&&\bullet\ar[1,1]\ar[-1,1]&\\
\bullet\ar[1,1]\ar[-1,1]&&\bullet&&\bullet\ar[1,1]\ar[-1,1]&&\bullet\\
\bullet\ar[r]&\bullet\ar[r]\ar[-1,1]\ar[1,1]&\bullet&{\hdots}&\bullet\ar[r]&\bullet\ar[r]\ar[-1,1]\ar[1,1]&\bullet&\\
\bullet\ar[-1,1]\ar@{.}[-9,0]&&\bullet\ar@{.}[-9,0]&&\bullet\ar[-1,1]\ar@{.}[-9,0]&&\bullet\ar@{.}[-9,0]}
\end{array}
&
\begin{array}{c}
\xymatrix@R=10pt@C=10pt{R\ar[1,1] &&\bullet&&\bullet\ar[1,1]&&R\\
\bullet\ar[r]&\bullet\ar[r]\ar[-1,1]\ar[1,1]&\bullet&\hdots&\bullet\ar[r]&\bullet\ar[r]\ar[-1,1]\ar[1,1]&\bullet\\
\bullet\ar[-1,1]\ar[1,1]&&\bullet&&\bullet\ar[-1,1]\ar[1,1]&&\bullet\\
&\bullet\ar[-1,1]\ar@{}[2,0]|{{\vdots}}\ar[1,1]&&&&\bullet\ar[-1,1]\ar@{}[2,0]|{{\vdots}}\ar[1,1]&\\
\ar[-1,1]&&&&\ar[-1,1]&&\\
\ar[1,1]&&&&\ar[1,1]&&\\
&\bullet\ar[1,1]\ar[-1,1]&&&&\bullet\ar[1,1]\ar[-1,1]&\\
\bullet\ar[r]\ar[-1,1]\ar[1,1]&\bullet\ar[r]&\bullet&{\hdots}&\bullet\ar[r]\ar[-1,1]\ar[1,1]&\bullet\ar[r]&\bullet\\
\ar@{.}[-8,0]&\bullet\ar[-1,1]&\ar@{.}[-8,0]&&\ar@{.}[-8,0]&\bullet\ar[-1,1]&\ar@{.}[-8,0]}
\end{array}\\
&\\
q \,\,\t{  even} & q\,\, \t{  odd}
\end{array}
\]
where the repeated block is just the AR quiver for $BD_{4\cdot q}$, the
binary dihedral group of order $4q$, and there are $n-q$ repetitions.  The left and right hand side of the picture are
identified. For completeness we mention that the AR quiver in the case $q=2$ is again
slightly different, but for such groups either we are inside $SL(2,\C{})$ or else all the specials have rank one.  Either way (using Theorem~\ref{TypeDR1} below for the later case) we understand the specials and so we can ignore the $q=2$
case.  Note also that in all cases when $n-q$ is odd that there are
no twists in the AR quiver.  When $n-q$ is even (which by $(n,q)=1$
forces $q$ odd) the AR quiver looks very similar, but now there is a
twist:
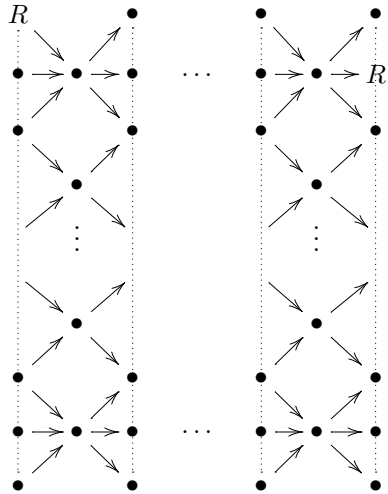
\[
\xymatrix@R=10pt@C=10pt{R\ar[1,1] &&\bullet&&\bullet\ar[1,1]&&\bullet\\
\bullet\ar[r]&\bullet\ar[r]\ar[-1,1]\ar[1,1]&\bullet&\hdots&\bullet\ar[r]&\bullet\ar[r]\ar[-1,1]\ar[1,1]&R\\
\bullet\ar[-1,1]\ar[1,1]&&\bullet&&\bullet\ar[-1,1]\ar[1,1]&&\bullet\\
&\bullet\ar[-1,1]\ar@{}[2,0]|{{\vdots}}\ar[1,1]&&&&\bullet\ar[-1,1]\ar@{}[2,0]|{{\vdots}}\ar[1,1]&\\
\ar[-1,1]&&&&\ar[-1,1]&&\\
\ar[1,1]&&&&\ar[1,1]&&\\
&\bullet\ar[1,1]\ar[-1,1]&&&&\bullet\ar[1,1]\ar[-1,1]&\\
\bullet\ar[1,1]\ar[-1,1]&&\bullet&&\bullet\ar[1,1]\ar[-1,1]&&\bullet\\
\bullet\ar[r]&\bullet\ar[r]\ar[-1,1]\ar[1,1]&\bullet&{\hdots}&\bullet\ar[r]&\bullet\ar[r]\ar[-1,1]\ar[1,1]&\bullet&\\
\bullet\ar[-1,1]\ar@{.}[-9,0]&&\bullet\ar@{.}[-9,0]&&\bullet\ar[-1,1]\ar@{.}[-9,0]&&\bullet\ar@{.}[-9,0]}
\]
Again the repeated block is just the AR quiver for $BD_{4\cdot q}$
and there are $n-q$ repetitions, but now the left hand side and
the right hand side of the picture are identified with a twist.  We
do not illustrate the twist fully as it is only the twist in the $R$
position that is important from the viewpoint of the proofs in this
section; for full details of the twist see \cite{AR_McKayGraphs}.

Note that since in the three cases the AR quivers are very similar the
proofs in this section which use the counting argument are all the
same, but care should be taken in the case when $n-q$ is even due to the
twist.

Let us now define some rank 1 and rank 2 CM modules as follows.  Define the rank 1 CM modules $W_+$, $W_-$ and for each $1\leq t\leq i_{\nu+1}+\nu (n-q)-1=q-1$ the rank 2 indecomposable CM module $V_t$ by the following positions in the AR quiver
{\tiny{
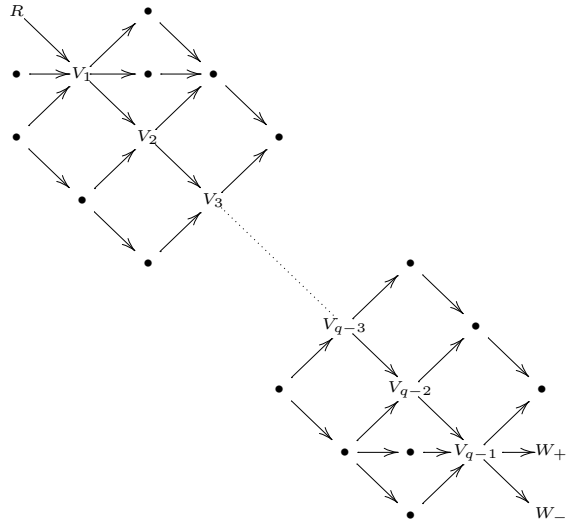
\[
\begin{array}{c}
\xymatrix@R=15pt@C=15pt{{}\save[]*{R}\restore\ar[dr(0.8)] &&\bullet\ar[dr(0.8)]&\\
\bullet\ar[r(0.8)]&{}\save[]*{V_1}\restore\ar[r(0.8)]\ar[ur(0.8)]\ar[dr(0.8)]&\bullet\ar[r(0.8)]&\bullet\ar[dr(0.8)]&\\
\bullet\ar[ur(0.8)]\ar[dr(0.8)]&&{}\save[]*{V_2}\restore\ar[ur(0.8)]\ar[dr(0.8)]&&\bullet&&\\
&\bullet\ar[ur(0.8)]\ar[dr(0.8)]&&{}\save[]*{V_3}\restore\ar[ur(0.8)]\ar@{.}[2,2]&&\\
&&\bullet\ar[ur(0.8)]&&&&\bullet\ar[dr(0.8)]&\\
&&&&&{}\save[]*{V_{q-3}}\restore\ar[dr(0.8)]\ar[ur(0.8)]&&\bullet\ar[dr(0.8)]&\\
&&&&\bullet\ar[dr(0.8)]\ar[ur(0.8)]&&{}\save[]*{V_{q-2}}\restore\ar[dr(0.8)]\ar[ur(0.8)]&&\bullet&\\
&&&&&\bullet\ar[r(0.8)]\ar[ur(0.8)]\ar[dr(0.8)]&\bullet\ar[r(0.6)]&{}\save[]*{V_{q-1}}\restore\ar[dr(0.8)]\ar[ur(0.8)]&{}\save[]*{\quad W_+}\restore\ar@{<-}[l(0.6)]\\
&&&&&&\bullet\ar[ur(0.8)]&&{}\save[]*{\quad W_-}\restore}
\end{array}
\]
}}i.e. all the $V_t$ lie on the diagonal leaving the vertex $R$, whilst $W_+$ and $W_-$ are the two rank 1 CM modules at the bottom of the diagonal.  Furthermore for every $1\leq t\leq n-q$ define the rank 1 CM module $W_t$ by the following position in the AR quiver:
{\tiny{
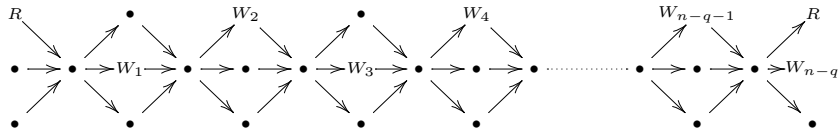
\[
\begin{array}{c}
\xymatrix@R=12pt@C=12pt{{}\save[]*{R}\restore\ar[1,1] &&\bullet\ar[1,1]&&{}\save[]*{W_2}\restore&&\bullet\ar[1,1]&&{}\save[]*{W_4}\restore&&&&{}\save[]*{W_{n-q-1}}\restore&&{}\save[]*{R}\restore\\
\bullet\ar[r]&\bullet\ar[r(0.7)]\ar[-1,1]\ar[1,1]&{}\save[]*{W_1}\restore&\bullet\ar@{<-}[l(0.75)]\ar[r]\ar[ur(0.8)]\ar[1,1]&\bullet\ar[r]&\bullet\ar[r(0.7)]\ar[-1,1]\ar[1,1]\ar@{<-}[ul(0.75)]&{}\save[]*{W_3}\restore&\bullet\ar@{<-}[l(0.75)]\ar[r]\ar[ur(0.8)]\ar[1,1]&\bullet\ar[r]&\bullet\ar@{<-}[ul(0.75)]\ar@{.}[0,2]&&\bullet\ar[r]\ar[ur(0.8)]\ar[1,1]&\bullet\ar[r]&\bullet\ar[r(0.5)]\ar[-1,1]\ar[1,1]\ar@{<-}[ul(0.75)]&{}\save[]*{W_{n-q}}\restore\\
\bullet\ar[-1,1]&&\bullet\ar[-1,1]&&\bullet\ar[-1,1]&&\bullet\ar[-1,1]&&\bullet\ar[-1,1]&&&&\bullet\ar[-1,1]&&\bullet}
\end{array}
\]
}}i.e. they all live on the non-zero zigzag leaving $R$.  Note that $W_t$ contains the polynomial $(xy)^t$. Also note that when $n-q$ is even the picture changes slightly since the position of $R$ on the right is twisted, but even then all the $W_t$ are mutually distinct.  The following is known:
\begin{thm}\cite[3.11]{Wemyss_reconstruct_D(i)}\label{TypeDR1}
For any $\mathbb{D}_{n,q}$ the following rank one CM modules are special: $W_{+}, W_{-}$ and also $W_{i_{\nu+1}},\hdots, W_{i_N}$. Further there are no other
indecomposable non-free rank one specials, so if $\nu=0$ these are all the non-free indecomposable special CM modules.
\end{thm}
Thus if $n>2q$ (i.e. $\nu=0$) there is nothing left to prove since
the above theorem gives all the specials.  We do however need to
take care of the case $n<2q$, when rank 2 indecomposable specials can occur. 

\begin{thm}
Consider the group $\mathbb{D}_{n,q}$ with $n<2q$,  then for all
$0\leq s\leq \nu-1$,  $V_{i_{\nu+1}+s(n-q)}$ is special. Furthermore
these are all the rank 2 indecomposable special CM modules.
\end{thm}
\begin{proof}
Trivially $\oplus_{s=0}^{\nu-1}V_{i_{\nu+1}+s(n-q)}$ is a CM module;
we aim to show that its first syzygy is
$\oplus_{s=0}^{\nu-1}V_{i_{\nu+1}+s(n-q)}^*$ then by
Theorem~\ref{characterization of SCM} it follows that each $V_{i_{\nu+1}+s(n-q)}$ is special.  We do this by using the counting argument on the AR quiver as shown in Section 4.  If $\nu=1$ this is an easy extension of the example given in Section 4;  the $\nu=2$ case is similarly easy.  Hence assume that $\nu=3$.  To illustrate this
technique let us first prove the theorem in a specific example.
Consider the group $\mathbb{D}_{23,18}$ - the continued fraction
expansion of $\frac{23}{18}$ is $[2,2,2,3,3]$ and so $\nu=3$,
$i_{\nu+1}=i_4=3$ and $n-q=5$.  Consequently we consider $V_3$,
$V_8$ and $V_{13}$.  To compute the syzygy of the sum of these,
start with {\tiny{
\[
\xymatrix@R=-4pt@C=-4pt{
R&&{\cdot}&&{\cdot}&&{\cdot}&&{\cdot}&&R&&{\cdot}&&\\
{\cdot}&3&{\cdot}&{\cdot}&{\cdot}&{\cdot}&{\cdot}&{\cdot}&{\cdot}&{\cdot}&{\cdot}&{\cdot}&{\cdot}&{\cdot}&\\
{\cdot}&&3&&{\cdot}&&{\cdot}&&{\cdot}&&{\cdot}&&{\cdot}&&\\
&{\cdot}&&{{}\drop\xycircle<4pt,4pt>{}{3}}&&{\cdot}&&{\cdot}&&{\cdot}&&{\cdot}&&{\cdot}&\\
{\cdot}&&{\cdot}&&2&&{\cdot}&&{\cdot}&&{\cdot}&&{\cdot}&&\\
&{\cdot}&&{\cdot}&&2&&{\cdot}&&{\cdot}&&{\cdot}&&{\cdot}&\\
{\cdot}&&{\cdot}&&{\cdot}&&2&&{\cdot}&&{\cdot}&&{\cdot}&&\\
&{\cdot}&&{\cdot}&&{\cdot}&&2&&{\cdot}&&{\cdot}&&{\cdot}&\\
{\cdot}&&{\cdot}&&{\cdot}&&{\cdot}&&{{}\drop\xycircle<4pt,4pt>{}{2}}&&{\cdot}&&{\cdot}&&\\
&{\cdot}&&{\cdot}&&{\cdot}&&{\cdot}&&1&&{\cdot}&&{\cdot}&\\
{\cdot}&&{\cdot}&&{\cdot}&&{\cdot}&&{\cdot}&&1&&{\cdot}&&\\
&{\cdot}&&{\cdot}&&{\cdot}&&{\cdot}&&{\cdot}&&1&&{\cdot}&\\
{\cdot}&&{\cdot}&&{\cdot}&&{\cdot}&&{\cdot}&&{\cdot}&&1&&\\
&{\cdot}&&{\cdot}&&{\cdot}&&{\cdot}&&{\cdot}&&{\cdot}&&{{}\drop\xycircle<4pt,4pt>{}{1}}&\\
{\cdot}&&{\cdot}&&{\cdot}&&{\cdot}&&{\cdot}&&{\cdot}&&{\cdot}&&\\
&{\cdot}&&{\cdot}&&{\cdot}&&{\cdot}&&{\cdot}&&{\cdot}&&{\cdot}&\\
{\cdot}&&{\cdot}&&{\cdot}&&{\cdot}&&{\cdot}&&{\cdot}&&{\cdot}&&\\
{\cdot}&{\cdot}&{\cdot}&{\cdot}&{\cdot}&{\cdot}&{\cdot}&{\cdot}&{\cdot}&{\cdot}&{\cdot}&{\cdot}&{\cdot}&{\cdot}&\\
{\cdot}&&{\cdot}&&{\cdot}&&{\cdot}&&{\cdot}&&{\cdot}&&{\cdot}&&}
\]
}}where we have circled the positions of $V_3$, $V_8$ and $V_{13}$
only for clarity; the circles do not effect the counting.  Now count
backwards using the rules in Section 4.  Doing this we obtain
{\tiny{
\[
\xymatrix@R=-4pt@C=-4pt{
.&&.&&R&&0&&1&&0&&1&&R\ar@{.}@<1ex>[18,18]&&0\ar@{.}@<-1ex>[9,-9]&&1&&0&&1&&R\ar@{.}@<1ex>[18,18]&\ar@{.}@<-0.75ex>[13,-13]&0\ar@{.}@<-1ex>[14,-14]&&1&&0&&1&&R&&0&&3&&0&&3&&R&&.&&.&&.&&.&&R&&.&&\\
.&.&.&0&0&1&1&1&0&1&1&1&0&1&1&2&1&1&0&1&1&1&0&1&1&2&1&1&0&1&1&1&0&1&1&4&3&3&0&3&3&3&0&3&3&3&.&.&.&.&.&.&.&.&.&.&.&.&\\
.&&0&&1&&1&&1&&1&&1&&2&&2&&1&&1&&1&&2&&2&&1&&1&&1&&4&&4&&3&&3&&3&&3&&3&&.&&.&&.&&.&&.&&\\
&\minus 1&&1&&1&&1&&1&&1&&2&&2&&2&&1&&1&&2&&2&&2&&1&&1&&4&&4&&4&&3&&3&&3&&3&&3&&.&&.&&.&&.&&.&\\
.&&0&&1&&1&&1&&1&&2&&2&&2&&2&&1&&2&&2&&2&&2&&1&&4&&4&&4&&4&&3&&3&&3&&3&&2&&.&&.&&.&&.&&\\
&.&&0&&1&&1&&1&&2&&2&&2&&2&&2&&2&&2&&2&&2&&2&&4&&4&&4&&4&&4&&3&&3&&3&&2&&2&&.&&.&&.&&.&\\
.&&.&&0&&1&&1&&2&&2&&2&&2&&2&&3&&2&&2&&2&&2&&5&&4&&4&&4&&4&&4&&3&&3&&2&&2&&2&&.&&.&&.&&\\
&.&&.&&0&&1&&2&&2&&2&&2&&2&&3&&3&&2&&2&&2&&5&&5&&4&&4&&4&&4&&4&&3&&2&&2&&2&&2&&.&&.&&.&\\
.&&.&&.&&\minus 1&&2&&2&&2&&2&&2&&3&&3&&3&&2&&2&&5&&5&&5&&4&&4&&4&&4&&4&&2&&2&&2&&2&&2&&.&&.&&\\
&.&&.&&.&&0&&2&&2&&2&&2&&3&&3&&3&&3&&2&&5&&5&&5&&5&&4&&4&&4&&4&&3&&2&&2&&2&&2&&1&&.&&.&\\
.&&.&&.&&.&&0&&2&&2&&2&&3&&3&&3&&3&&3&&5&&5&&5&&5&&5&&4&&4&&4&&3&&3&&2&&2&&2&&1&&1&&.&&\\
&.&&.&&.&&.&&0&&2&&2&&3&&3&&3&&3&&3&&6&&5&&5&&5&&5&&5&&4&&4&&3&&3&&3&&2&&2&&1&&1&&1&&.&\\
.&&.&&.&&.&&.&&0&&2&&3&&3&&3&&3&&3&&6&&6&&5&&5&&5&&5&&5&&4&&3&&3&&3&&3&&2&&1&&1&&1&&1&&\\
&.&&.&&.&&.&&.&&\minus 1&&3&&3&&3&&3&&3&&6&&6&&6&&5&&5&&5&&5&&5&&3&&3&&3&&3&&3&&1&&1&&1&&1&&1&\\
.&&.&&.&&.&&.&&.&&0&&3&&3&&3&&3&&6&&6&&6&&6&&5&&5&&5&&5&&4&&3&&3&&3&&3&&2&&1&&1&&1&&1&&\ar@{.}@<-1.5ex>[-11,-11]\\
&.&&.&&.&&.&&.&&.&&0&&3&&3&&3&&6&&6&&6&&6&&6&&5&&5&&5&&4&&4&&3&&3&&3&&2&&2&&1&&1&&1&&.&\\
.&&.&&.&&.&&.&&.&&.&&0&&3&&3&&6&&6&&6&&6&&6&&6&&5&&5&&4&&4&&4&&3&&3&&2&&2&&2&&1&&1&&.&&\\
.&.&.&.&.&.&.&.&.&.&.&.&.&.&.&0&0&3&3&6&3&6&3&6&3&6&3&6&3&6&3&6&3&5&2&4&2&4&2&4&2&4&2&3&1&2&1&2&1&2&1&2&1&1&.&.&.&.&\\
.&&.&&.&&.&&.&&.&&.&&.&&0\ar@<-1ex>@<-1ex>@{.}[-14,14]&\ar@<-1ex>@<-0.5ex>@{.}[-18,18]&3\ar@{.}@<1ex>[-11,-11]&&3&&3&&3&&3&&3&&3&&3\ar@<-1.2ex>@{.}[-11,11]&\ar@<-1ex>@{.}[-15,15]&2\ar@{.}@<1ex>[-14,-14]&&2&&2&&2&&2\ar@<-1ex>@{.}[-11,11]&\ar@<-1ex>@{.}[-10,10]&1\ar@<1ex>@{.}[-14,-14]&&1&&1&&1&&1\ar@{.}@<-1.5ex>[-6,6]&\ar@{.}@<1ex>[-18,-18]&.\ar@{.}@<1ex>[-11,-11]&&.&&}
\]
}}where the dotted lines in the above picture simply illustrate the
pattern; they do not effect the counting argument.  From the positions of the $\minus 1$'s in the above picture we can read off the syzygy of $\oplus_{s=0}^{2}V_{3+5s}$.  There positions correspond to $V_3^*$, $V_8^*$ and $V_{13}^*$ since $(-)^*$ gives an anti-isomorphism of the AR quiver.
So we read off that there is a short exact sequence
\[
\xymatrix{0\ar[r]& \oplus_{s=0}^{2}V_{3+5s}^*\ar[r]&R^{12}\ar[r]&\oplus_{s=0}^{2}V_{3+5s}\ar[r]&0}
\]
proving $V_3$, $V_5$ and $V_{13}$ are special.  Now for the general case, notice that
for any $\mathbb{D}_{n,q}$ with $\nu=3$ the proof is identical to the above but for practical purposes we only illustrate the pattern:
{\tiny{
\[
\xymatrix@R=-2pt@C=-2pt{
&&&&{}\save[]*\txt{R}\restore&&&&&&&&&&{}\save[]*\txt{R}\restore\ar@{.}@<1ex>[19,19]&&\ar@{.}@<-1ex>[9,-9]&&&&&&&&{}\save[]*\txt{R}\restore\ar@{.}@<1ex>[19,19]&&\ar@{.}@<-1ex>[14,-14]&&&&&&&&{}\save[]*\txt{R}\restore&&&&&&&&&&{}\save[]*\txt{R}\restore&&&&&&&&&&&&&&\\
&&&&&&&&&&&&&&&&&&&&&&&&&&&&&&&&&&&&&&&&&&&&&&&&&&&&&&&&&&\\
&&&&&&&&&&&&&&&&&&&&&&&&&&&&&&&&&&&&&&&&&&&&&&&&&&&&&&&&&&\\
&{}\save[]*\txt{-1}\restore&&&&&&&&&&&&&&&&&&&&&&&&&&&&&&&&&&&&&&&&&&&&&&&&&&&&&&&&&\\
&&&&&&&&&&&&&&&&&&&&&&&&&&&&&&&&&&&&&&&&&&&&&&&&&&&&&&&&&&\\
&&&&&&&&&&&&&&{}\save[]*\txt{2}\restore&&&&&&&&&&{}\save[]*\txt{2}\restore&&&&&&&&&&&&&&&&&&&&&&&&&&&&&&&&&\\
&&&&&&&&&&&&&&&&&&&&&&&&&&&&&&&&&&&&&&&&&&&&&&&&&&&&&&&&&&\\
&&&&&&&&&&&&&&&&&&&&&&&&&&&&&&&&&&&&&{}\save[]*\txt{4}\restore&&&&&&&&&&&&&&&&&&&&&\\
&&&&&&{}\save[]*\txt{-1}\restore&&&&&&&&&&&&&&&&&&&&&&&&&&&&&&&&&&&&&&&&&{}\save[]*\txt{2}\restore&&&&&&&\\
&&&&&&&&&&&&&&&&&&&&&&&&&&&&&&&&&&&&&&&&&&&&&&&&&&&&&&&&&&\\
&&&&&&&&&&&&&&&&&&&&&&&&&&&&&&&&&&&&&&&&&&&&&&&&&&&&&&&&&&\\
&&&&&&&&&&&&&&&&&&&&&&&&&&&&&&&&&&&&&&&&&&&&&&&&&&&&&&&&&&\\
&&&&&&&&&&&&&&&&&&&{}\save[]*\txt{3}\restore&&&&&&&&&&&&&{}\save[]*\txt{5}\restore&&&&&&&&&&&&&&&&&&&&&&&&&&\\
&&&&&&&&&&&{}\save[]*\txt{-1}\restore&&&&&&&&&&&&&&&&&&&&&&&&&&&&&&&&{}\save[]*\txt{3}\restore&&&&&&&&&{}\save[]*\txt{1}\restore&&&&&&\\
&&&&&&&&&&&&&&&&&&&&&&&&&&&&&&&&&&&&&&&&&&&&&&&&&&&&&&&&&\ar@{.}@<-1ex>[-11,-11]\\
&&&&&&&&&&&&&&&&&&&&&&&&&&&&&&&&&&&&&&&&&&&&&&&&&&&&&&&&&&\\
&&&&&&&&&&&&&&&&&&&&&&&&&&&&&&&&&&&&&&&&&&&&&&&&&&&&&&&&&&\\
&&&&&&&&&&&&&&&&&&&&&&&&&&&&&&&&&&&&&&&&&&&&&&&&&&&&&&&&&&\\
&&&&&&&&&&&&&&&&&\ar@<-1ex>@<-0.5ex>@{.}[-18,18]&\ar@{.}@<1ex>[-11,-11]&&&&&&&&&&&&&&\ar@<-1.2ex>@{.}[-16,16]&&&&&&&&&&\ar@<-1ex>@{.}[-11,11]&&&&&&&&&\ar@{.}@<-1ex>[-6,6]&&\ar@{.}@<1ex>[-18,-18]&&&&&\\
&&&&&&&&&&&&&&&&&&&&&&&&&&&&&&&&&&&&&&&&&&&&&&&&&&&&&&&&&&}
\]
}}In general there are two sizes of box: the smaller is $(n-q)\times (n-q)$
whereas the other is $i_{\nu}\times (n-q)$.  Notice that
$i_{\nu}=i_{\nu+1}+(n-q)$ and so the boxes always stay within the AR
quiver.  Care should be taken over the twist when $n-q$ is even, but
we suppress the details since the proof remains the same.

For any $\mathbb{D}_{n,q}$ with $\nu=4$ it is clear how this game
continues - again for practical purposes we only illustrate the
pattern: {\tiny{
\[
\xymatrix@R=-2pt@C=-2pt{
&&&&{}\save[]*\txt{R}\restore&&&&&&&&&&{}\save[]*\txt{R}\restore\ar@{.}@<1ex>[24,24]&&\ar@{.}@<-1ex>[9,-9]&&&&&&&&{}\save[]*\txt{R}\restore\ar@{.}@<1ex>[24,24]&&\ar@{.}@<-1ex>[14,-14]&&&&&&&&{}\save[]*\txt{R}\restore\ar@{.}@<1ex>[24,24]&&\ar@{.}@<-1ex>[19,-19]&&&&&&&&{}\save[]*\txt{R}\restore\ar@{.}@<1ex>[24,24]&&&&&&&&&&{}\save[]*\txt{R}\restore&&&&&&&&&&&&&&\\
&&&&&&&&&&&&&&&&&&&&&&&&&&&&&&&&&&&&&&&&&&&&&&&&&&&&&&&&&&&&&&&&&&&&\\
&&&&&&&&&&&&&&&&&&&&&&&&&&&&&&&&&&&&&&&&&&&&&&&&&&&&&&&&&\ar@{.}[17,17]&&&&&&&&&&&\\
&{}\save[]*\txt{-1}\restore&&&&&&&&&&&&&&&&&&&&&&&&&&&&&&&&&&&&&&&&&&&&&&&&&&&&&&&&&&&&&&&&&&&\\
&&&&&&&&&&&&&&&&&&&&&&&&&&&&&&&&&&&&&&&&&&&&&&&&&&&&&&&&&&&&&&&&&&&&\\
&&&&&&&&&&&&&&&&&&&&&&&&&{}\save[]*\txt{2}\restore&&&&&&&&&&{}\save[]*\txt{2}\restore&&&&&&&&&&&&&&&&&&&&&&&&&&&&&&&&\\
&&&&&&&&&&&&&{}\save[]*\txt{2}\restore&&&&&&&&&&&&&&&&&&&&&&&&&&&&&&&&&&&&&&&&&&&&&&&&&&&&&&&\\
&&&&&&&&&&&&&&&&&&&&&&&&&&&&&&&&&&&&&&&&&&&&&&&{}\save[]*\txt{5}\restore&&&&&&&&&&&&&&&&&&&&&\\
&&&&&&{}\save[]*\txt{-1}\restore&&&&&&&&&&&&&&&&&&&&&&&&&&&&&&&&&&&&&&&&&&&&&&&&&&&&{}\save[]*\txt{3}\restore&&&&&&\\
&&&&&&&&&&&&&&&&&&&&&&&&&&&&&&&&&&&&&&&&&&&&&&&&&&&&&&&&&&&&&&&&&&&&\\
&&&&&&&&&&&&&&&&&&&&&&&&&&&&&&{}\save[]*\txt{3}\restore&&&&&&&&&&&&&&&&&&&&&&&&&&&&&&&&&&&&&&\\
&&&&&&&&&&&&&&&&&&{}\save[]*\txt{3}\restore&&&&&&&&&&&&&&&&&&&&&&&&&&&&&&&&&&&&&&&&&&&&&&&&&&\\
&&&&&&&&&&&&&&&&&&&&&&&&&&&&&&&&&&&&&&&&&&{}\save[]*\txt{6}\restore&&&&&&&&&&&&&&&&&&&&&&&&&&&\\
&&&&&&&&&&&{}\save[]*\txt{-1}\restore&&&&&&&&&&&&&&&&&&&&&&&&&&&&&&&&&&&&&&&&&&{}\save[]*\txt{4}\restore&&&&&&&&&&{}\save[]*\txt{2}\restore&&&&&\\
&&&&&&&&&&&&&&&&&&&&&&&&&&&&&&&&&&&&&&&&&&&&&&&&&&&&&&&&&&&&&&&&&&&&&&&&&&\\
&&&&&&&&&&&&&&&&&&&&&&&&&&&&&&&&&&&&&&&&&&&&&&&&&&&&&&&&&&&&&&&&&&&&&&&&&&\\
&&&&&&&&&&&&&&&&&&&&&&&&&&&&&&&&&&&&&&&&&&&&&&&&&&&&&&&&&&&&&&&&&&&&&&&&&&\\
&&&&&&&&&&&&&&&&&&&&&&&&{}\save[]*\txt{4}\restore&&&&&&&&&&&&&{}\save[]*\txt{7}\restore&&&&&&&&&&&&&&&&&&&&&&&&&&&&&&&&&&&&&\\
&&&&&&&&&&&&&&&&{}\save[]*\txt{-1}\restore&&&&&&&&&&&&&&&&&&&&&&&&&&&&&&&&{}\save[]*\txt{5}\restore&&&&&&&&&&{}\save[]*\txt{3}\restore&&&&&&&&&&{}\save[]*\txt{1}\restore&&&&&&&&\\
&&&&&&&&&&&&&&&&&&&&&&&&&&&&&&&&&&&&&&&&&&&&&&&&&&&&&&&&&&&&&&&&&&&&&&&&&&\\
&&&&&&&&&&&&&&&&&&&&&&&&&&&&&&&&&&&&&&&&&&&&&&&&&&&&&&&&&&&&&&&&&&&&&&&&&&\\
&&&&&&&&&&&&&&&&&&&&&&&&&&&&&&&&&&&&&&&&&&&&&&&&&&&&&&&&&&&&&&&&&&&&&&&&&&\\
&&&&&&&&&&&&&&&&&&&&&&&&&&&&&&&&&&&&&&&&&&&&&&&&&&&&&&&&&&&&&&&&&&&&&&&&&&\\
&&&&&&&&&&&&&&&&&&&&&&\ar@<-1ex>@<-0.5ex>@{.}[-23,23]&\ar@{.}@<1ex>[-16,-16]&&&&&&&&&&&&&&\ar@<-1.2ex>@{.}[-21,21]&&&&&&&&&&\ar@<-1ex>@{.}[-16,16]&&&&&&&&&&\ar@{.}@<-1ex>[-11,11]&&&&&&&&&&\ar@{.}@<-1ex>[-6,6]&&&&&&&&\\
&&&&&&&&&&&&&&&&&&&&&&&&&&&&&&&&&&&&&&&&&&&&&&&&&&&&&&&&&&&&&&&&&&&&}
\]
}}Again there are two sizes of box: the smaller is $(n-q)\times
(n-q)$ whereas the other is $i_{\nu}\times (n-q)$.  Again since
$i_{\nu}=i_{\nu+1}+(n-q)$ the boxes always stay within the AR
quiver.   The pattern and argument is the same for arbitrary $\nu\geq
3$.  These are all the rank two indecomposable specials since (as explained
above) there are precisely $\nu$ rank two indecomposable special CM modules.
\end{proof}

\begin{remark}
\t{In this section we have assumed Wunram's results to obtain the
classification of the specials; in particular we have assumed
knowledge of the dual graph of the minimal resolution to get the
correct number of special CM modules with the correct ranks.
Note that our counting argument described in Section 4 can be used
to classify the specials without assuming any of the geometry, but
the proof is very hard to write down and involves splitting into
many cases, so we refrain from doing it. In all remaining sections
we never assume any of the geometry as the counting argument gives
us the answer without requiring it.}
\end{remark}

\section{Type $\mathbb{T}$}
Here we have $\mathbb{T}_m$ with $m\equiv 1,3$ or $5$ mod $6$.  By
\cite{AR_McKayGraphs} the AR quiver of $\C{}[[x,y]]^{\mathbb{T}_m}$
with $m\equiv 1,5$ mod 6 is {\tiny{
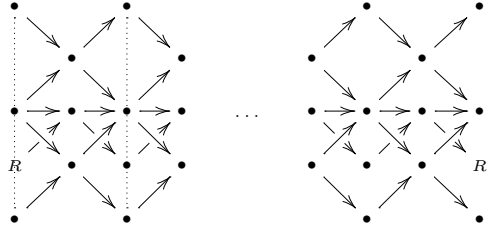
\[
\begin{array}{rcl}
\begin{array}{c}
\xymatrix@C=11pt@R=11pt{
\bullet\ar[1,1]&&\bullet\ar[1,1]\\
&\bullet\ar[1,1]\ar[-1,1]&&\bullet\\
\bullet\ar[1,1]\ar[-1,1]\ar[0,1]&\bullet\ar[0,1]\ar@{-->}[1,1]&\bullet\ar[1,1]\ar[-1,1]\ar[0,1]&\bullet\\
R\ar@{-->}[-1,1]&\bullet\ar[1,1]\ar[-1,1]&\bullet\ar@{-->}[-1,1]&\bullet\\
\ar@{.}[-4,0]\bullet\ar[-1,1]&&\ar@{.}[-4,0]\bullet\ar[-1,1]}
\end{array}
&\hdots&
\begin{array}{c}
\xymatrix@C=11pt@R=11pt{
&\bullet\ar[1,1]&&\bullet\\
\bullet\ar[1,1]\ar[-1,1]&&\bullet\ar[-1,1]\ar[1,1]\\
\bullet\ar[0,1]\ar@{-->}[1,1]&\bullet\ar[1,1]\ar[-1,1]\ar[0,1]&\bullet\ar[0,1]\ar@{-->}[1,1]&\bullet \\
\bullet\ar[1,1]\ar[-1,1]&\bullet\ar@{-->}[-1,1]&\bullet\ar[-1,1]\ar[1,1]&R\\
&\bullet\ar[-1,1]&&\bullet}
\end{array}
\end{array}
\]
}}where there are precisely $m$ repetitions of the original $\tilde{E}_6$ shown
in dotted lines.  The left and right hand sides of the picture are identified, and
there is no twist in this AR quiver.\\

For the group $\mathbb{T}_m$ with $m\equiv 3$ mod 6 the underlying
AR quiver is the same as the AR quiver for the other $\mathbb{T}_m$
above, just that there are now twists.  The best way to see this is
via an example - for the group $\mathbb{T}_{3}$ the AR quiver is
{\tiny
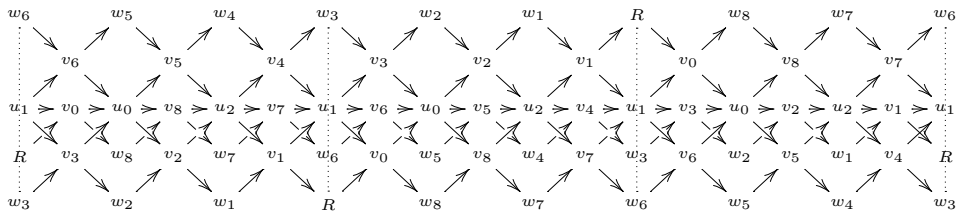
{
\[
\xymatrix@C=5pt@R=8pt{
w_{6}\ar[1,1]&&w_{5}\ar[1,1]&&w_{4}\ar[1,1]&&w_{3}\ar[1,1]&&w_{2}\ar[1,1]&&w_{1}\ar[1,1]&&R\ar[1,1]&&w_{8}\ar[1,1]&&w_{7}\ar[1,1]&&w_{6} \\
&v_{6}\ar[1,1]\ar[-1,1]&&v_{5}\ar[1,1]\ar[-1,1]&&v_{4}\ar[1,1]\ar[-1,1]&&v_{3}\ar[1,1]\ar[-1,1]&&v_{2}\ar[1,1]\ar[-1,1]&&v_{1}\ar[1,1]\ar[-1,1]&&v_{0}\ar[1,1]\ar[-1,1]&&v_{8}\ar[1,1]\ar[-1,1]&&v_{7}\ar[1,1]\ar[-1,1]&\\
u_{1}\ar[1,1]\ar[-1,1]\ar[0,1]&v_{0}\ar[0,1]\ar@{-->}[1,1]&u_{0}\ar[1,1]\ar[-1,1]\ar[0,1]&v_{8}\ar[0,1]\ar@{-->}[1,1]&u_{2}\ar[1,1]\ar[-1,1]\ar[0,1]&v_{7}\ar[0,1]\ar@{-->}[1,1]&u_{1}\ar[1,1]\ar[-1,1]\ar[0,1]&v_{6}\ar[0,1]\ar@{-->}[1,1]&u_{0}\ar[1,1]\ar[-1,1]\ar[0,1]&v_{5}\ar[0,1]\ar@{-->}[1,1]&u_{2}\ar[1,1]\ar[-1,1]\ar[0,1]&v_{4}\ar[0,1]\ar@{-->}[1,1]&u_{1}\ar[1,1]\ar[-1,1]\ar[0,1]&v_{3}\ar[0,1]\ar@{-->}[1,1]&u_{0}\ar[1,1]\ar[-1,1]\ar[0,1]&v_{2}\ar[0,1]\ar@{-->}[1,1]&u_{2}\ar[1,1]\ar[-1,1]\ar[0,1]&v_{1}\ar[0,1]\ar@{->}[1,1]&u_{1} \\
R\ar@{-->}[-1,1]&v_{3}\ar[1,1]\ar[-1,1]&w_{8}\ar@{-->}[-1,1]&v_{2}\ar@{-->}[-1,1]\ar[1,1]&w_{7}\ar@{-->}[-1,1]&v_{1}\ar[1,1]\ar[-1,1]&w_{6}\ar@{-->}[-1,1]&v_{0}\ar@{-->}[-1,1]\ar[1,1]&w_{5}\ar@{-->}[-1,1]&v_{8}\ar[1,1]\ar[-1,1]&w_{4}\ar@{-->}[-1,1]&v_{7}\ar@{-->}[-1,1]\ar[1,1]&w_{3}\ar@{-->}[-1,1]&v_{6}\ar[1,1]\ar[-1,1]&w_{2}\ar@{-->}[-1,1]&v_{5}\ar@{-->}[-1,1]\ar[1,1]&w_{1}\ar@{-->}[-1,1]&v_{4}\ar[1,1]\ar[-1,1]&R&\\
\ar@{.}[-4,0]w_{3}\ar[-1,1]&&w_{2}\ar[-1,1]&&w_{1}\ar[-1,1]&&\ar@{.}[-4,0]R\ar[-1,1]&&w_{8}\ar[-1,1]&&w_{7}\ar[-1,1]&&\ar@{.}[-4,0]w_{6}\ar[-1,1]&&w_{5}\ar[-1,1]&&w_{4}\ar[-1,1]&&w_{3}\ar@{.}[-4,0]&}
\]
}}The right and left hand sides of the picture are
identified. Notice that inside each segment we have the same
CM modules, in fact in each column of each segment there are
the same CM modules, just that they are rotated in each piece,
giving a twist to the AR quiver.  For full details see
\cite{AR_McKayGraphs}.

Before splitting into subfamilies to prove the results, it is
necessary to control what we call the free expansion:
\begin{defin}
For a given vertex $M$ in the AR quiver of ${\CM}(R)$, define the free expansion leaving $M$ to be the calculation of $Y_n$ given in Theorem~\ref{formula} for $\CC=\CM(R)$.
\end{defin}
We illustrate this in the example below (c.f. Example~\ref{D52(i)}).
\begin{example}\label{D52(ii)}\t{For the group $\mathbb{D}_{5,2}$ the AR quiver is
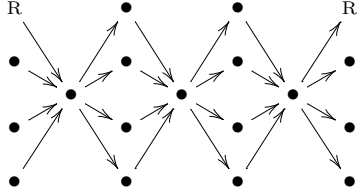
\[
\xymatrix@C=10pt@R=2pt{{\txt{\scriptsize R}}\ar[3,1]&&\bullet\ar[3,1]&&\bullet\ar[3,1]&&{\txt{\scriptsize R}}\\
&&&&&&\\
\bullet\ar[1,1]&&\bullet\ar[1,1]&&\bullet\ar[1,1]&&\bullet\\
&\bullet\ar[-3,1]\ar[-1,1]\ar[1,1]\ar[3,1]&&\bullet\ar[-3,1]\ar[-1,1]\ar[1,1]\ar[3,1]&&\bullet\ar[-3,1]\ar[-1,1]\ar[1,1]\ar[3,1]&\\
\bullet\ar[-1,1]&&\bullet\ar[-1,1]&&\bullet\ar[-1,1]&&\bullet\\
&&&&&&\\
\bullet\ar[-3,1]&&\bullet\ar[-3,1]&&\bullet\ar[-3,1]&&\bullet}
\]
where the left and right hand sides are identified. The free expansion leaving $\tau^{-1}R$ begins as follows:
\[
\begin{array}{ccccc}
\begin{array}{c}
\xymatrix@C=0pt@R=1pt{{{}\save[]*\txt{\scriptsize{R}}\restore}&&{{}\save[]*\txt{\scriptsize{1}}\restore}&&.&&.\\
&&&&&&\\
.&&.&&.&&.\\
&.&&.&&.&\\
.&&.&&.&&.\\
&&&&&&\\
.&&.&&.&&.}\end{array}
&
\begin{array}{c}
\xymatrix@C=0pt@R=1pt{{}\save[]*\txt{\scriptsize R}\restore&&{}\save[]*\txt{\scriptsize 1}\restore&&.&&.\\
&&&&&&\\
.&&.&&.&&.\\
&.&&{}\save[]*\txt{\scriptsize 1}\restore&&.&\\
.&&.&&.&&.\\
&&&&&&\\
.&&.&&.&&.}\end{array}
&
\begin{array}{c}
\xymatrix@C=0pt@R=1pt{{}\save[]*\txt{\scriptsize R}\restore&&{}\save[]*\txt{\scriptsize 1}\restore&&{}\save[]*\txt{\scriptsize 0}\restore&&.\\
&&&&&&\\
.&&.&&{}\save[]*\txt{\scriptsize 1}\restore&&.\\
&.&&{}\save[]*\txt{\scriptsize 1}\restore&&.&\\
.&&.&&{}\save[]*\txt{\scriptsize 1}\restore&&.\\
&&&&&&\\
.&&.&&{}\save[]*\txt{\scriptsize 1}\restore&&.}\end{array}
&
\begin{array}{c}
\xymatrix@C=0pt@R=1pt{{}\save[]*\txt{\scriptsize R}\restore&&{}\save[]*\txt{\scriptsize 1}\restore&&{}\save[]*\txt{\scriptsize 0}\restore&&.\\
&&&&&&\\
.&&.&&{}\save[]*\txt{\scriptsize 1}\restore&&.\\
&.&&{}\save[]*\txt{\scriptsize 1}\restore&&{}\save[]*\txt{\scriptsize 2}\restore&\\
.&&.&&{}\save[]*\txt{\scriptsize 1}\restore&&.\\
&&&&&&\\
.&&.&&{}\save[]*\txt{\scriptsize 1}\restore&&.}\end{array}
&
\begin{array}{c}
\xymatrix@C=0pt@R=1pt{{}\save[]*\txt{\scriptsize R}\restore&&{}\save[]*\txt{\scriptsize 1}\restore&&{}\save[]*\txt{\scriptsize 0}\restore&&{}\save[]*\txt{\scriptsize 2}\restore\\
&&&&&&\\
.&&.&&{}\save[]*\txt{\scriptsize 1}\restore&&{}\save[]*\txt{\scriptsize 1}\restore\\
&.&&{}\save[]*\txt{\scriptsize 1}\restore&&{}\save[]*\txt{\scriptsize 2}\restore&\\
.&&.&&{}\save[]*\txt{\scriptsize 1}\restore&&{}\save[]*\txt{\scriptsize 1}\restore\\
&&&&&&\\
.&&.&&{}\save[]*\txt{\scriptsize 1}\restore&&{}\save[]*\txt{\scriptsize 1}\restore}\end{array}\\
&&&\\
\mbox{\scriptsize Step 1: $Y_0=\tau^-R$}&\mbox{\scriptsize Step 2: $Y_1$}&\mbox{\scriptsize Step 3: $Y_2$}&\mbox{\scriptsize Step 4: $Y_3$}&\mbox{\scriptsize Step 5: $Y_4$}
\end{array}
\]
The calculation continues as
\[
\begin{array}{c}
\xymatrix@C=2pt@R=1pt{{}\save[]*\txt{\scriptsize R}\restore&&{}\save[]*\txt{\scriptsize 1}\restore&&{}\save[]*\txt{\scriptsize 0}\restore&&{}\save[]*\txt{\scriptsize 2}\restore&&{}\save[]*\txt{\scriptsize 1}\restore&&{}\save[]*\txt{\scriptsize 3}\restore\\
&&&&&&&&&&\\
.&&.&&{}\save[]*\txt{\scriptsize 1}\restore&&{}\save[]*\txt{\scriptsize 1}\restore&&{}\save[]*\txt{\scriptsize 2}\restore&&{}\save[]*\txt{\scriptsize 2}\restore\\
&.&&{}\save[]*\txt{\scriptsize 1}\restore&&{}\save[]*\txt{\scriptsize 2}\restore&&{}\save[]*\txt{\scriptsize 3}\restore&&{}\save[]*\txt{\scriptsize 4}\restore&\\
.&&.&&{}\save[]*\txt{\scriptsize 1}\restore&&{}\save[]*\txt{\scriptsize 1}\restore&&{}\save[]*\txt{\scriptsize 2}\restore&&{}\save[]*\txt{\scriptsize 2}\restore\\
&&&&&&\\
.&&.&&{}\save[]*\txt{\scriptsize 1}\restore&&{}\save[]*\txt{\scriptsize 1}\restore&&{}\save[]*\txt{\scriptsize 2}\restore&&{}\save[]*\txt{\scriptsize 2}\restore}
\end{array}\hdots
\]
and does not stop.}
\end{example}
Since the free expansion takes place in ${\CM}(R)$ and not $\underline{\CM}(R)$ the numbers always become larger and larger.  The reason we introduce the free expansion is that to understand the counting argument in $\underline{\CM}(R)$ one must first be able to control the counting argument in ${\CM}(R)$ (i.e. the free expansion). 

In type $\mathbb{T}$, since the underlying AR quivers are the same in all cases the free expansion can be verified in one proof, but beware of the possible twist when using this
lemma:
\begin{lemma}\label{freeT}
In type $\mathbb{T}$ consider the free expansion from $\tau^{-1}R$ and choose $t\geq 3$. Then
between columns $12(t-2)-1$ and $12(t-2)+10$ the free expansion
looks like {\tiny{
\[
\xymatrix@C=0pt@R=0pt{
&&t\minus 2&&t\minus 2&&t\minus 1&&t\minus 2&&t\minus 1&&t\minus 1\\
&2t\minus 4&&2t\minus 4&&2t\minus 3&&2t\minus 3&&2t\minus 3&&2t\minus 2\\
&2t\minus 4&3t\minus 6&2t\minus 3&3t\minus 5&2t\minus 4&3t\minus 5&2t\minus 3&3t\minus 4&2t\minus 2&3t\minus 4&2t\minus 3&3t\minus 3\\
&2t\minus 4&t\minus 1&2t\minus 4&t\minus 2&2t\minus 3&t\minus 2&2t\minus 3&t\minus 1&2t\minus 3&t\minus 1&2t\minus 2&t\minus 2\\
&&t\minus 2&&t\minus 2&&t\minus 1&&t\minus 2&&t\minus 1&&t\minus 1\\
\ar@{-}[0,13]&&&&&&&&&&&&&&&&&\\
&-1&0&1&2&3&4&5&6&7&8&9&10\\
}
\]
}}Furthermore after column 10 there are no more zeroes.
\end{lemma}
\begin{proof}
Proceed by induction.  The $t=3$ case can be done by inspection:
{\tiny{
\[
\xymatrix@C=0pt@R=0pt{
.&&.&&.&&1&&0&&1&&1&&1&&1&&2&&1&&2&&2 \\
&.&&.&&1&&1&&1&&2&&2&&2&&3&&3&&3&&4&\\
.&.&.&1&1&0&1&1&2&2&2&1&3&2&3&3&4&2&4&3&5&4&5&3&6\\
R&.&1&.&0&1&0&1&1&1&1&2&0&2&2&2&1&3&1&3&2&3&2&4&1\\
.&&.&&.&&1&&0&&1&&1&&1&&1&&2&&1&&2&&2\\
&&&&&&&&&&&&&&&&&&&&&&&\\
&&{}\save[]*\txt{0}\restore&&&&&&&&&&&{}\save[]*\txt{11}\restore&&&&&&&&&&&{}\save[]*\txt{22}\restore&
}
\]
}}For the induction step, since the statement in the lemma satisfies
the counting rules we just need to verify the induction at the end
point.  But by the counting rules this is trivial.
\end{proof}

\textbf{The case $m\equiv 1$.}
In this subfamily we have $m=6(b-2)+1$.
In the case $\mathbb{T}_1=E_6\leq SL(2,\C{})$ there is nothing to prove since all CM modules are special.
\begin{lemma}
For $\mathbb{T}_{6(b-2)+1}$ with $b\geq 3$ the specials are precisely those CM modules circled below:
{\tiny{
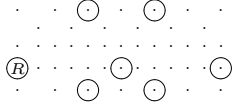
\[
\begin{array}{c}
\xymatrix@C=-2pt@R=0pt{
.&&.&&{{{}\drop\xycircle<4pt,4pt>{}.}}&&.&&{{{}\drop\xycircle<4pt,4pt>{}.}}&&.&&. \\
&.&&.&&.&&.&&.&&.\\
.&.&.&.&.&.&.&.&.&.&.&.&.\\
{{}\drop\xycircle<4pt,4pt>{}R}&.&.&.&.&.&{{{}\drop\xycircle<4pt,4pt>{}.}}&.&.&.&.&.&{{{}\drop\xycircle<4pt,4pt>{}.}}\\
.&&.&&{{{}\drop\xycircle<4pt,4pt>{}.}}&&.&&{{{}\drop\xycircle<4pt,4pt>{}.}}&&.&&.}
\end{array}
\]}}
\end{lemma}
\begin{proof}
As in Example~\ref{D52(i)} we start counting from $\tau^{-1}R$.  $R$ is a distance of $12(b-2)$ away from $\tau^{-1}R$ and so by Lemma~\ref{freeT} the calculation reaches $R$ as
{\tiny{
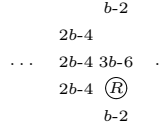
\[
\hdots
\begin{array}{c}
\xymatrix@C=-4pt@R=0pt{
&b\minus 2\\
2b\minus 4&\\
2b\minus 4&3b\minus 6\\
2b\minus 4&{{}\drop\xycircle<4pt,4pt>{}R}\\
&b\minus 2}
\end{array}.
\]
}}But here we are counting in $\underline{\CM}(R)$ and so we treat $R$ as zero.  Thus the calculation now ends as
{\tiny{
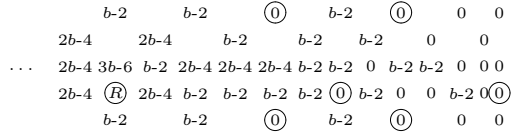
\[
\hdots
\begin{array}{c}
\xymatrix@C=-4pt@R=0pt{
&b\minus 2&&b\minus 2&&{{{}\drop\xycircle<4pt,4pt>{}0}}&&b\minus 2&&{{{}\drop\xycircle<4pt,4pt>{}0}}&&0&&0 \\
2b\minus 4&&2b\minus 4&&b\minus 2&&b\minus 2&&b\minus 2&&0&&0\\
2b\minus 4&3b\minus 6&b\minus 2&2b\minus 4&2b\minus 4&2b\minus 4&b\minus 2&b\minus 2&0&b\minus 2&b\minus 2&0&0&0\\
2b\minus 4&{{}\drop\xycircle<4pt,4pt>{}R}&2b\minus 4&b\minus 2&b\minus 2&b\minus 2&b\minus 2&{{{}\drop\xycircle<4pt,4pt>{}0}}&b\minus 2&0&0&b\minus 2&0&{{{}\drop\xycircle<4pt,4pt>{}0}}\\
&b\minus 2&&b\minus 2&&{{{}\drop\xycircle<4pt,4pt>{}0}}&&b\minus 2&&{{{}\drop\xycircle<4pt,4pt>{}0}}&&0&&0}
\end{array}
\]
}}The numbers obtained in the above picture are now added back to the numbers in the initial free expansion from $\tau^{-1}R$ (just like in Example~\ref{D52(i)}) and the modules that still have number zero are precisely the specials.
\end{proof}

\textbf{The case $m\equiv 5$.}
In this subfamily we have $m=6(b-2)+5$.
\begin{lemma}\label{T5}
For the group $\mathbb{T}_{5}$ the following calculation determines the specials:
{\tiny
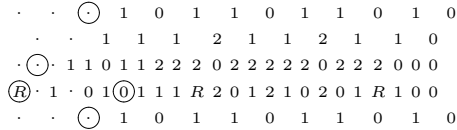
{
\[
\xymatrix@C=-3pt@R=0pt{
.&&.&&{{}\drop\xycircle<4pt,4pt>{}.}&&1&&0&&1&&1&&0&&1&&1&&0&&1&&0\\
&.&&.&&1&&1&&1&&2&&1&&1&&2&&1&&1&&0\\
.&{{}\drop\xycircle<4pt,4pt>{}.}&.&1&1&0&1&1&2&2&2&0&2&2&2&2&2&0&2&2&2&0&0&0\\
{{}\drop\xycircle<4pt,4pt>{}R}&.&1&.&0&1&{{}\drop\xycircle<4pt,4pt>{}0}&1&1&1&R&2&0&1&2&1&0&2&0&1&R&1&0&0\\
.&&.&&{{}\drop\xycircle<4pt,4pt>{}.}&&1&&0&&1&&1&&0&&1&&1&&0&&1&&0
}
\]}}
\end{lemma}
\begin{lemma}\label{T5b}
For $\mathbb{T}_{6(b-2)+5}$ with $b\geq 3$ the specials are precisely those CM modules circled below:
{\tiny{
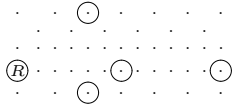
\[
\begin{array}{c}
\xymatrix@C=-2pt@R=0pt{
.&&.&&{{{}\drop\xycircle<4pt,4pt>{}.}}&&.&&.&&.&&. \\
&.&&.&&.&&.&&.&&.\\
.&.&.&.&.&.&.&.&.&.&.&.&.\\
{{}\drop\xycircle<4pt,4pt>{}R}&.&.&.&.&.&{{{}\drop\xycircle<4pt,4pt>{}.}}&.&.&.&.&.&{{{}\drop\xycircle<4pt,4pt>{}.}}\\
.&&.&&{{{}\drop\xycircle<4pt,4pt>{}.}}&&.&&.&&.&&.}
\end{array}
\]}}
\end{lemma}
\begin{proof}
$R$ is now a distance of $12(b-2)+8$ away from $\tau^{-1}R$ and so by Lemma~\ref{freeT}
{\tiny{
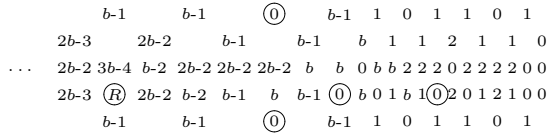
\[
\hdots
\begin{array}{c}
\xymatrix@C=-4pt@R=0pt{
&b\minus 1&&b\minus 1&&{{{}\drop\xycircle<4pt,4pt>{}0}}&&b\minus 1&&1&&0&&1&&1&&0&&1 \\
2b\minus 3&&2b\minus 2&&b\minus 1&&b\minus 1&&b&&1&&1&&2&&1&&1&&0\\
2b\minus 2&3b\minus 4&b\minus 2&2b\minus 2&2b\minus 2&2b\minus 2&b&b&0&b&b&2&2&2&0&2&2&2&2&0&0\\
2b\minus 3&{{}\drop\xycircle<4pt,4pt>{}R}&2b\minus 2&b\minus 2&b\minus 1&b&b\minus 1&{{{}\drop\xycircle<4pt,4pt>{}0}}&b&0&1&b&1&{{{}\drop\xycircle<4pt,4pt>{}0}}&2&0&1&2&1&0&0\\
&b\minus 1&&b\minus 1&&{{{}\drop\xycircle<4pt,4pt>{}0}}&&b\minus 1&&1&&0&&1&&1&&0&&1}
\end{array}
\]}}
\end{proof}

\textbf{The case $m\equiv 3$.}
In this subfamily we have $m=6(b-2)+3$.
\begin{lemma}
For the group $\mathbb{T}_{3}$ (i.e. $b=2$) the following calculation determines the specials:
{\tiny
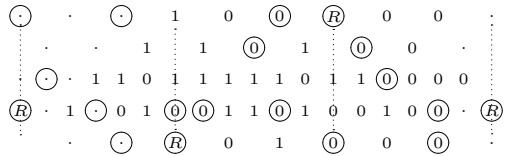
{
\[
\xymatrix@C=0pt@R=2pt{
{{}\drop\xycircle<4pt,4pt>{}.}&&.&&{{}\drop\xycircle<4pt,4pt>{}.}&&1&&0&&{{{}\drop\xycircle<4pt,4pt>{}0}}&&{{}\drop\xycircle<4pt,4pt>{}R}&&0&&0&&. \\
&.&&.&&1&&1&&{{{}\drop\xycircle<4pt,4pt>{}0}}&&1&&{{{}\drop\xycircle<4pt,4pt>{}0}}&&0&&.&\\
.&{{}\drop\xycircle<4pt,4pt>{}.}&.&1&1&0&1&1&1&1&1&0&1&1&{{{}\drop\xycircle<4pt,4pt>{}0}}&0&0&0 \\
{{}\drop\xycircle<4pt,4pt>{}R}&.&1&{{}\drop\xycircle<4pt,4pt>{}.}&0&1&{{{}\drop\xycircle<4pt,4pt>{}0}}&{{{}\drop\xycircle<4pt,4pt>{}0}}&1&1&{{{}\drop\xycircle<4pt,4pt>{}0}}&1&0&0&1&0&{{{}\drop\xycircle<4pt,4pt>{}0}}&.&{{}\drop\xycircle<4pt,4pt>{}R}&\\
\ar@{.}[-4,0].&&.&&{{}\drop\xycircle<4pt,4pt>{}.}&&\ar@{.}[-4,0]{{}\drop\xycircle<4pt,4pt>{}R}&&0&&1&&\ar@{.}[-4,0]{{{}\drop\xycircle<4pt,4pt>{}0}}&&0&&{{{}\drop\xycircle<4pt,4pt>{}0}}&&.\ar@{.}[-4,0]&}
\]}}
\end{lemma}

\begin{lemma}
For $\mathbb{T}_{6(b-2)+3}$ with $b\geq 3$ the specials are precisely those CM modules circled below:
{\tiny{
\[
\begin{array}{c}
\xymatrix@C=-2pt@R=0pt{
.&&.&&{{{}\drop\xycircle<4pt,4pt>{}.}}&&.&&.&&.&&. \\
&.&&.&&.&&.&&.&&.\\
.&.&.&.&.&.&.&.&.&.&.&.&.\\
{{}\drop\xycircle<4pt,4pt>{}R}&.&.&.&.&.&{{{}\drop\xycircle<4pt,4pt>{}.}}&.&.&.&.&.&{{{}\drop\xycircle<4pt,4pt>{}.}}\\
.&&.&&{{{}\drop\xycircle<4pt,4pt>{}.}}&&.&&{{{}\drop\xycircle<4pt,4pt>{}.}}&&.&&.}
\end{array}
\]}}
\end{lemma}
\begin{proof}
$R$ is now a distance of $12(b-2)+4$ away from $\tau^{-1}R$ and so by Lemma~\ref{freeT}
{\tiny{
\[
\hdots
\begin{array}{c}
\xymatrix@C=-4pt@R=0pt{
&b\minus 1&&b\minus 2&&{{{}\drop\xycircle<4pt,4pt>{}0}}&&b\minus 1&&{{{}\drop\xycircle<4pt,4pt>{}0}}&&0&&1&&0&&0&&1 \\
2b\minus 3&&2b\minus 3&&b\minus 2&&b\minus 1&&b\minus 1&&0&&1&&1&&0&&1\\
2b\minus 4&3b\minus 5&2b\minus 3&2b\minus 3&b\minus 1&2b\minus 3&b\minus 2&b\minus 1&b\minus 1&b\minus 1&1&1&0&1&1&1&1&1&0&1\\
2b\minus 3&b\minus 2&b\minus 2&b\minus 1&2b\minus 3&{{{}\drop\xycircle<4pt,4pt>{}0}}&b\minus 1&b\minus 2&0&1&b\minus 1&0&1&0&0&1&1&0&1&0\\
&{{}\drop\xycircle<4pt,4pt>{}R}&&b\minus 2&&b\minus 1&&{{{}\drop\xycircle<4pt,4pt>{}0}}&&0&&b\minus 1&\ar@{.}@<1ex>[-4,0]&{{{}\drop\xycircle<4pt,4pt>{}0}}&&0&&1&\ar@{.}@<1ex>[-4,0]&0}
\end{array}
\hdots
\begin{array}{c}
\xymatrix@C=-4pt@R=2pt{
&{{}\drop\xycircle<4pt,4pt>{}R}&&0\\
1&&0&&0\\
0&1&1&0&0\\
1&0&0&1&0&0\\
&0&&0}
\end{array}
\]
}}where in the above we have circled the specials, taking into account the twist.
\end{proof}

\section{Type $\mathbb{O}$}
Here we have $\mathbb{O}_m$ with $m\equiv 1,5,7$ or $11$ mod $12$.
By \cite{AR_McKayGraphs} the AR quiver of any
$\C{}[[x,y]]^{\mathbb{O}_m}$ for such $m$ is {\tiny{
\[
\begin{array}{rcl}
\begin{array}{c}
\xymatrix@C=8pt@R=8pt{
\bullet\ar[1,1]&&\bullet\ar[1,1]&&\bullet\\
&\bullet\ar[1,1]\ar[-1,1]&&\bullet\ar[1,1]\ar[-1,1]\\
\bullet\ar[1,1]\ar[-1,1]&&\bullet\ar[1,1]\ar[-1,1]&&\bullet\\
\bullet\ar[0,1]&\bullet\ar[1,1]\ar[-1,1]\ar[0,1]&\bullet\ar[0,1]&\bullet\ar[1,1]\ar[-1,1]\ar[0,1]&\bullet\\
\bullet\ar[1,1]\ar[-1,1]&&\bullet\ar[1,1]\ar[-1,1]&&\bullet\\
&\bullet\ar[1,1]\ar[-1,1]&&\bullet\ar[1,1]\ar[-1,1]\\
R\ar[-1,1]\ar@{.}[-6,0]&&\bullet\ar[-1,1]\ar@{.}[-6,0]&&\bullet}
\end{array}
&\hdots&
\begin{array}{c}
\xymatrix@C=8pt@R=8pt{
\bullet\ar[1,1]&&\bullet\\
&\bullet\ar[1,1]\ar[-1,1]&&\\
\bullet\ar[1,1]\ar[-1,1]&&\bullet\\
\bullet\ar[0,1]&\bullet\ar[1,1]\ar[-1,1]\ar[0,1]&\bullet\\
\bullet\ar[1,1]\ar[-1,1]&&\bullet\\
&\bullet\ar[1,1]\ar[-1,1]&&\\
\bullet\ar[-1,1]&&R}
\end{array}
\end{array}
\]
}}where there are precisely $m$ repetitions of the original $\tilde{E}_7$ shown
in dotted lines.  The left and right hand sides of the picture are identified, and
there is no twist in this AR quiver.  Because the AR quiver is the
same in all subfamilies and there is no twist, proofs become easier
than in type $\mathbb{T}$.

\begin{lemma}\label{freeO}
Consider the free expansion from $\tau^{-1}R$ and choose $t\geq 3$. Then
between columns $24(t-2)-1$ and $24(t-2)+22$ the free expansion
looks like {\tiny{
\[
\xymatrix@C=-3.2pt@R=-3pt{
&&&&&&&&&&&&&&&&&&&&&&&&&&&&\\
&&t\minus 2&&t\minus 2&&t\minus 2&&t\minus 1&&t\minus 2&&t\minus 2&&t\minus 1&&t\minus 1&&t\minus 2&&t\minus 1&&t\minus 1&&t\minus 1\\
&2t\minus 4&&2t\minus 4&&2t\minus 4&&2t\minus 3&&2t\minus 3&&2t\minus 4&&2t\minus 3&&2t\minus 2&&2t\minus 3&&2t\minus 3&&2t\minus 2&&2t\minus 2&\\
&&3t\minus 6&&3t\minus 6&&3t\minus 5&&3t\minus 5&&3t\minus 5&&3t\minus 5&&3t\minus 4&&3t\minus 4&&3t\minus 4&&3t\minus 4&&3t\minus 3&&3t\minus 3\\
&4t\minus 8&2t\minus 4&4t\minus 8&2t\minus 4&4t\minus 7&2t\minus 3&4t\minus 7&2t\minus 4&4t\minus 7&2t\minus 3&4t\minus 6&2t\minus 3&4t\minus 6&2t\minus 3&4t\minus 6&2t\minus 3&4t\minus 5&2t\minus 2&4t\minus 5&2t\minus 3&4t\minus 5&2t\minus 2&4t\minus 4&2t\minus 2&\\
&&3t\minus 6&&3t\minus 5&&3t\minus 6&&3t\minus 5&&3t\minus 5&&3t\minus 4&&3t\minus 5&&3t\minus 4&&3t\minus 4&&3t\minus 3&&3t\minus 4&&3t\minus 3&\\
&2t\minus 4&&2t\minus 3&&2t\minus 4&&2t\minus 4&&2t\minus 3&&2t\minus 3&&2t\minus 3&&2t\minus 3&&2t\minus 3&&2t\minus 2&&2t\minus 2&&2t\minus 3&\\
&&t\minus 1&&t\minus 2&&t\minus 2&&t\minus 2&&t\minus 1&&t\minus 2&&t\minus 1&&t\minus 2&&t\minus 1&&t\minus 1&&t\minus 1&&t\minus 2\\
\ar@{-}[0,25]&&&&&&&&&&&&&&&&&&&&&&&&&&&&\\
&-1&0&1&2&3&4&5&6&7&8&9&10&11&12&13&14&15&16&17&18&19&20&21&22\\
}
\]
}}Furthermore after column 23 there are no more zeroes.
\end{lemma}
\begin{proof}
Proceed by induction.  The $t=3$ case can be done by inspection:
{\tiny{
\[
\xymatrix@C=-5pt@R=-3pt{
.&&.&&.&&.&&1&&0&&0&&1&&1&&0&&1&&1&&1&&1&&1&&1&&2&&1&&1&&2&&2&&1&&2&&2&&2\\
&.&&.&&.&&1&&1&&0&&1&&2&&1&&1&&2&&2&&2&&2&&2&&3&&3&&2&&3&&4&&3&&3&&4&&4&\\
.&&.&&.&&1&&1&&1&&1&&2&&2&&2&&2&&3&&3&&3&&3&&4&&4&&4&&4&&5&&5&&5&&5&&6&&6\\
.&.&.&.&.&1&1&1&0&1&1&2&1&2&1&2&1&3&2&3&1&3&2&4&2&4&2&4&2&5&3&5&2&5&3&6&3&6&3&6&3&7&4&7&3&7&4&8&4\\
.&&.&&1&&0&&1&&1&&2&&1&&2&&2&&3&&2&&3&&3&&4&&3&&4&&4&&5&&4&&5&&5&&6&&5&&6\\
&.&&1&&0&&0&&1&&1&&1&&1&&1&&2&&2&&1&&2&&3&&2&&2&&3&&3&&3&&3&&3&&4&&4&&3&\\
R&&1&&0&&0&&0&&1&&0&&1&&0&&1&&1&&1&&0&&2&&1&&1&&1&&2&&1&&2&&1&&2&&2&&2&&1\\
&&&&&&&&&&&&&&&&&&&&&&&&&&&&&&&&&&&&&&&&&&&&&&&&&&\\
&&0&&&&&&&&&&&&&&&&&&&&&&&{}\drop{23}{}&&&&&&&&&&&&&&&&&&&&&&&{}\drop{46}{}
}
\]
}}For the induction step, since the statement in the lemma satisfies
the counting rules we just need to verify the induction at the end
point.  But by the counting rules this is trivial.
\end{proof}

\textbf{The case $m\equiv 1$.} In this subfamily we have
$m=12(b-2)+1$. In the case $\mathbb{O}_{1}=E_7\leq SL(2,\C{})$ (i.e.
$b=2$) there is nothing to prove since all CM modules are
special.  Thus we just need to deal with the case $b\geq 3$.
\begin{lemma}
For $\mathbb{O}_{12(b-2)+1}$ with $b\geq 3$ the specials are precisely those CM modules circled below
{\tiny{
\[
\xymatrix@C=-5pt@R=-3pt{
&&&&&&&&&&&&&&&&&&&&&&&&&\\
.&&.&&.&&{{{}\drop\xycircle<4pt,4pt>{}.}}&&1&&0&&{{{}\drop\xycircle<4pt,4pt>{}0}}&&1&&1&&{{{}\drop\xycircle<4pt,4pt>{}0}}&&1&&1&&1&\\
&.&&.&&.&&1&&1&&0&&1&&2&&1&&1&&2&&2&&2\\
.&&.&&.&&1&&1&&1&&1&&2&&2&&2&&2&&3&&3&\\
&.&.&.&.&1&1&1&0&1&1&2&1&2&1&2&1&3&2&3&1&3&2&4&2&4\\
.&&.&&1&&0&&1&&1&&2&&1&&2&&2&&3&&2&&3&\\
&.&&1&&0&&0&&1&&1&&1&&1&&1&&2&&2&&1&&2&\\
{{{}\drop\xycircle<4pt,4pt>{}R}}&&1&&0&&0&&{{{}\drop\xycircle<4pt,4pt>{}0}}&&1&&{{{}\drop\xycircle<4pt,4pt>{}0}}&&1&&{{{}\drop\xycircle<4pt,4pt>{}0}}&&1&&1&&1&&{{{}\drop\xycircle<4pt,4pt>{}0}}&\\
&&&&&&&&&&&&&&&&&&&&&&&&&}
\]
}}where the numbers are just the free expansion from $\tau^{-1}R$.
\end{lemma}
\begin{proof}
Continuing the calculation in the statement, by Lemma~\ref{freeO}
there are no zeroes in the free expansion after the right hand side.
Now the free expansion continues until it reaches $R$, which is a
distance of $24(b-2)$ away from $\tau^{-1}R$.  Thus by Lemma~\ref{freeO}
the calculation finishes as {\tiny{
\[
...
\begin{array}{c}
\xymatrix@C=-5pt@R=-3pt{
&&&&&&&&&&&&&&&&&&&&&&&&&\\
&b\minus 2&&b\minus 2&&b\minus 2&&{{{}\drop\xycircle<4pt,4pt>{}0}}&&b\minus 2&&b\minus 2&&{{{}\drop\xycircle<4pt,4pt>{}0}}&&0&&b\minus 2&&{{{}\drop\xycircle<4pt,4pt>{}0}}&&&\\
2b\minus 4&&2b\minus 4&&2b\minus 4&&b\minus 2&&b\minus 2&&2b\minus 4&&b\minus 2&&0&&b\minus 2&&b\minus 2&&0&&&\\
&3b\minus 6&&3b\minus 6&&2b\minus 4&&2b\minus 4&&2b\minus 4&&2b\minus 4&&b\minus 2&&b\minus 2&&b\minus 2&&b\minus 2&&0\\
4b\minus 8&2b\minus 4&4b\minus 8&2b\minus 4&3b\minus 6&b\minus 2&3b\minus 6&2b\minus 4&3b\minus 6&b\minus 2&2b\minus 4&b\minus 2&2b\minus 4&b\minus 2&2b\minus 4&b\minus 2&b\minus 2&0&b\minus 2&b\minus 2&b\minus 2&0&0&\\
&3b\minus 6&&2b\minus 4&&3b\minus 6&&2b\minus 4&&2b\minus 4&&b\minus 2&&2b\minus 4&&b\minus 2&&b\minus 2&&0&&b\minus 2&&0&\\
2b\minus 4&&b\minus 2&&2b\minus 4&&2b\minus 4&&b\minus 2&&b\minus 2&&b\minus 2&&b\minus 2&&b\minus 2&&0&&0&&b\minus 2&&0&\\
&{{{}\drop\xycircle<4pt,4pt>{}R}}&&b\minus 2&&b\minus 2&&b\minus 2&&{{{}\drop\xycircle<4pt,4pt>{}0}}&&b\minus 2&&{{{}\drop\xycircle<4pt,4pt>{}0}}&&b\minus 2&&{{{}\drop\xycircle<4pt,4pt>{}0}}&&0&&0&&b\minus 2&&{{{}\drop\xycircle<4pt,4pt>{}0}}\\
&&&&&&&&&&&&&&&&&&&&&&&&&}
\end{array}
\]}}
\end{proof}

\textbf{The case $m\equiv 5$.}
In this subfamily we have $m=12(b-2)+5$.

\begin{lemma}
For the group $\mathbb{O}_{5}$ (i.e. $b=2$) the following calculation determines the specials:
{\tiny{
\[
\xymatrix@C=-4pt@R=-3pt{
.&&{{}\drop\xycircle<4pt,4pt>{}.}&&.&&{{}\drop\xycircle<4pt,4pt>{}.}&&1&&0&&0&&1&&0&&0&&1&&0&\\
&.&&.&&.&&1&&1&&0&&1&&1&&0&&1&&1&&0&\\
.&&.&&.&&1&&1&&1&&1&&1&&1&&1&&1&&1&&0\\
.&.&.&.&{{}\drop\xycircle<4pt,4pt>{}.}&1&1&1&0&1&1&2&1&1&0&1&1&2&1&1&0&1&1&1&0&0\\
.&&.&&1&&0&&1&&1&&1&&1&&1&&1&&1&&0&&1&&0\\
&{{}\drop\xycircle<4pt,4pt>{}.}&&1&&0&&{{}\drop\xycircle<4pt,4pt>{}0}&&1&&0&&1&&1&&0&&1&&0&&0&&1&&0\\
{{{}\drop\xycircle<4pt,4pt>{}R}}&&1&&0&&0&&{{}\drop\xycircle<4pt,4pt>{}0}&&R&&0&&1&&0&&0&&R&&0&&0&&1&&0}
\]}}
\end{lemma}
Thus we need only deal with $b\geq 3$:
\begin{lemma}
For $\mathbb{O}_{12(b-2)+5}$ with $b\geq 3$ the specials are precisely those CM modules circled below.
{\tiny{
\[
\xymatrix@C=-5pt@R=-3pt{
&&&&&&&&&&&&&&&&&&&&&&&&&\\
.&&.&&.&&{{{}\drop\xycircle<4pt,4pt>{}.}}&&1&&0&&{{{}\drop\xycircle<4pt,4pt>{}0}}&&1&&1&&{{{}\drop\xycircle<4pt,4pt>{}0}}&&1&&1&&1&\\
&.&&.&&.&&1&&1&&0&&1&&2&&1&&1&&2&&2&&2\\
.&&.&&.&&1&&1&&1&&1&&2&&2&&2&&2&&3&&3&\\
&.&.&.&.&1&1&1&0&1&1&2&1&2&1&2&1&3&2&3&1&3&2&4&2&4\\
.&&.&&1&&0&&1&&1&&2&&1&&2&&2&&3&&2&&3&\\
&.&&1&&0&&0&&1&&1&&1&&1&&1&&2&&2&&1&&2&\\
{{{}\drop\xycircle<4pt,4pt>{}R}}&&1&&0&&0&&{{{}\drop\xycircle<4pt,4pt>{}0}}&&1&&{{{}\drop\xycircle<4pt,4pt>{}0}}&&1&&0&&1&&1&&1&&{{{}\drop\xycircle<4pt,4pt>{}0}}&\\
&&&&&&&&&&&&&&&&&&&&&&&&&}
\]}}
\end{lemma}
\begin{proof}
$R$ is now a distance of $24(b-2)+8$ away from $\tau^{-1}R$, thus by
Lemma~\ref{freeO} we have {\tiny{
\[
\hdots
\begin{array}{c}
\xymatrix@C=-5.5pt@R=-3pt{
&b\minus 2&&b\minus 2&&b\minus 1&&{{{}\drop\xycircle<4pt,4pt>{}0}}&&b\minus 2&&b\minus 1&&{{{}\drop\xycircle<4pt,4pt>{}0}}&&0&&b\minus 1&&{{{}\drop\xycircle<4pt,4pt>{}0}}&&0&&1&&0&&0&&1&&0\\
2b\minus 3&&2b\minus 4&&2b\minus 3&&b\minus 1&&b\minus 2&&2b\minus 3&&b\minus 1&&0&&b\minus 1&&b\minus 1&&0&&1&&1&&0&&1&&1&\\
&3b\minus 5&&3b\minus 5&&2b\minus 3&&2b\minus 3&&2b\minus 3&&2b\minus 3&&b\minus 1&&b\minus 1&&b\minus 1&&b\minus 1&&1&&1&&1&&1&&1&&1\\
4b\minus 7&2b\minus 3&4b\minus 6&2b\minus 3&3b\minus 5&b\minus 2&3b\minus 5&2b\minus 3&3b\minus 4&b\minus 1&2b\minus 3&b\minus 2&2b\minus 3&b\minus 1&2b\minus 2&b\minus 1&b\minus 1&0&b\minus 1&b\minus 1&b&1&1&0&1&1&2&1&1&0&1&1\\
&3b\minus 5&&2b\minus 3&&3b\minus 5&&2b\minus 3&&2b\minus 3&&b\minus 1&&2b\minus 3&&b\minus 1&&b\minus 1&&1&&b\minus 1&&1&&1&&1&&1&&1\\
2b\minus 3&&b\minus 2&&2b\minus 3&&2b\minus 3&&b\minus 2&&b\minus 1&&b\minus 1&&b\minus 2&&b\minus 1&&1&&0&&b\minus 1&&1&&0&&1&&1\\
&{{{}\drop\xycircle<4pt,4pt>{}R}}&&b\minus 2&&b\minus 1&&b\minus 2&&{{{}\drop\xycircle<4pt,4pt>{}0}}&&b\minus 1&&{{{}\drop\xycircle<4pt,4pt>{}0}}&&b\minus 2&&1&&0&&0&&b\minus 1&\ar@{.}@<1ex>[-6,0]&{{{}\drop\xycircle<4pt,4pt>{}0}}&&0&&1&\ar@{.}@<1ex>[-6,0]&0}
\end{array}
\hdots
\]
}}Now the calculation continues by repeating the segment within the dotted lines until it reaches $R$ as follows:
{\tiny{
\[
\hdots
\begin{array}{c}
\xymatrix@C=-5pt@R=-3pt{
&1&&0\\
1&&1&&0\\
&1&&1&&0\\
1&0&1&1&1&0&0\\
&1&&0&&1&&0\\
1&&0&&0&&1&&0\\
&{{{}\drop\xycircle<4pt,4pt>{}R}}&&0&&0&&1&&{{{}\drop\xycircle<4pt,4pt>{}0}}}
\end{array}
\]}}
\end{proof}

\textbf{The case $m\equiv 7$.}
In this subfamily we have $m=12(b-2)+7$.

\begin{lemma}
For the group $\mathbb{O}_{7}$ (i.e. $b=2$) the following calculation determines the specials:
{\tiny{
\[
\xymatrix@C=-5pt@R=-3pt{
.&&.&&.&&{{}\drop\xycircle<4pt,4pt>{}.}&&1&&0&&{{}\drop\xycircle<4pt,4pt>{}0}&&1&&1&&0&&0&&1&&1&&0&&0&&1&&1&&0&&0&&1&&0&&0\\
&.&&.&&{{}\drop\xycircle<4pt,4pt>{}.}&&1&&1&&0&&1&&2&&1&&0&&1&&2&&1&&0&&1&&2&&0&&0&&1&&1&&0\\
.&&.&&.&&1&&1&&1&&1&&2&&2&&1&&1&&2&&2&&1&&1&&2&&1&&0&&1&&1&&0\\
.&.&.&.&.&1&1&1&0&1&1&2&1&2&1&2&1&2&1&2&1&2&1&2&1&2&1&2&1&2&1&1&0&1&1&1&0&1&0&0\\
.&&.&&1&&0&&1&&1&&2&&1&&1&&2&&2&&1&&1&&2&&2&&0&&1&&1&&1&&0&&\\
&{{}\drop\xycircle<4pt,4pt>{}.}&&1&&0&&0&&1&&1&&1&&0&&1&&2&&1&&0&&1&&2&&0&&0&&1&&1&&0\\
{{{}\drop\xycircle<4pt,4pt>{}R}}&&1&&0&&0&&{{}\drop\xycircle<4pt,4pt>{}0}&&1&&0&&R&&0&&1&&1&&0&&0&&1&&R&&0&&0&&1&&0}
\]}}
\end{lemma}
\begin{lemma}
For $\mathbb{O}_{12(b-2)+7}$ with $b\geq 3$ the specials are precisely those CM modules circled below.
{\tiny{
\[
\xymatrix@C=-5pt@R=-3pt{
&&&&&&&&&&&&&&&&&&&&&&&&&\\
.&&.&&.&&{{{}\drop\xycircle<4pt,4pt>{}.}}&&1&&0&&{{{}\drop\xycircle<4pt,4pt>{}0}}&&1&&1&&0&&1&&1&&1&\\
&.&&.&&.&&1&&1&&0&&1&&2&&1&&1&&2&&2&&2\\
.&&.&&.&&1&&1&&1&&1&&2&&2&&2&&2&&3&&3&\\
&.&.&.&.&1&1&1&0&1&1&2&1&2&1&2&1&3&2&3&1&3&2&4&2&4\\
.&&.&&1&&0&&1&&1&&2&&1&&2&&2&&3&&2&&3&\\
&.&&1&&0&&0&&1&&1&&1&&1&&1&&2&&2&&1&&2&\\
{{{}\drop\xycircle<4pt,4pt>{}R}}&&1&&0&&0&&{{{}\drop\xycircle<4pt,4pt>{}0}}&&1&&0&&1&&{{{}\drop\xycircle<4pt,4pt>{}0}}&&1&&1&&1&&{{{}\drop\xycircle<4pt,4pt>{}0}}&}
\]}}
\end{lemma}
\begin{proof}
$R$ is a distance of $24(b-2)+12$ away from $\tau^{-1}R$, thus by
Lemma~\ref{freeO} we have {\tiny{
\[
\hdots
\begin{array}{c}
\xymatrix@C=-5.5pt@R=-3pt{
&b\minus 1&&b\minus 1&&b\minus 2&&{{{}\drop\xycircle<4pt,4pt>{}0}}&&b\minus 1&&b\minus 1&&{{{}\drop\xycircle<4pt,4pt>{}0}}&&0&&b\minus 1&&1&&0&&0&&1&&1&&0&&0&&1\\
2b\minus 3&&2b\minus 2&&2b\minus 3&&b\minus 2&&b\minus 1&&2b\minus 2&&b\minus 1&&0&&b\minus 1&&b&&1&&0&&1&&2&&1&&0&&1\\
&3b\minus 4&&3b\minus 4&&2b\minus 3&&2b\minus 3&&2b\minus 2&&2b\minus 2&&b\minus 1&&b\minus 1&&b&&b&&1&&1&&2&&2&&1&&1&&2\\
4b\minus 6&2b\minus 3&4b\minus 6&2b\minus 3&3b\minus 4&b\minus 1&3b\minus 4&2b\minus 3&3b\minus 4&b\minus 1&2b\minus 2&b\minus 1&2b\minus 2&b\minus 1&2b\minus 2&b\minus 1&b&1&b&b\minus 1&b&1&2&1&2&1&2&1&2&1&2&1&2&1\\
&3b\minus 5&&2b\minus 3&&3b\minus 4&&2b\minus 2&&2b\minus 3&&b\minus 1&&2b\minus 2&&b&&b\minus 1&&1&&b&&2&&1&&1&&2&&2&&1\\
2b\minus 3&&b\minus 2&&2b\minus 3&&2b\minus 2&&b\minus 1&&b\minus 2&&b\minus 1&&b&&b\minus 1&&0&&1&&b&&1&&0&&1&&2&&1\\
&{{{}\drop\xycircle<4pt,4pt>{}R}}&&b\minus 2&&b\minus 1&&b\minus 1&&{{{}\drop\xycircle<4pt,4pt>{}0}}&&b\minus 2&&1&&b\minus 1&&{{{}\drop\xycircle<4pt,4pt>{}0}}&&0&&1&&b\minus 1&\ar@{.}@<1ex>[-6,0]&{{{}\drop\xycircle<4pt,4pt>{}0}}&&0&&1&&1&\ar@{.}@<1ex>[-6,0]&0}
\end{array}
...
\]
}}The calculation continues by repeating the segment within the dotted lines until it reaches $R$ as:
{\tiny{
\[
\hdots
\begin{array}{c}
\xymatrix@C=-4pt@R=-3pt{
&0&&1&&1&&{{{}\drop\xycircle<4pt,4pt>{}0}}&&0&&1&&{{{}\drop\xycircle<4pt,4pt>{}0}}&&0\\
0&&1&&2&&0&&0&&1&&1&&0\\
&1&&2&&1&&0&&1&&1&&1&&0\\
2&1&2&1&1&0&1&1&1&0&1&1&1&0&1&1\\
&2&&0&&1&&1&&1&&0&&1&&1\\
2&&0&&0&&1&&1&&0&&0&&1\\
&{{{}\drop\xycircle<4pt,4pt>{}R}}&&0&&0&\ar@{.}@<1ex>[-6,0]&1&&{{{}\drop\xycircle<4pt,4pt>{}0}}&&0&&0&\ar@{.}@<1ex>[-6,0]&1}
\end{array}
...
\begin{array}{c}
\xymatrix@C=-5pt@R=-3pt{
&0&&0&&1&&0\\
0&&0&&1&&0\\
&0&&1&&0\\
1&1&1&0&0\\
&1&&0\\
1&&0\\
&{{{}\drop\xycircle<4pt,4pt>{}R}}}
\end{array}
\]}}
\end{proof}

\textbf{The case $m\equiv 11$.}
In this subfamily we have $m=12(b-2)+11$.
\begin{lemma}
For the group $\mathbb{O}_{11}$ (i.e. $b=2$) the following calculation determines the specials:
{\tiny{
\[
\xymatrix@C=-5pt@R=-3pt{
.&&.&&.&&{{}\drop\xycircle<4pt,4pt>{}.}&&1&&0&&{{}\drop\xycircle<4pt,4pt>{}0}&&1&&1&&0&&1&&1&&1&&1&&0&&1&&2&&0&&0&&2&&1&&0&&1&&1&&1&&0&&0&&1&&0&&0&&0&&1&&0&&0&&0&&1&&0\\
&.&&.&&.&&1&&1&&0&&1&&2&&1&&1&&2&&2&&2&&1&&1&&3&&2&&0&&2&&3&&1&&1&&2&&2&&0&&0&&1&&1&&0&&0&&1&&1&&0&&0&&1&&0&&&\\
.&&.&&.&&1&&1&&1&&1&&2&&2&&2&&2&&3&&3&&2&&2&&3&&3&&2&&2&&3&&3&&2&&2&&3&&1&&0&&1&&1&&1&&0&&1&&1&&1&&0&&1&&0&&\\
.&.&.&.&.&1&1&1&0&1&1&2&1&2&1&2&1&3&2&3&1&3&2&4&2&3&1&3&2&4&2&3&1&3&2&4&2&3&1&3&2&4&2&3&1&3&2&2&0&1&1&1&0&1&1&1&0&1&1&1&0&1&1&1&0&1&1&1&0&0&&&\\
.&&.&&1&&0&&1&&1&&2&&1&&2&&2&&3&&2&&2&&3&&3&&2&&2&&3&&3&&2&&2&&3&&3&&0&&2&&1&&1&&0&&1&&1&&1&&0&&1&&1&&0&&&&\\
&{{}\drop\xycircle<4pt,4pt>{}.}&&1&&0&&0&&1&&1&&1&&1&&1&&2&&2&&0&&2&&3&&1&&1&&2&&2&&2&&1&&1&&3&&0&&0&&2&&1&&0&&0&&1&&1&&0&&0&&1&&0&&&&&\\
{{{}\drop\xycircle<4pt,4pt>{}R}}&&1&&0&&0&&{{}\drop\xycircle<4pt,4pt>{}0}&&1&&0&&1&&0&&1&&1&&R&&0&&2&&1&&0&&1&&1&&1&&1&&0&&1&&R&&0&&0&&2&&0&&0&&0&&1&&0&&0&&0&&R}
\]}}
\end{lemma}

\begin{lemma}
For $\mathbb{O}_{12(b-2)+11}$ with $b\geq 3$ the specials are precisely those CM modules circled below.
{\tiny{
\[
\xymatrix@C=-5pt@R=-3pt{
&&&&&&&&&&&&&&&&&&&&&&&&&\\
.&&.&&.&&{{{}\drop\xycircle<4pt,4pt>{}.}}&&1&&0&&{{{}\drop\xycircle<4pt,4pt>{}0}}&&1&&1&&0&&1&&1&&1&\\
&.&&.&&.&&1&&1&&0&&1&&2&&1&&1&&2&&2&&2\\
.&&.&&.&&1&&1&&1&&1&&2&&2&&2&&2&&3&&3&\\
&.&.&.&.&1&1&1&0&1&1&2&1&2&1&2&1&3&2&3&1&3&2&4&2&4\\
.&&.&&1&&0&&1&&1&&2&&1&&2&&2&&3&&2&&3&\\
&.&&1&&0&&0&&1&&1&&1&&1&&1&&2&&2&&1&&2&\\
{{{}\drop\xycircle<4pt,4pt>{}R}}&&1&&0&&0&&{{{}\drop\xycircle<4pt,4pt>{}0}}&&1&&0&&1&&0&&1&&1&&1&&{{{}\drop\xycircle<4pt,4pt>{}0}}&\\
&&&&&&&&&&&&&&&&&&&&&&&&&}
\]}}
\end{lemma}
\begin{proof}
$R$ is a distance of $24(b-2)+20$ away from $\tau^{-1}R$,  thus by
Lemma~\ref{freeO} we have {\tiny{
\[
\hdots
\begin{array}{c}
\xymatrix@C=-6.15pt@R=-3pt{
&b\minus 1&&b\minus 1&&b\minus 1&&{{{}\drop\xycircle<4pt,4pt>{}0}}&&b\minus 1&&b&&{{{}\drop\xycircle<4pt,4pt>{}0}}&&0&&b&&1&&0&&1&&1&&1&&1&&0&&1&&2&&0&&0&&2&&1&&0&&1&&1\\
2b\minus 2&&2b\minus 2&&2b\minus 2&&b\minus 1&&b\minus 1&&2b\minus 1&&b&&0&&b&&b\plus 1&&1&&1&&2&&2&&2&&1&&1&&3&&2&&0&&2&&3&&1&&1&&2\\
&3b\minus 3&&3b\minus 3&&2b\minus 2&&2b\minus 2&&2b\minus 1&&2b\minus 1&&b&&b&&b\plus 1&&b\plus 1&&2&&2&&3&&3&&2&&2&&3&&3&&2&&2&&3&&3&&2&&2&&3\\
4b\minus 5&2b\minus 2&4b\minus 4&2b\minus 2&3b\minus 3&b\minus 1&3b\minus 3&2b\minus 2&3b\minus 2&b&2b\minus 1&b\minus 1&2b\minus 1&b&2b&b&b\plus 1&1&b\plus 1&b&b\plus 2&2&3&1&3&2&4&2&3&1&3&2&4&2&3&1&3&2&4&2&3&1&3&2&4&2&3&1&3&2\\
&3b\minus 4&&2b\minus 2&&3b\minus 3&&2b\minus 1&&2b\minus 2&&b&&2b\minus 1&&b\plus 1&&b&&2&&b\plus 1&&3&&2&&2&&3&&3&&2&&2&&3&&3&&2&&2&&3&&3&&2\\
2b\minus 2&&b\minus 2&&2b\minus 2&&2b\minus 1&&b\minus 1&&b\minus 1&&b&&b&&b&&1&&1&&b\plus 1&&2&&0&&2&&3&&1&&1&&2&&2&&2&&1&&1&&3&&2\\
&{{{}\drop\xycircle<4pt,4pt>{}R}}&&b\minus 2&&b&&b\minus 1&&{{{}\drop\xycircle<4pt,4pt>{}0}}&&b\minus 1&&1&&b\minus 1&&1&&0&&1&&b&\ar@{.}@<0.5ex>[-6,0]&{{{}\drop\xycircle<4pt,4pt>{}0}}&&0&&2&&1&&0&&1&&1&&1&&1&&0&&1&&2&\ar@{.}@<0.5ex>[-6,0]&0}
\end{array}
\]
}}Now the calculation continues by repeating the segment within the dotted lines until it reaches $R$ as:
{\tiny{
\[
...
\begin{array}{c}
\xymatrix@C=-4pt@R=-3pt{
&1&&1&&1&&{{{}\drop\xycircle<4pt,4pt>{}0}}&&0&&1&&{{{}\drop\xycircle<4pt,4pt>{}0}}&&0&&0\\
1&&2&&2&&0&&0&&1&&1&&0&&0\\
&2&&3&&1&&0&&1&&1&&1&&0&&1\\
3&1&3&2&2&0&1&1&1&0&1&1&1&0&1&1&1&0\\
&3&&0&&2&&1&&1&&0&&1&&1&&1\\
3&&0&&0&&2&&1&&0&&0&&1&&1\\
&{{{}\drop\xycircle<4pt,4pt>{}R}}&&0&&0&&2&\ar@{.}@<1ex>[-6,0]&{{{}\drop\xycircle<4pt,4pt>{}0}}&&0&&0&&1&\ar@{.}@<1ex>[-6,0]&0}
\end{array}
...
\begin{array}{c}
\xymatrix@C=-5pt@R=-3pt{
&0&&0&&1&&0\\
0&&0&&1&&0\\
&0&&1&&0\\
1&1&1&0&0\\
&1&&0\\
1&&0\\
&{{{}\drop\xycircle<4pt,4pt>{}R}}}
\end{array}
\]}}
\end{proof}

\section{Type $\mathbb{I}$}
Here we have $\mathbb{I}_m$ with $m\equiv 1,7,11,13,17,19,23$ or
$29$ mod $30$.    By \cite{AR_McKayGraphs} the AR quiver of
$\C{}[[x,y]]^{\mathbb{I}_m}$ for any such $m$ is {\tiny{
\[
\begin{array}{rcl}
\begin{array}{c}
\xymatrix@C=8pt@R=8pt{
&\bullet\ar[1,1]&&\bullet\ar[1,1]\\
\bullet\ar[1,1]\ar[-1,1]&&\bullet\ar[1,1]\ar[-1,1]&&\bullet\\
\bullet\ar[0,1]&\bullet\ar[1,1]\ar[-1,1]\ar[0,1]&\bullet\ar[0,1]&\bullet\ar[1,1]\ar[-1,1]\ar[0,1]&\bullet\\
\bullet\ar[1,1]\ar[-1,1]&&\bullet\ar[1,1]\ar[-1,1]&&\bullet\\
&\bullet\ar[1,1]\ar[-1,1]&&\bullet\ar[1,1]\ar[-1,1]\\
\bullet\ar[1,1]\ar[-1,1]&&\bullet\ar[1,1]\ar[-1,1]&&\bullet\\
&\bullet\ar[1,1]\ar[-1,1]&&\bullet\ar[1,1]\ar[-1,1]\\
R\ar[-1,1]\ar@{.}[-7,0]&&\bullet\ar[-1,1]\ar@{.}[-7,0]&&\bullet}
\end{array}
&\hdots&
\begin{array}{c}
\xymatrix@C=8pt@R=8pt{
&\bullet\ar[1,1]&&\\
\bullet\ar[1,1]\ar[-1,1]&&\bullet\\
\bullet\ar[0,1]&\bullet\ar[1,1]\ar[-1,1]\ar[0,1]&\bullet\\
\bullet\ar[1,1]\ar[-1,1]&&\bullet\\
&\bullet\ar[1,1]\ar[-1,1]&&\\
\bullet\ar[1,1]\ar[-1,1]&&\bullet\\
&\bullet\ar[1,1]\ar[-1,1]&\\
\bullet\ar[-1,1]&&R}
\end{array}
\end{array}
\]
}}where there are precisely $m$ repetitions of the original $\tilde{E}_8$ shown
in dotted lines.  The left and right hand sides of the picture are identified, and
there is no twist in this AR quiver.  As in type $\mathbb{O}$ the AR
quiver is the same in all subfamilies and there is no twist, making
proofs a little easier.  As in previous sections before we proceed
case by case it is necessary to control the free expansion of the AR
quiver:
\begin{lemma}\label{freeI}
Consider the free expansion from $\tau^{-1}R$ and choose $t\geq 3$. Then
between columns $60(t-2)-1$ and $60(t-2)+58$ the free expansion
looks like {\tiny{
\[
\xymatrix@C=-4pt@R=-3pt{
&2t\minus 4&&2t\minus 4&&2t\minus 4&&2t\minus 4&&2t\minus 3&&2t\minus 4&&2t\minus 4&&2t\minus 3&&2t\minus 4&&2t\minus 3&&2t\minus 3&\\
&&4t\minus 8&&4t\minus 8&&4t\minus 8&&4t\minus 7&&4t\minus 7&&4t\minus 8&&4t\minus 7&&4t\minus 7&&4t\minus 7&&4t\minus 6&&4t\minus 7\\
&6t\minus 12&3t\minus 6&6t\minus 12&3t\minus 6&6t\minus 12&3t\minus 6&6t\minus 11&3t\minus 5&6t\minus 11&3t\minus 6&6t\minus 11&3t\minus 5&6t\minus 11&3t\minus 6&6t\minus 11&3t\minus 5&6t\minus 10&3t\minus 5&6t\minus 10&3t\minus 5&6t\minus 10&3t\minus 5\\
&&5t\minus 10&&5t\minus 10&&5t\minus 9&&5t\minus 10&&5t\minus 9&&5t\minus 9&&5t\minus 9&&5t\minus 9&&5t\minus 8&&5t\minus 9&&5t\minus 8&\\
&4t\minus 8&&4t\minus 8&&4t\minus 7&&4t\minus 8&&4t\minus 8&&4t\minus 7&&4t\minus 7&&4t\minus 7&&4t\minus 7&&4t\minus 7&&4t\minus 7\\
&&3t\minus 6&&3t\minus 5&&3t\minus 6&&3t\minus 6&&3t\minus 6&&3t\minus 5&&3t\minus 5&&3t\minus 5&&3t\minus 6&&3t\minus 5&&3t\minus 5&\\
&2t\minus 4&&2t\minus 3&&2t\minus 4&&2t\minus 4&&2t\minus 4&&2t\minus 4&&2t\minus 3&&2t\minus 3&&2t\minus 4&&2t\minus 4&&2t\minus 3&\\
&&t\minus 1&&t\minus 2&&t\minus 2&&t\minus 2&&t\minus 2&&t\minus 2&&t\minus 1&&t\minus 2&&t\minus 2&&t\minus 2&&t\minus 1\\
\ar@{-}[0,23]&&&&&&&&&&&&&&&&&&&&&&&&&&\\
&-1&0&1&2&3&4&5&6&7&8&9&10&11&12&13&14&15&16&17&18&19&20\\
}
\]}}
{\tiny{
\[
\xymatrix@C=-3pt@R=-3pt{
&2t\minus 4&&2t\minus 3&&2t\minus 3&&2t\minus 3&&2t\minus 3&&2t\minus 3&&2t\minus 3&&2t\minus 3&&2t\minus 2&&2t\minus 3&&2t\minus 3&\\
&&4t\minus 7&&4t\minus 6&&4t\minus 6&&4t\minus 6&&4t\minus 6&&4t\minus 6&&4t\minus 6&&4t\minus 5&&4t\minus 5&&4t\minus 6&&4t\minus 5\\
&6t\minus 10&3t\minus 5&6t\minus 10&3t\minus 5&6t\minus 9&3t\minus 4&6t\minus 9&3t\minus 5&6t\minus 9&3t\minus 4&6t\minus 9&3t\minus 5&6t\minus 9&3t\minus 4&6t\minus 8&3t\minus 4&6t\minus 8&3t\minus 4&6t\minus 8&3t\minus 4&6t\minus 8&3t\minus 4\\
&&5t\minus 8&&5t\minus 8&&5t\minus 8&&5t\minus 7&&5t\minus 8&&5t\minus 7&&5t\minus 7&&5t\minus 7&&5t\minus 7&&5t\minus 6&&5t\minus 7&\\
&4t\minus 6&&4t\minus 6&&4t\minus 7&&4t\minus 6&&4t\minus 6&&4t\minus 6&&4t\minus 5&&4t\minus 6&&4t\minus 6&&4t\minus 5&&4t\minus 5\\
&&3t\minus 4&&3t\minus 5&&3t\minus 5&&3t\minus 5&&3t\minus 4&&3t\minus 4&&3t\minus 4&&3t\minus 5&&3t\minus 4&&3t\minus 4&&3t\minus 3&\\
&2t\minus 3&&2t\minus 3&&2t\minus 3&&2t\minus 4&&2t\minus 3&&2t\minus 2&&2t\minus 3&&2t\minus 3&&2t\minus 3&&2t\minus 3&&2t\minus 2&\\
&&t\minus 2&&t\minus 1&&t\minus 2&&t\minus 2&&t\minus 1&&t\minus 1&&t\minus 2&&t\minus 1&&t\minus 2&&t\minus 1&&t\minus 1\\
\ar@{-}[0,23]&&&&&&&&&&&&&&&&&&&&&&&&&&\\
&21&22&23&24&25&26&27&28&29&30&31&32&33&34&35&36&37&38&39&40&41&42\\
}
\]}}
{\tiny{
\[
\xymatrix@C=-3pt@R=-3pt{
&2t\minus 2&&2t\minus 3&&2t\minus 2&&2t\minus 2&&2t\minus 3&&2t\minus 2&&2t\minus 2&&2t\minus 2\\
&&4t\minus 5&&4t\minus 5&&4t\minus 4&&4t\minus 5&&4t\minus 5&&4t\minus 4&&4t\minus 4&&4t\minus 4\\
&6t\minus 8&3t\minus 4&6t\minus 7&3t\minus 3&6t\minus 7&3t\minus 4&6t\minus 7&3t\minus 3&6t\minus 7&3t\minus 4&6t\minus 7&3t\minus 3&6t\minus 6&3t\minus 3&6t\minus 6&3t\minus 3\\
&&5t\minus 6&&5t\minus 6&&5t\minus 6&&5t\minus 6&&5t\minus 5&&5t\minus 6&&5t\minus 5&&5t\minus 5\\
&4t\minus 5&&4t\minus 5&&4t\minus 5&&4t\minus 5&&4t\minus 4&&4t\minus 4&&4t\minus 5&&4t\minus 4\\
&&3t\minus 4&&3t\minus 4&&3t\minus 4&&3t\minus 3&&3t\minus 3&&3t\minus 3&&3t\minus 4&&3t\minus 3\\
&2t\minus 2&&2t\minus 3&&2t\minus 3&&2t\minus 2&&2t\minus 2&&2t\minus 2&&2t\minus 2&&2t\minus 3\\
&&t\minus 1&&t\minus 2&&t\minus 1&&t\minus 1&&t\minus 1&&t\minus 1&&t\minus 1&&t\minus 2\\
\ar@{-}[0,17]&&&&&&&&&&&&&&&&&&&&\\
&43&44&45&46&47&48&49&50&51&52&53&54&55&56&57&58\\
}
\]}}
\end{lemma}
\begin{proof}
Proceed by induction.  The $t=3$ case can be done by inspection:
{\tiny{
\[
\xymatrix@C=-5pt@R=-4pt{
&.&&.&&.&&.&&1&&0&&0&&1&&0&&1&&1&&0&&1&&1&&1&&1&&1&&1&&1&&2&&1&&1&&2&&1&&2&&2&&1&&2&&2&&2&\\
.&&.&&.&&.&&1&&1&&0&&1&&1&&1&&2&&1&&1&&2&&2&&2&&2&&2&&2&&3&&3&&2&&3&&3&&3&&4&&3&&3&&4&&4&&4\\
.&.&.&.&.&.&.&1&1&1&0&1&1&1&0&1&1&2&1&2&1&2&1&2&1&2&1&3&2&3&1&3&2&3&1&3&2&4&2&4&2&4&2&4&2&4&2&5&3&5&2&5&3&5&2&5&3&6&3&6&3\\
.&&.&&.&&1&&0&&1&&1&&1&&1&&2&&1&&2&&2&&2&&2&&3&&2&&3&&3&&3&&3&&4&&3&&4&&4&&4&&4&&5&&4&&5&&5\\
&.&&.&&1&&0&&0&&1&&1&&1&&1&&1&&1&&2&&2&&1&&2&&2&&2&&3&&2&&2&&3&&3&&3&&3&&3&&3&&4&&4&&3&&4&\\
.&&.&&1&&0&&0&&0&&1&&1&&1&&0&&1&&1&&2&&1&&1&&1&&2&&2&&2&&1&&2&&2&&3&&2&&2&&2&&3&&3&&3&&2&&3\\
&.&&1&&0&&0&&0&&0&&1&&1&&0&&0&&1&&1&&1&&1&&0&&1&&2&&1&&1&&1&&1&&2&&2&&1&&1&&2&&2&&2&&2&&1&\\
R&&1&&0&&0&&0&&0&&0&&1&&0&&0&&0&&1&&0&&1&&0&&0&&1&&1&&0&&1&&0&&1&&1&&1&&0&&1&&1&&1&&1&&1&&0\\
&&&&&&&&&&&&&&&&&&&&&&&&&&&&&&&&&&&&&&&&&&&&&&&&\\
&&0&&&&&&&&&&&&&&&&&&&&&&&&&&&&&&&&&&&&&&&&&&&&&&&&&&&&&&&&&&{}\drop{58}{}
}
\]
\[
\xymatrix@C=-5.5pt@R=-4pt{
&2&&2&&2&&2&&3&&2&&2&&3&&2&&3&&3&&2&&3&&3&&3&&3&&3&&3&&3&&4&&3&&3&&4&&3&&4&&4&&3&&4&&4&&4&\\
4&&4&&4&&4&&5&&5&&4&&5&&5&&5&&6&&5&&5&&6&&6&&6&&6&&6&&6&&7&&7&&6&&7&&7&&7&&8&&7&&7&&8&&8&&8&\\
3&6&3&6&3&6&3&7&4&7&3&7&4&7&3&7&4&8&4&8&4&8&4&8&4&8&4&9&5&9&4&9&5&9&4&9&5&10&5&10&5&10&5&10&5&10&5&11&6&11&5&11&6&11&5&11&6&12&6&12&6\\
5&&5&&5&&6&&5&&6&&6&&6&&6&&7&&6&&7&&7&&7&&7&&8&&7&&8&&8&&8&&8&&9&&8&&9&&9&&9&&9&&10&&9&&10&&10&\\
&4&&4&&5&&4&&4&&5&&5&&5&&5&&5&&5&&6&&6&&5&&6&&6&&6&&7&&6&&6&&7&&7&&7&&7&&7&&7&&8&&8&&7&&8&\\
3&&3&&4&&3&&3&&3&&4&&4&&4&&3&&4&&4&&5&&4&&4&&4&&5&&5&&5&&4&&5&&5&&6&&5&&5&&5&&6&&6&&6&&5&&6\\
&2&&3&&2&&2&&2&&2&&3&&3&&2&&2&&3&&3&&3&&3&&2&&3&&4&&3&&3&&3&&3&&4&&3&&3&&3&&4&&4&&4&&4&&3&\\
0&&2&&1&&1&&1&&1&&1&&2&&1&&1&&1&&2&&1&&2&&1&&1&&2&&2&&1&&2&&1&&2&&2&&2&&1&&2&&2&&2&&2&&2&&1\\
&&&&&&&&&&&&&&&&&&&&&&&&&&&&&&&&&&&&&&&&&&&&&&&&&&&&&&&&&&&&&&&\\
&&&&&&&&&&&&&&&&&&&&&&&&&&&&&&&&&&&&&&&&&&&&&&&&\\
&{}\drop{59}{}&&&&&&&&&&&&&&&&&&&&&&&&&&&&&&&&&&&&&&&&&&&&&&&&&&&&&&&&&&&{}\drop{118}{}}
\]
}}For the induction step, since the statement in the lemma
satisfies the counting rules we just need to verify the induction at
the end point.  But by the counting rules this is trivial.
\end{proof}

\textbf{The case $m\equiv 1$.}
In this subfamily we have $m=30(b-2)+1$.
For the group $\mathbb{I}_1=E_8$ (i.e. $b=2$) there is nothing to prove since all CM modules are special.

\begin{lemma}
For $\mathbb{I}_{30(b-2)+1}$ with $b\geq 3$ the specials are precisely those CM modules circled below.
{\tiny{
\[
\xymatrix@C=-5pt@R=-4pt{
&.&&.&&.&&.&&1&&0&&0&&1&&0&&1&&1&&0&&1&&1&&1&&1&&1&&1&&1&&2&&1&&1&&2&&1&&2&&2&&1&&2&&2&&2&\\
&&.&&.&&.&&1&&1&&0&&1&&1&&1&&2&&1&&1&&2&&2&&2&&2&&2&&2&&3&&3&&2&&3&&3&&3&&4&&3&&3&&4&&4&&4\\
.&.&.&.&.&.&.&1&1&1&0&1&1&1&0&1&1&2&1&2&1&2&1&2&1&2&1&3&2&3&1&3&2&3&1&3&2&4&2&4&2&4&2&4&2&4&2&5&3&5&2&5&3&5&2&5&3&6&3&6&3\\
.&&.&&.&&1&&0&&1&&1&&1&&1&&2&&1&&2&&2&&2&&2&&3&&2&&3&&3&&3&&3&&4&&3&&4&&4&&4&&4&&5&&4&&5&&5\\
&.&&.&&1&&0&&0&&1&&1&&1&&1&&1&&1&&2&&2&&1&&2&&2&&2&&3&&2&&2&&3&&3&&3&&3&&3&&3&&4&&4&&3&&4&\\
.&&.&&1&&0&&0&&0&&1&&1&&1&&0&&1&&1&&2&&1&&1&&1&&2&&2&&2&&1&&2&&2&&3&&2&&2&&2&&3&&3&&3&&2&&3\\
&.&&1&&0&&0&&0&&0&&1&&1&&0&&0&&1&&1&&1&&1&&0&&1&&2&&1&&1&&1&&1&&2&&2&&1&&1&&2&&2&&2&&2&&1&\\
{{{}\drop\xycircle<4pt,4pt>{}R}}&&1&&0&&0&&0&&0&&{{{}\drop\xycircle<4pt,4pt>{}0}}&&1&&0&&0&&{{{}\drop\xycircle<4pt,4pt>{}0}}&&1&&{{{}\drop\xycircle<4pt,4pt>{}0}}&&1&&0&&{{{}\drop\xycircle<4pt,4pt>{}0}}&&1&&1&&{{{}\drop\xycircle<4pt,4pt>{}0}}&&1&&{{{}\drop\xycircle<4pt,4pt>{}0}}&&1&&1&&1&&{{{}\drop\xycircle<4pt,4pt>{}0}}&&1&&1&&1&&1&&1&&{{{}\drop\xycircle<4pt,4pt>{}0}}\\
&&&&&&&&&&&&&&&&&&&&&&&&&&&&&&&&&&&&&&&&&&&&&&
}
\]}}
\end{lemma}
\begin{proof}
$R$ is a distance of $30(b-2)$ away from $\tau^{-1}R$, thus by Lemma~\ref{freeI}
we have {\tiny{
\[
\xymatrix@C=-6pt@R=-4pt{
2b\minus 4&&2b\minus 4&&2b\minus 4&&2b\minus 4&&b\minus 2&&2b\minus 4&&2b\minus 4&&b\minus 2&&2b\minus 4&&b\minus 2&&b\minus 2&&2b\minus 4&&b\minus 2\\
&4b\minus 8&&4b\minus 8&&4b\minus 8&&3b\minus 6&&3b\minus 6&&4b\minus 8&&3b\minus 6&&3b\minus 6&&3b\minus 6&&2b\minus 4&&3b\minus 6&&3b\minus 6&&2b\minus 4\\
6b\minus 12&3b\minus 6&6b\minus 12&3b\minus 6&6b\minus 12&3b\minus 6&5b\minus 10&2b\minus 4&5b\minus 10&3b\minus 6&5b\minus 10&2b\minus 4&5b\minus 10&3b\minus 6&5b\minus 10&2b\minus 4&4b\minus 8&2b\minus 4&4b\minus 8&2b\minus 4&4b\minus 8&2b\minus 4&4b\minus 8&2b\minus 4&4b\minus 8&2b\minus 4&\\
&5b\minus 10&&5b\minus 10&&4b\minus 8&&5b\minus 10&&4b\minus 8&&4b\minus 8&&4b\minus 8&&4b\minus 8&&3b\minus 6&&4b\minus 8&&3b\minus 6&&3b\minus 6&&3b\minus 6\\
4b\minus 8&&4b\minus 8&&3b\minus 6&&4b\minus 8&&4b\minus 8&&3b\minus 6&&3b\minus 6&&3b\minus 6&&3b\minus 6&&3b\minus 6&&3b\minus 6&&2b\minus 4&&2b\minus 4\\
&3b\minus 6&&2b\minus 4&&3b\minus 6&&3b\minus 6&&3b\minus 6&&2b\minus 4&&2b\minus 4&&2b\minus 4&&3b\minus 6&&2b\minus 4&&2b\minus 4&&b\minus 2&&2b\minus 4\\
2b\minus 4&&b\minus 2&&2b\minus 4&&2b\minus 4&&2b\minus 4&&2b\minus 4&&b\minus 2&&b\minus 2&&2b\minus 4&&2b\minus 4&&b\minus 2&&b\minus 2&&b\minus 2\\
&{{{}\drop\xycircle<4pt,4pt>{}R}}&&b\minus 2&&b\minus 2&&b\minus 2&&b\minus 2&&b\minus 2&&{{{}\drop\xycircle<4pt,4pt>{}0}}&&b\minus 2&&b\minus 2&&b\minus 2&&{{{}\drop\xycircle<4pt,4pt>{}0}}&&b\minus 2&&{{{}\drop\xycircle<4pt,4pt>{}0}}
}
\]
\[
\xymatrix@C=-7pt@R=-4pt{
b\minus 2&&b\minus 2&&b\minus 2&&b\minus 2&&b\minus 2&&b\minus 2&&0&&b\minus 2&&b\minus 2&&0&&b\minus 2&&0&&0&&b\minus 2&&0\\
&2b\minus 4&&2b\minus 4&&2b\minus 4&&2b\minus 4&&2b\minus 4&&b\minus 2&&b\minus 2&&2b\minus 4&&b\minus 2&&b\minus 2&&b\minus 2&&0&&b\minus 2&&b\minus 2&&0\\
3b\minus 6&b\minus 2&3b\minus 6&2b\minus 4&3b\minus 6&b\minus 2&3b\minus 6&2b\minus 4&3b\minus 6&b\minus 2&2b\minus 4&b\minus 2&2b\minus 4&b\minus 2&2b\minus 4&b\minus 2&2b\minus 4&b\minus 2&2b\minus 4&b\minus 2&b\minus 2&0&b\minus 2&b\minus 2&b\minus 2&0&b\minus 2&b\minus 2&b\minus 2&0&0\\
&3b\minus 6&&2b\minus 4&&3b\minus 6&&2b\minus 4&&2b\minus 4&&2b\minus 4&&2b\minus 4&&b\minus 2&&2b\minus 4&&b\minus 2&&b\minus 2&&b\minus 2&&b\minus 2&&0&&b\minus 2&&0&&&\\
3b\minus 6&&2b\minus 4&&2b\minus 4&&2b\minus 4&&b\minus 2&&2b\minus 4&&2b\minus 4&&b\minus 2&&b\minus 2&&b\minus 2&&b\minus 2&&b\minus 2&&b\minus 2&&0&&0&&b\minus 2&&0&&\\
&2b\minus 4&&2b\minus 4&&b\minus 2&&b\minus 2&&b\minus 2&&2b\minus 4&&b\minus 2&&b\minus 2&&0&&b\minus 2&&b\minus 2&&b\minus 2&&0&&0&&0&&b\minus 2&&0&\\
b\minus 2&&2b\minus 4&&b\minus 2&&0&&b\minus 2&&b\minus 2&&b\minus 2&&b\minus 2&&0&&0&&b\minus 2&&b\minus 2&&0&&0&&0&&0&&b\minus 2&&0\\
&b\minus 2&&b\minus 2&&{{{}\drop\xycircle<4pt,4pt>{}0}}&&0&&b\minus 2&&{{{}\drop\xycircle<4pt,4pt>{}0}}&&b\minus 2&&{{{}\drop\xycircle<4pt,4pt>{}0}}&&0&&0&&b\minus 2&&{{{}\drop\xycircle<4pt,4pt>{}0}}&&0&&0&&0&&0&&b\minus 2&&{{{}\drop\xycircle<4pt,4pt>{}0}}
}
\]}}
\end{proof}

\textbf{The case $m\equiv 7$.}
In this subfamily we have $m=30(b-2)+7$.
\begin{lemma}
For the group $\mathbb{I}_{7}$ (i.e. $b=2$) the following calculation determines the specials:
{\tiny{
\[
\xymatrix@C=-5pt@R=-3pt{
&.&&.&&.&&.&&1&&0&&{{}\drop\xycircle<4pt,4pt>{}0}&&1&&0&&1&&0&&0&&1&&0&&1&&0&&0&&1&&0&&1&&0&&0&&1&&0&\\
.&&.&&.&&.&&1&&1&&0&&1&&1&&1&&1&&0&&1&&1&&1&&1&&0&&1&&1&&1&&1&&0&&1&&1&&0&\\
.&.&.&.&.&.&{{}\drop\xycircle<4pt,4pt>{}.}&1&1&1&0&1&1&1&0&1&1&2&1&1&0&1&1&1&0&1&1&2&1&1&0&1&1&1&0&1&1&2&1&1&0&1&1&1&0&1&1&1&0&0&\\
.&&.&&.&&1&&0&&1&&1&&1&&1&&1&&1&&1&&1&&1&&1&&1&&1&&1&&1&&1&&1&&1&&1&&0&&1&&0&\\
&.&&.&&1&&0&&0&&1&&1&&1&&0&&1&&1&&1&&1&&0&&1&&1&&1&&1&&0&&1&&1&&1&&0&&0&&1&&0\\
.&&.&&1&&0&&0&&0&&1&&1&&0&&0&&1&&1&&1&&0&&0&&1&&1&&1&&0&&0&&1&&1&&0&&0&&0&&1&&0\\
&{{}\drop\xycircle<4pt,4pt>{}.}&&1&&0&&0&&0&&{{}\drop\xycircle<4pt,4pt>{}0}&&1&&0&&0&&0&&1&&1&&0&&0&&0&&1&&1&&0&&0&&0&&1&&0&&0&&0&&0&&1&&0&\\
{{{}\drop\xycircle<4pt,4pt>{}R}}&&1&&0&&{{}\drop\xycircle<4pt,4pt>{}0}&&0&&0&&{{}\drop\xycircle<4pt,4pt>{}0}&&R&&0&&0&&0&&1&&0&&0&&R&&0&&1&&0&&0&&0&&0&&R&&0&&0&&0&&0&&1&&0}
\]}}
\end{lemma}

\begin{lemma}
For $\mathbb{I}_{30(b-2)+7}$ with $b\geq 3$ the specials are precisely those CM modules circled below.
{\tiny{
\[
\xymatrix@C=-5pt@R=-4pt{
&.&&.&&.&&.&&1&&0&&0&&1&&0&&1&&1&&0&&1&&1&&1&&1&&1&&1&&1&&2&&1&&1&&2&&1&&2&&2&&1&&2&&2&&2&\\
&&.&&.&&.&&1&&1&&0&&1&&1&&1&&2&&1&&1&&2&&2&&2&&2&&2&&2&&3&&3&&2&&3&&3&&3&&4&&3&&3&&4&&4&&4\\
.&.&.&.&.&.&.&1&1&1&0&1&1&1&0&1&1&2&1&2&1&2&1&2&1&2&1&3&2&3&1&3&2&3&1&3&2&4&2&4&2&4&2&4&2&4&2&5&3&5&2&5&3&5&2&5&3&6&3&6&3\\
.&&.&&.&&1&&0&&1&&1&&1&&1&&2&&1&&2&&2&&2&&2&&3&&2&&3&&3&&3&&3&&4&&3&&4&&4&&4&&4&&5&&4&&5&&5\\
&.&&.&&1&&0&&0&&1&&1&&1&&1&&1&&1&&2&&2&&1&&2&&2&&2&&3&&2&&2&&3&&3&&3&&3&&3&&3&&4&&4&&3&&4&\\
.&&.&&1&&0&&0&&0&&1&&1&&1&&0&&1&&1&&2&&1&&1&&1&&2&&2&&2&&1&&2&&2&&3&&2&&2&&2&&3&&3&&3&&2&&3\\
&.&&1&&0&&0&&0&&0&&1&&1&&0&&0&&1&&1&&1&&1&&0&&1&&2&&1&&1&&1&&1&&2&&2&&1&&1&&2&&2&&2&&2&&1&\\
{{{}\drop\xycircle<4pt,4pt>{}R}}&&1&&0&&0&&0&&0&&{{{}\drop\xycircle<4pt,4pt>{}0}}&&1&&0&&0&&{{{}\drop\xycircle<4pt,4pt>{}0}}&&1&&0&&1&&0&&{{{}\drop\xycircle<4pt,4pt>{}0}}&&1&&1&&{{{}\drop\xycircle<4pt,4pt>{}0}}&&1&&{{{}\drop\xycircle<4pt,4pt>{}0}}&&1&&1&&1&&0&&1&&1&&1&&1&&1&&{{{}\drop\xycircle<4pt,4pt>{}0}}\\
&&&&&&&&&&&&&&&&&&&&&&&&&&&&&&&&&&&&&&&&&&&&&&
}
\]}}
\end{lemma}
\begin{proof}
$R$ is a distance of $30(b-2)+12$ away from $\tau^{-1}R$, thus by
Lemma~\ref{freeI} we have {\tiny{
\[
\xymatrix@C=-6pt@R=-4pt{
2b\minus 4&&2b\minus 3&&2b\minus 4&&2b\minus 2&&b\minus 2&&2b\minus 4&&2b\minus 3&&b\minus 2&&2b\minus 3&&b\minus 2&&b\minus 2&&2b\minus 3&&b\minus 2\\
&4b\minus 7&&4b\minus 7&&4b\minus 7&&3b\minus 5&&3b\minus 6&&4b\minus 7&&3b\minus 5&&3b\minus 5&&3b\minus 5&&2b\minus 4&&3b\minus 5&&3b\minus 5&&2b\minus 3\\
6b\minus 11&3b\minus 6&6b\minus 11&3b\minus 5&6b\minus 10&3b\minus 5&5b\minus 9&2b\minus 4&5b\minus 9&3b\minus 5&5b\minus 9&2b\minus 4&5b\minus 9&3b\minus 5&5b\minus 8&2b\minus 3&4b\minus 7&2b\minus 4&4b\minus 7&2b\minus 3&4b\minus 7&2b\minus 4&4b\minus 7&2b\minus 3&4b\minus 6&2b\minus 3&\\
&5b\minus 9&&5b\minus 9&&4b\minus 7&&5b\minus 9&&4b\minus 7&&4b\minus 7&&4b\minus 7&&4b\minus 7&&3b\minus 5&&4b\minus 7&&3b\minus 5&&3b\minus 5&&3b\minus 5\\
4b\minus 7&&4b\minus 7&&3b\minus 6&&4b\minus 7&&4b\minus 7&&3b\minus 5&&3b\minus 5&&3b\minus 6&&3b\minus 5&&3b\minus 5&&3b\minus 5&&2b\minus 3&&2b\minus 4\\
&3b\minus 5&&2b\minus 4&&3b\minus 6&&3b\minus 5&&3b\minus 5&&2b\minus 3&&2b\minus 4&&2b\minus 4&&3b\minus 5&&2b\minus 3&&2b\minus 3&&b\minus 2&&2b\minus 4\\
2b\minus 3&&b\minus 2&&2b\minus 4&&2b\minus 4&&2b\minus 3&&2b\minus 3&&b\minus 2&&b\minus 2&&2b\minus 4&&2b\minus 3&&b\minus 1&&b\minus 2&&b\minus 2\\
&{{{}\drop\xycircle<4pt,4pt>{}R}}&&b\minus 2&&b\minus 2&&b\minus 2&&b\minus 1&&b\minus 2&&{{{}\drop\xycircle<4pt,4pt>{}0}}&&b\minus 2&&b\minus 2&&b\minus 1&&{{{}\drop\xycircle<4pt,4pt>{}0}}&&b\minus 2&&0
}
\]
\[
\xymatrix@C=-7pt@R=-4pt{
b\minus 1&&b\minus 2&&b\minus 2&&b\minus 1&&b\minus 2&&b\minus 1&&0&&b\minus 2&&b\minus 1&&0&&b\minus 1&&0&&0&&b\minus 1&&0&&1&&0\\
&2b\minus 3&&2b\minus 4&&2b\minus 3&&2b\minus 3&&2b\minus 3&&b\minus 1&&b\minus 2&&2b\minus 3&&b\minus 1&&b\minus 1&&b\minus 1&&0&&b\minus 1&&b\minus 1&&1&&1&&0\\
3b\minus 5&b\minus 2&3b\minus 5&2b\minus 3&3b\minus 5&b\minus 2&3b\minus 5&2b\minus 3&3b\minus 4&b\minus 1&2b\minus 3&b\minus 2&2b\minus 3&b\minus 1&2b\minus 3&b\minus 2&2b\minus 3&b\minus 1&2b\minus 2&b\minus 1&b\minus 1&0&b\minus 1&b\minus 1&b\minus 1&0&b\minus 1&b\minus 1&b&1&1&0&1&1\\
&3b\minus 5&&2b\minus 3&&3b\minus 5&&2b\minus 3&&2b\minus 3&&2b\minus 3&&2b\minus 3&&b\minus 1&&2b\minus 3&&b\minus 1&&b\minus 1&&b\minus 1&&b\minus 1&&1&&b\minus 1&&1&&1&\\
3b\minus 5&&2b\minus 3&&2b\minus 3&&2b\minus 3&&b\minus 2&&2b\minus 3&&2b\minus 3&&b\minus 1&&b\minus 1&&b\minus 2&&b\minus 1&&b\minus 1&&b\minus 1&&1&&0&&b\minus 1&&1\\
&2b\minus 3&&2b\minus 3&&b\minus 1&&b\minus 2&&b\minus 2&&2b\minus 3&&b\minus 1&&b\minus 1&&0&&b\minus 2&&b\minus 1&&b\minus 1&&1&&0&&0&&b\minus 1&&1&\\
b\minus 2&&2b\minus 3&&b\minus 1&&0&&b\minus 2&&b\minus 2&&b\minus 1&&b\minus 1&&0&&0&&b\minus 2&&b\minus 1&&1&&0&&0&&0&&b\minus 1&\\
&b\minus 2&&b\minus 1&&{{{}\drop\xycircle<4pt,4pt>{}0}}&&0&&b\minus 2&&{{{}\drop\xycircle<4pt,4pt>{}0}}&&b\minus 1&&{{{}\drop\xycircle<4pt,4pt>{}0}}&&0&&0&&b\minus 2&&1&&0&&0&&0&&0&&b\minus 1
}
\]}}
{\tiny{
\[
\begin{array}{c}
\xymatrix@C=-5pt@R=-4pt{
0&&1&&0&&1&&0&&0&\\
&1&&1&&1&&1&&0&&1\\
1&0&1&1&2&1&1&0&1&1&1&0\\
&1&&1&&1&&1&&1&&1\\
1&&1&&0&&1&&1&&1&\\
&1&&0&&0&&1&&1&&1\\
1&&0&&0&&0&&1&&1\\
\ar@{.}@<1ex>[-7,0]&{{{}\drop\xycircle<4pt,4pt>{}0}}&&0&&0&&0&&1&\ar@{.}@<1ex>[-7,0]&0
}
\end{array}
...
\begin{array}{c}
\xymatrix@C=-4pt@R=-4pt{
0&&0&&1&&0&\\
&0&&1&&1&&0\\
1&1&1&0&1&1&1&0&0&\\
&1&&1&&0&&1&&0\\
1&&1&&0&&0&&1&&0\\
&1&&0&&0&&0&&1&&0\\
1&&0&&0&&0&&0&&1&&0\\
&{{{}\drop\xycircle<4pt,4pt>{}R}}&&0&&0&&0&&0&&1&&{{{}\drop\xycircle<4pt,4pt>{}0}}&
}
\end{array}
\]}}
\end{proof}

\textbf{The case $m\equiv 11$.}
In this subfamily we have $m=30(b-2)+11$.
\begin{lemma}
For the group $\mathbb{I}_{11}$ (i.e. $b=2$) the following calculation determines the specials:
{\tiny{
\[
\xymatrix@C=-5pt@R=-3pt{
&.&&.&&.&&{{}\drop\xycircle<4pt,4pt>{}.}&&1&&0&&{{}\drop\xycircle<4pt,4pt>{}0}&&1&&0&&1&&1&&0&&1&&1&&0&&1&&1&&0&&1&&1&&0&&1&&1&&0&&1&&0&&0&&1&&0&\\
.&&.&&.&&.&&1&&1&&0&&1&&1&&1&&2&&1&&1&&2&&1&&1&&2&&1&&1&&2&&1&&1&&2&&1&&1&&1&&0&&1&&1&&0&\\
.&.&.&.&.&.&.&1&1&1&0&1&1&1&0&1&1&2&1&2&1&2&1&2&1&2&1&2&1&2&1&2&1&2&1&2&1&2&1&2&1&2&1&2&1&2&1&2&1&1&0&1&1&1&0&1&1&1&0&0&\\
.&&.&&.&&1&&0&&1&&1&&1&&1&&2&&1&&2&&2&&1&&2&&2&&1&&2&&2&&1&&2&&2&&1&&2&&1&&1&&1&&1&&0&&1&&0&\\
&.&&.&&1&&0&&0&&1&&1&&1&&1&&1&&1&&2&&1&&1&&2&&1&&1&&2&&1&&1&&2&&1&&1&&1&&1&&1&&1&&0&&0&&1&&0\\
.&&.&&1&&0&&0&&0&&1&&1&&1&&0&&1&&1&&1&&1&&1&&1&&1&&1&&1&&1&&1&&1&&1&&0&&1&&1&&1&&0&&0&&0&&1&&0&\\
&{{}\drop\xycircle<4pt,4pt>{}.}&&1&&0&&0&&0&&0&&1&&1&&0&&{{}\drop\xycircle<4pt,4pt>{}0}&&1&&0&&1&&1&&0&&1&&1&&0&&1&&1&&0&&1&&0&&0&&1&&1&&0&&0&&0&&0&&1&&0&\\
{{{}\drop\xycircle<4pt,4pt>{}R}}&&1&&0&&0&&{{}\drop\xycircle<4pt,4pt>{}0}&&0&&{{}\drop\xycircle<4pt,4pt>{}0}&&1&&0&&0&&{{}\drop\xycircle<4pt,4pt>{}0}&&R&&0&&1&&0&&0&&1&&0&&0&&1&&0&&0&&R&&0&&0&&1&&0&&0&&0&&0&&0&&1&&0}
\]}}
\end{lemma}

\begin{lemma}
For $\mathbb{I}_{30(b-2)+11}$ with $b\geq 3$ the specials are precisely those CM modules circled below.
{\tiny{
\[
\xymatrix@C=-5pt@R=-4pt{
&.&&.&&.&&.&&1&&0&&0&&1&&0&&1&&1&&0&&1&&1&&1&&1&&1&&1&&1&&2&&1&&1&&2&&1&&2&&2&&1&&2&&2&&2&\\
&&.&&.&&.&&1&&1&&0&&1&&1&&1&&2&&1&&1&&2&&2&&2&&2&&2&&2&&3&&3&&2&&3&&3&&3&&4&&3&&3&&4&&4&&4\\
.&.&.&.&.&.&.&1&1&1&0&1&1&1&0&1&1&2&1&2&1&2&1&2&1&2&1&3&2&3&1&3&2&3&1&3&2&4&2&4&2&4&2&4&2&4&2&5&3&5&2&5&3&5&2&5&3&6&3&6&3\\
.&&.&&.&&1&&0&&1&&1&&1&&1&&2&&1&&2&&2&&2&&2&&3&&2&&3&&3&&3&&3&&4&&3&&4&&4&&4&&4&&5&&4&&5&&5\\
&.&&.&&1&&0&&0&&1&&1&&1&&1&&1&&1&&2&&2&&1&&2&&2&&2&&3&&2&&2&&3&&3&&3&&3&&3&&3&&4&&4&&3&&4&\\
.&&.&&1&&0&&0&&0&&1&&1&&1&&0&&1&&1&&2&&1&&1&&1&&2&&2&&2&&1&&2&&2&&3&&2&&2&&2&&3&&3&&3&&2&&3\\
&.&&1&&0&&0&&0&&0&&1&&1&&0&&0&&1&&1&&1&&1&&0&&1&&2&&1&&1&&1&&1&&2&&2&&1&&1&&2&&2&&2&&2&&1&\\
{{{}\drop\xycircle<4pt,4pt>{}R}}&&1&&0&&0&&0&&0&&{{{}\drop\xycircle<4pt,4pt>{}0}}&&1&&0&&0&&{{{}\drop\xycircle<4pt,4pt>{}0}}&&1&&{{{}\drop\xycircle<4pt,4pt>{}0}}&&1&&0&&{{{}\drop\xycircle<4pt,4pt>{}0}}&&1&&1&&{{{}\drop\xycircle<4pt,4pt>{}0}}&&1&&0&&1&&1&&1&&{{{}\drop\xycircle<4pt,4pt>{}0}}&&1&&1&&1&&1&&1&&{{{}\drop\xycircle<4pt,4pt>{}0}}\\
&&&&&&&&&&&&&&&&&&&&&&&&&&&&&&&&&&&&&&&&&&&&&&
}
\]}}
\end{lemma}
\begin{proof}
$R$ is a distance of $30(b-2)+20$ away from $\tau^{-1}R$, thus by
Lemma~\ref{freeI} we have {\tiny{
\[
\xymatrix@C=-6pt@R=-4pt{
2b\minus 3&&2b\minus 4&&2b\minus 3&&2b\minus 3&&b\minus 2&&2b\minus 3&&2b\minus 3&&b\minus 2&&2b\minus 3&&b\minus 1&&b\minus 2&&2b\minus 3&&b\minus 1\\
&4b\minus 7&&4b\minus 7&&4b\minus 6&&3b\minus 5&&3b\minus 5&&4b\minus 6&&3b\minus 5&&3b\minus 5&&3b\minus 4&&2b\minus 3&&3b\minus 5&&3b\minus 4&&2b\minus 3\\
6b\minus 10&3b\minus 5&6b\minus 10&3b\minus 5&6b\minus 10&3b\minus 5&5b\minus 8&2b\minus 3&5b\minus 8&3b\minus 5&5b\minus 8&2b\minus 3&5b\minus 8&3b\minus 5&5b\minus 8&2b\minus 3&4b\minus 6&2b\minus 3&4b\minus 6&2b\minus 3&4b\minus 6&2b\minus 3&4b\minus 6&2b\minus 3&4b\minus 6&2b\minus 3&\\
&5b\minus 8&&5b\minus 8&&4b\minus 7&&5b\minus 8&&4b\minus 6&&4b\minus 7&&4b\minus 6&&4b\minus 6&&3b\minus 5&&4b\minus 6&&3b\minus 4&&3b\minus 5&&3b\minus 4\\
4b\minus 7&&4b\minus 6&&3b\minus 5&&4b\minus 7&&4b\minus 6&&3b\minus 5&&3b\minus 5&&3b\minus 4&&3b\minus 5&&3b\minus 5&&3b\minus 4&&2b\minus 3&&2b\minus 3\\
&3b\minus 5&&2b\minus 3&&3b\minus 5&&3b\minus 5&&3b\minus 5&&2b\minus 3&&2b\minus 3&&2b\minus 3&&3b\minus 5&&2b\minus 3&&2b\minus 3&&b\minus 1&&2b\minus 3\\
2b\minus 3&&b\minus 2&&2b\minus 3&&2b\minus 3&&2b\minus 4&&2b\minus 3&&b\minus 1&&b\minus 2&&2b\minus 3&&2b\minus 3&&b\minus 2&&b\minus 1&&b\minus 1\\
&{{{}\drop\xycircle<4pt,4pt>{}R}}&&b\minus 2&&b\minus 1&&b\minus 2&&b\minus 2&&b\minus 1&&{{{}\drop\xycircle<4pt,4pt>{}0}}&&b\minus 2&&b\minus 1&&b\minus 2&&{{{}\drop\xycircle<4pt,4pt>{}0}}&&b\minus 1&&{{{}\drop\xycircle<4pt,4pt>{}0}}
}
\]
\[
\xymatrix@C=-7pt@R=-4pt{
b\minus 2&&b\minus 1&&b\minus 1&&b\minus 2&&b\minus 1&&b\minus 1&&0&&b\minus 1&&b\minus 1&&0&&b\minus 1&&1&&0&&b\minus 1&&1&&0&&1\\
&2b\minus 3&&2b\minus 2&&2b\minus 3&&2b\minus 3&&2b\minus 2&&b\minus 1&&b\minus 1&&2b\minus 2&&b\minus 1&&b\minus 1&&b&&1&&b\minus 1&&b&&1&&1&&2\\
3b\minus 4&b\minus 1&3b\minus 4&2b\minus 3&3b\minus 4&b\minus 1&3b\minus 4&2b\minus 3&3b\minus 3&b\minus 1&2b\minus 2&b\minus 1&2b\minus 2&b\minus 1&2b\minus 2&b\minus 1&2b\minus 2&b\minus 1&2b\minus 2&b\minus 1&b&1&b&b\minus 1&b&1&b&b\minus 1&b&1&2&1&2&1\\
&3b\minus 4&&2b\minus 3&&3b\minus 4&&2b\minus 2&&2b\minus 3&&2b\minus 2&&2b\minus 2&&b\minus 1&&2b\minus 2&&b&&b\minus 1&&b&&b&&1&&b&&2&&1&\\
3b\minus 4&&2b\minus 3&&2b\minus 3&&2b\minus 2&&b\minus 1&&2b\minus 3&&2b\minus 2&&b\minus 1&&b\minus 1&&b&&b\minus 1&&b\minus 1&&b&&1&&1&&b&&1\\
&2b\minus 3&&2b\minus 3&&b\minus 1&&b\minus 1&&b\minus 1&&2b\minus 3&&b\minus 1&&b\minus 1&&1&&b\minus 1&&b\minus 1&&b\minus 1&&1&&1&&1&&b\minus 1&&1&\\
b\minus 2&&2b\minus 3&&b\minus 1&&0&&b\minus 1&&b\minus 1&&b\minus 2&&b\minus 1&&1&&0&&b\minus 1&&b\minus 1&&0&&1&&1&&0&&b\minus 1&\\
&b\minus 2&&b\minus 1&&{{{}\drop\xycircle<4pt,4pt>{}0}}&&0&&b\minus 1&&{{{}\drop\xycircle<4pt,4pt>{}0}}&&b\minus 2&&1&&0&&0&&b\minus 1&&{{{}\drop\xycircle<4pt,4pt>{}0}}&&0&&1&&0&&0&&b\minus 1
}
\]}}
{\tiny{
\[
\begin{array}{c}
\xymatrix@C=-5pt@R=-4pt{
1&&0&&1&&1&&0\\
&1&&1&&2&&1\\
2&1&2&1&2&1&2&1&2\\
&2&&2&&1&&2&\\
1&&2&&1&&1&&2\\
&1&&1&&1&&1\\
1&&0&&1&&1&&0&\\
&\ar@{.}@<1ex>[-7,0]{{{}\drop\xycircle<4pt,4pt>{}0}}&&0&&1&&0\ar@{.}@<1ex>[-7,0]
}
\end{array}
...
\begin{array}{c}
\xymatrix@C=-4pt@R=-4pt{
1&&1&&0&&1&&0&&0&&1&&0& \\
&2&&1&&1&&1&&0&&1&&1&&0 \\
2&1&2&1&2&1&1&0&1&1&1&0&1&1&1&0&0\\
&1&&2&&1&&1&&1&&1&&0&&1&&0\\
1&&1&&1&&1&&1&&1&&0&&0&&1&&0\\
&1&&0&&1&&1&&1&&0&&0&&0&&1&&0\\
1&&0&&0&&1&&1&&0&&0&&0&&0&&1&&0\\
&{{{}\drop\xycircle<4pt,4pt>{}R}}&&0&&0&&1&&0&&0&&{{{}\drop\xycircle<4pt,4pt>{}0}}&&0&&0&&1&&{{{}\drop\xycircle<4pt,4pt>{}0}}
}
\end{array}
\]}}
\end{proof}

\textbf{The case $m\equiv 13$.}
In this subfamily we have $m=30(b-2)+13$.
\begin{lemma}
For the group $\mathbb{I}_{13}$ (i.e $b=2$) the following calculation determines the specials:
{\tiny{
\[
\xymatrix@C=-5pt@R=-3pt{
&.&&.&&.&&{{}\drop\xycircle<4pt,4pt>{}.}&&1&&0&&0&&1&&{{}\drop\xycircle<4pt,4pt>{}0}&&1&&1&&0&&1&&1&&1&&1&&0&&1&&1&&1&&1&&0&&1&&1&&1&&1&&0&&1&&1&&0&&1&&0&&0&&1&&0&\\
.&&.&&.&&.&&1&&1&&0&&1&&1&&1&&2&&1&&1&&2&&2&&2&&1&&1&&2&&2&&2&&1&&1&&2&&2&&2&&1&&1&&2&&1&&1&&1&&0&&1&&1&&0&\\
.&.&.&.&.&.&.&1&1&1&0&1&1&1&0&1&1&2&1&2&1&2&1&2&1&2&1&3&2&3&1&2&1&2&1&2&1&3&2&3&1&2&1&2&1&2&1&3&2&3&1&2&1&2&1&2&1&2&1&2&1&1&0&1&1&1&0&1&1&1&0&0\\
.&&.&&.&&1&&0&&1&&1&&1&&1&&2&&1&&2&&2&&2&&2&&2&&2&&2&&2&&2&&2&&2&&2&&2&&2&&2&&2&&2&&1&&2&&1&&1&&1&&1&&0&&1&&0&\\
&.&&.&&1&&0&&0&&1&&1&&1&&1&&1&&1&&2&&2&&1&&1&&2&&2&&2&&1&&1&&2&&2&&2&&1&&1&&2&&2&&1&&1&&1&&1&&1&&1&&0&&0&&1&&0\\
.&&.&&1&&0&&0&&0&&1&&1&&1&&0&&1&&1&&2&&1&&0&&1&&2&&2&&1&&0&&1&&2&&2&&1&&0&&1&&2&&1&&1&&0&&1&&1&&1&&0&&0&&0&&1&&0&\\
&.&&1&&0&&0&&0&&0&&1&&1&&0&&0&&1&&1&&1&&0&&0&&1&&2&&1&&0&&0&&1&&2&&1&&0&&0&&1&&1&&1&&0&&0&&1&&1&&0&&0&&0&&0&&1&&0&\\
{{{}\drop\xycircle<4pt,4pt>{}R}}&&1&&{{}\drop\xycircle<4pt,4pt>{}0}&&0&&0&&0&&{{}\drop\xycircle<4pt,4pt>{}0}&&1&&0&&0&&{{}\drop\xycircle<4pt,4pt>{}0}&&1&&{{}\drop\xycircle<4pt,4pt>{}0}&&R&&0&&0&&1&&1&&0&&0&&0&&1&&1&&0&&0&&0&&R&&1&&0&&0&&0&&1&&0&&0&&0&&0&&0&&1&&0}
\]}}
\end{lemma}
\begin{lemma}
For $\mathbb{I}_{30(b-2)+13}$ with $b\geq 3$ the specials are precisely those CM modules circled below.
{\tiny{
\[
\xymatrix@C=-5pt@R=-4pt{
&.&&.&&.&&.&&1&&0&&0&&1&&0&&1&&1&&0&&1&&1&&1&&1&&1&&1&&1&&2&&1&&1&&2&&1&&2&&2&&1&&2&&2&&2&\\
&&.&&.&&.&&1&&1&&0&&1&&1&&1&&2&&1&&1&&2&&2&&2&&2&&2&&2&&3&&3&&2&&3&&3&&3&&4&&3&&3&&4&&4&&4\\
.&.&.&.&.&.&.&1&1&1&0&1&1&1&0&1&1&2&1&2&1&2&1&2&1&2&1&3&2&3&1&3&2&3&1&3&2&4&2&4&2&4&2&4&2&4&2&5&3&5&2&5&3&5&2&5&3&6&3&6&3\\
.&&.&&.&&1&&0&&1&&1&&1&&1&&2&&1&&2&&2&&2&&2&&3&&2&&3&&3&&3&&3&&4&&3&&4&&4&&4&&4&&5&&4&&5&&5\\
&.&&.&&1&&0&&0&&1&&1&&1&&1&&1&&1&&2&&2&&1&&2&&2&&2&&3&&2&&2&&3&&3&&3&&3&&3&&3&&4&&4&&3&&4&\\
.&&.&&1&&0&&0&&0&&1&&1&&1&&0&&1&&1&&2&&1&&1&&1&&2&&2&&2&&1&&2&&2&&3&&2&&2&&2&&3&&3&&3&&2&&3\\
&.&&1&&0&&0&&0&&0&&1&&1&&0&&0&&1&&1&&1&&1&&0&&1&&2&&1&&1&&1&&1&&2&&2&&1&&1&&2&&2&&2&&2&&1&\\
{{{}\drop\xycircle<4pt,4pt>{}R}}&&1&&0&&0&&0&&0&&{{{}\drop\xycircle<4pt,4pt>{}0}}&&1&&0&&0&&{{{}\drop\xycircle<4pt,4pt>{}0}}&&1&&{{{}\drop\xycircle<4pt,4pt>{}0}}&&1&&0&&{{{}\drop\xycircle<4pt,4pt>{}0}}&&1&&1&&0&&1&&{{{}\drop\xycircle<4pt,4pt>{}0}}&&1&&1&&1&&0&&1&&1&&1&&1&&1&&{{{}\drop\xycircle<4pt,4pt>{}0}}\\
&&&&&&&&&&&&&&&&&&&&&&&&&&&&&&&&&&&&&&&&&&&&&&
}
\]}}
\end{lemma}
\begin{proof}
$R$ is a distance of $30(b-2)+24$ away from $\tau^{-1}R$, thus by
Lemma~\ref{freeI} we have {\tiny{
\[
\xymatrix@C=-6pt@R=-4pt{
2b\minus 3&&2b\minus 3&&2b\minus 3&&2b\minus 3&&b\minus 2&&2b\minus 3&&2b\minus 3&&b\minus 1&&2b\minus 3&&b\minus 2&&b\minus 1&&2b\minus 3&&b\minus 1\\
&4b\minus 6&&4b\minus 6&&4b\minus 6&&3b\minus 5&&3b\minus 5&&4b\minus 6&&3b\minus 4&&3b\minus 4&&3b\minus 5&&2b\minus 3&&3b\minus 4&&3b\minus 4&&2b\minus 2\\
6b\minus 10&3b\minus 5&6b\minus 9&3b\minus 4&6b\minus 9&3b\minus 5&5b\minus 8&2b\minus 3&5b\minus 8&3b\minus 5&5b\minus 8&2b\minus 3&5b\minus 7&3b\minus 4&5b\minus 7&2b\minus 3&4b\minus 6&2b\minus 3&4b\minus 6&2b\minus 3&4b\minus 6&2b\minus 3&4b\minus 5&2b\minus 2&4b\minus 5&2b\minus 3&\\
&5b\minus 8&&5b\minus 8&&4b\minus 6&&5b\minus 8&&4b\minus 6&&4b\minus 6&&4b\minus 6&&4b\minus 6&&3b\minus 4&&4b\minus 6&&3b\minus 4&&3b\minus 4&&3b\minus 4\\
4b\minus 6&&4b\minus 7&&3b\minus 5&&4b\minus 6&&4b\minus 6&&3b\minus 4&&3b\minus 5&&3b\minus 5&&3b\minus 4&&3b\minus 4&&3b\minus 4&&2b\minus 3&&2b\minus 3\\
&3b\minus 5&&2b\minus 4&&3b\minus 5&&3b\minus 4&&3b\minus 4&&2b\minus 3&&2b\minus 4&&2b\minus 3&&3b\minus 4&&2b\minus 2&&2b\minus 3&&b\minus 2&&2b\minus 3\\
2b\minus 3&&b\minus 2&&2b\minus 4&&2b\minus 3&&2b\minus 2&&2b\minus 3&&b\minus 2&&b\minus 2&&2b\minus 3&&2b\minus 2&&b\minus 1&&b\minus 2&&b\minus 2\\
&{{{}\drop\xycircle<4pt,4pt>{}R}}&&b\minus 2&&b\minus 2&&b\minus 1&&b\minus 1&&b\minus 2&&{{{}\drop\xycircle<4pt,4pt>{}0}}&&b\minus 2&&b\minus 1&&b\minus 1&&{{{}\drop\xycircle<4pt,4pt>{}0}}&&b\minus 2&&{{{}\drop\xycircle<4pt,4pt>{}0}}
}
\]
\[
\xymatrix@C=-7pt@R=-4pt{
b\minus 1&&b\minus 2&&b\minus 1&&b\minus 1&&b\minus 1&&b\minus 1&&0&&b\minus 1&&b\minus 1&&0&&b\minus 2&&0&&0&&b\minus 3&&0\\
&2b\minus 3&&2b\minus 3&&2b\minus 2&&2b\minus 2&&2b\minus 2&&b\minus 1&&b\minus 1&&2b\minus 2&&b\minus 1&&b\minus 2&&b\minus 3&&0&&b\minus 3&&b\minus 3&&0\\
3b\minus 4&b\minus 1&3b\minus 4&2b\minus 3&3b\minus 4&b\minus 1&3b\minus 3&2b\minus 2&3b\minus 3&b\minus 1&2b\minus 2&b\minus 1&2b\minus 2&b\minus 1&2b\minus 2&b\minus 1&2b\minus 1&b\minus 1&2b\minus 3&b\minus 2&b\minus 3&0&b\minus 3&b\minus 3&b\minus 3&0&b\minus 3&b\minus 3&b\minus 3&0&0\\
&3b\minus 4&&2b\minus 2&&3b\minus 4&&2b\minus 2&&2b\minus 2&&2b\minus 2&&2b\minus 2&&b&&2b\minus 2&&b\minus 2&&b\minus 3&&b\minus 3&&b\minus 3&&0&&b\minus 3&&0&&&\\
3b\minus 4&&2b\minus 2&&2b\minus 2&&2b\minus 3&&b\minus 1&&2b\minus 2&&2b\minus 2&&b&&b\minus 1&&b\minus 1&&b\minus 2&&b\minus 3&&b\minus 3&&0&&0&&b\minus 3&&0&&\\
&2b\minus 2&&2b\minus 2&&b\minus 1&&b\minus 2&&b\minus 1&&2b\minus 2&&b&&b\minus 1&&0&&b\minus 1&&b\minus 2&&b\minus 3&&0&&0&&0&&b\minus 3&&0&\\
b\minus 1&&2b\minus 2&&b\minus 1&&0&&b\minus 2&&b\minus 1&&b&&b\minus 1&&0&&0&&b\minus 1&&b\minus 2&&0&&0&&0&&0&&b\minus 3&&0\\
&b\minus 1&&b\minus 1&&{{{}\drop\xycircle<4pt,4pt>{}0}}&&0&&b\minus 2&&1&&b\minus 1&&{{{}\drop\xycircle<4pt,4pt>{}0}}&&0&&0&&b\minus 1&&1&&0&&0&&0&&0&&b\minus 3&&{{{}\drop\xycircle<4pt,4pt>{}0}}
}
\]}}
\end{proof}

\textbf{The case $m\equiv 17$.}
In this subfamily we have $m=30(b-2)+17$.
\begin{lemma}
For the group $\mathbb{I}_{17}$ (i.e. $b=2$) the following calculation determines the specials:
{\tiny{
\[
\xymatrix@C=-6.25pt@R=-4pt{
&.&&.&&.&&.&&1&&0&&{{}\drop\xycircle<4pt,4pt>{}0}&&1&&0&&1&&1&&0&&1&&1&&1&&1&&1&&1&&1&&2&&0&&1&&2&&0&&2&&1&&0&&2&&1&&1&&1&&1&&1&&1&&2&&0&&1&&1&&0&&2&&0&&0&&1&&0&&1&&0&&0&&1&&0&&1&&0&&0&&1&&0&&\\
.&&.&&.&&.&&1&&1&&0&&1&&1&&1&&2&&1&&1&&2&&2&&2&&2&&2&&2&&3&&2&&1&&3&&2&&2&&3&&1&&2&&3&&2&&2&&2&&2&&2&&3&&2&&1&&2&&1&&2&&2&&0&&1&&1&&1&&1&&0&&1&&1&&1&&1&&0&&1&&1&&0&&\\
.&.&.&.&.&.&.&1&1&1&0&1&1&1&0&1&1&2&1&2&1&2&1&2&1&2&1&3&2&3&1&3&2&3&1&3&2&4&2&3&1&3&2&3&1&3&2&4&2&3&1&3&2&3&1&3&2&4&2&3&1&3&2&3&1&3&2&4&2&3&1&3&2&2&0&2&2&3&1&2&1&2&1&1&0&1&1&2&1&1&0&1&1&1&0&1&1&2&1&1&0&1&1&1&0&1&1&1&0&0\\
.&&.&&.&&1&&0&&1&&1&&1&&1&&2&&1&&2&&2&&2&&2&&3&&2&&3&&3&&2&&3&&3&&2&&3&&3&&2&&3&&3&&2&&3&&3&&2&&3&&3&&2&&3&&2&&2&&2&&2&&1&&2&&1&&1&&1&&1&&1&&1&&1&&1&&1&&1&&1&&0&&1&&0\\
&.&&.&&1&&0&&0&&1&&1&&1&&1&&1&&1&&2&&2&&1&&2&&2&&2&&3&&1&&2&&3&&2&&2&&2&&2&&2&&3&&2&&1&&3&&2&&2&&3&&1&&2&&2&&2&&2&&1&&1&&1&&2&&1&&0&&1&&1&&1&&1&&0&&1&&1&&1&&0&&0&&1&&0&&\\
.&&.&&1&&0&&0&&0&&1&&1&&1&&0&&1&&1&&2&&1&&1&&1&&2&&2&&1&&1&&2&&2&&2&&1&&1&&2&&2&&2&&1&&1&&2&&2&&2&&1&&1&&1&&2&&2&&1&&0&&1&&1&&2&&0&&0&&1&&1&&1&&0&&0&&1&&1&&0&&0&&0&&1&&0\\
&{{}\drop\xycircle<4pt,4pt>{}.}&&1&&0&&0&&0&&0&&1&&1&&0&&0&&1&&1&&1&&1&&0&&1&&2&&0&&1&&1&&1&&2&&1&&0&&1&&2&&1&&1&&1&&0&&2&&2&&0&&1&&0&&1&&2&&1&&0&&0&&1&&1&&1&&0&&0&&1&&1&&0&&0&&0&&1&&0&&0&&0&&0&&1&&0\\
{{{}\drop\xycircle<4pt,4pt>{}R}}&&1&&0&&0&&0&&0&&{{}\drop\xycircle<4pt,4pt>{}0}&&1&&0&&0&&{{}\drop\xycircle<4pt,4pt>{}0}&&1&&0&&1&&0&&{{}\drop\xycircle<4pt,4pt>{}0}&&1&&R&&0&&1&&0&&1&&1&&0&&0&&1&&1&&0&&1&&0&&0&&2&&0&&0&&R&&0&&1&&1&&0&&0&&0&&1&&0&&1&&0&&0&&1&&0&&0&&0&&0&&R&&0&&0&&0&&0&&1&&0}
\]}}
\end{lemma}
\begin{lemma}
For $\mathbb{I}_{30(b-2)+17}$ with $b\geq 3$ the specials are precisely those CM modules circled below.
{\tiny{
\[
\xymatrix@C=-5pt@R=-4pt{
&.&&.&&.&&.&&1&&0&&0&&1&&0&&1&&1&&0&&1&&1&&1&&1&&1&&1&&1&&2&&1&&1&&2&&1&&2&&2&&1&&2&&2&&2&\\
&&.&&.&&.&&1&&1&&0&&1&&1&&1&&2&&1&&1&&2&&2&&2&&2&&2&&2&&3&&3&&2&&3&&3&&3&&4&&3&&3&&4&&4&&4\\
.&.&.&.&.&.&.&1&1&1&0&1&1&1&0&1&1&2&1&2&1&2&1&2&1&2&1&3&2&3&1&3&2&3&1&3&2&4&2&4&2&4&2&4&2&4&2&5&3&5&2&5&3&5&2&5&3&6&3&6&3\\
.&&.&&.&&1&&0&&1&&1&&1&&1&&2&&1&&2&&2&&2&&2&&3&&2&&3&&3&&3&&3&&4&&3&&4&&4&&4&&4&&5&&4&&5&&5\\
&.&&.&&1&&0&&0&&1&&1&&1&&1&&1&&1&&2&&2&&1&&2&&2&&2&&3&&2&&2&&3&&3&&3&&3&&3&&3&&4&&4&&3&&4&\\
.&&.&&1&&0&&0&&0&&1&&1&&1&&0&&1&&1&&2&&1&&1&&1&&2&&2&&2&&1&&2&&2&&3&&2&&2&&2&&3&&3&&3&&2&&3\\
&.&&1&&0&&0&&0&&0&&1&&1&&0&&0&&1&&1&&1&&1&&0&&1&&2&&1&&1&&1&&1&&2&&2&&1&&1&&2&&2&&2&&2&&1&\\
{{{}\drop\xycircle<4pt,4pt>{}R}}&&1&&0&&0&&0&&0&&{{{}\drop\xycircle<4pt,4pt>{}0}}&&1&&0&&0&&{{{}\drop\xycircle<4pt,4pt>{}0}}&&1&&0&&1&&0&&{{{}\drop\xycircle<4pt,4pt>{}0}}&&1&&1&&{{{}\drop\xycircle<4pt,4pt>{}0}}&&1&&0&&1&&1&&1&&0&&1&&1&&1&&1&&1&&{{{}\drop\xycircle<4pt,4pt>{}0}}\\
&&&&&&&&&&&&&&&&&&&&&&&&&&&&&&&&&&&&&&&&&&&&&&
}
\]}}
\end{lemma}
\begin{proof}
$R$ is a distance of $30(b-2)+32$ away from $\tau^{-1}R$, thus by
Lemma~\ref{freeI} we have {\tiny{
\[
\xymatrix@C=-6pt@R=-4pt{
2b\minus 3&&2b\minus 3&&2b\minus 3&&2b\minus 2&&b\minus 2&&2b\minus 3&&2b\minus 2&&b\minus 2&&2b\minus 2&&b\minus 1&&b\minus 2&&2b\minus 2&&b\minus 1\\
&4b\minus 6&&4b\minus 6&&4b\minus 5&&3b\minus 4&&3b\minus 5&&4b\minus 5&&3b\minus 4&&3b\minus 4&&3b\minus 3&&2b\minus 3&&3b\minus 4&&3b\minus 3&&2b\minus 2\\
6b\minus 9&3b\minus 5&6b\minus 9&3b\minus 4&6b\minus 8&3b\minus 4&5b\minus 7&2b\minus 3&5b\minus 7&3b\minus 4&5b\minus 7&2b\minus 3&5b\minus 7&3b\minus 4&5b\minus 6&2b\minus 2&4b\minus 5&2b\minus 3&4b\minus 5&2b\minus 2&4b\minus 5&2b\minus 3&4b\minus 5&2b\minus 2&4b\minus 4&2b\minus 2&\\
&5b\minus 7&&5b\minus 7&&4b\minus 6&&5b\minus 7&&4b\minus 5&&4b\minus 6&&4b\minus 5&&4b\minus 5&&3b\minus 4&&4b\minus 5&&3b\minus 3&&3b\minus 4&&3b\minus 3\\
4b\minus 6&&4b\minus 5&&3b\minus 5&&4b\minus 6&&4b\minus 5&&3b\minus 4&&3b\minus 4&&3b\minus 4&&3b\minus 4&&3b\minus 4&&3b\minus 3&&2b\minus 2&&2b\minus 3\\
&3b\minus 4&&2b\minus 3&&3b\minus 5&&3b\minus 4&&3b\minus 4&&2b\minus 2&&2b\minus 3&&2b\minus 3&&3b\minus 4&&2b\minus 2&&2b\minus 2&&b\minus 1&&2b\minus 3\\
2b\minus 2&&b\minus 2&&2b\minus 3&&2b\minus 3&&2b\minus 3&&2b\minus 2&&b\minus 1&&b\minus 2&&2b\minus 3&&2b\minus 2&&b\minus 1&&b\minus 1&&b\minus 1\\
&{{{}\drop\xycircle<4pt,4pt>{}R}}&&b\minus 2&&b\minus 1&&b\minus 2&&b\minus 1&&b\minus 1&&{{{}\drop\xycircle<4pt,4pt>{}0}}&&b\minus 2&&b\minus 1&&b\minus 1&&{{{}\drop\xycircle<4pt,4pt>{}0}}&&b\minus 1&&0
}
\]
\[
\xymatrix@C=-7pt@R=-4pt{
b\minus 1&&b\minus 1&&b\minus 1&&b\minus 1&&b\minus 1&&b&&0&&b\minus 1&&b&&0&&b&&1&&0&&b&&1\\
&2b\minus 2&&2b\minus 2&&2b\minus 2&&2b\minus 2&&2b\minus 1&&b&&b\minus 1&&2b\minus 1&&b&&b&&b\plus 1&&1&&b&&b\plus 1&&2\\
3b\minus 3&b\minus 1&3b\minus 3&2b\minus 2&3b\minus 3&b\minus 1&3b\minus 3&2b\minus 2&3b\minus 2&b&2b\minus 1&b\minus 1&2b\minus 1&b&2b\minus 1&b\minus 1&2b\minus 1&b&2b&b&b\plus 1&1&b\plus 1&b&b\plus 1&1&b\plus 1&b&b\plus 2&2\\
&3b\minus 3&&2b\minus 2&&3b\minus 3&&2b\minus 1&&2b\minus 2&&2b\minus 1&&2b\minus 1&&b&&2b\minus 1&&b\plus 1&&b&&b\plus 1&&b\plus 1&&2&&b\plus 1\\
3b\minus 3&&2b\minus 2&&2b\minus 2&&2b\minus 1&&b\minus 1&&2b\minus 2&&2b\minus 1&&b&&b&&b&&b&&b&&b\plus 1&&2&&1&\\
&2b\minus 2&&2b\minus 2&&b&&b\minus 1&&b\minus 1&&2b\minus 2&&b&&b&&1&&b\minus 1&&b&&b&&2&&1&&1&\\
b\minus 2&&2b\minus 2&&b&&0&&b\minus 1&&b\minus 1&&b\minus 1&&b&&1&&0&&b\minus 1&&b&&1&&1&&1&\\
&b\minus 2&&b&&{{{}\drop\xycircle<4pt,4pt>{}0}}&&0&&b\minus 1&&{{{}\drop\xycircle<4pt,4pt>{}0}}&&b\minus 1&&1&&0&&0&&b\minus 1&&1&&0&&1&&0
}
\]}}
{\tiny{
\[\begin{array}{c}
\xymatrix@C=-5pt@R=-4pt{
1&&1&&1&&1&&1&&2&&0&&1&&2&&0&&2&&1&&0&&2&&1&&1&&1&&1&\\
&2&&2&&2&&2&&3&&2&&1&&3&&2&&2&&3&&1&&2&&3&&2&&2&&2&&2\\
3&1&3&2&3&1&3&2&4&2&3&1&3&2&3&1&3&2&4&2&3&1&3&2&3&1&3&2&4&2&3&1&3&2&3&1\\
&3&&2&&3&&3&&2&&3&&3&&2&&3&&3&&2&&3&&3&&2&&3&&3&&2&&3\\
b\plus 1&&2&&2&&3&&1&&2&&3&&2&&2&&2&&2&&2&&3&&2&&1&&3&&2&&2&\\
&b&&2&&2&&1&&1&&2&&2&&2&&1&&1&&2&&2&&2&&1&&1&&2&&2&&2\\
0&&b&&2&&0&&1&&1&&1&&2&&1&&0&&1&&2&&1&&1&&1&&0&&2&&2\\
&0&&b&\ar@{.}@<1ex>[-7,0]&{{{}\drop\xycircle<4pt,4pt>{}0}}&&0&&1&&0&&1&&1&&0&&0&&1&&1&&0&&1&&0&&0&
&2&\ar@{.}@<1ex>[-7,0]&0
}
\end{array}
...
\]
}}Now the dotted segment repeats until it reaches $R$:
{\tiny{
\[
...
\begin{array}{c}
\xymatrix@C=-5pt@R=-4pt{
1&&2&&0&&1&&1&&0&&2&&0&&0&&1&&0&&1&&0&&0&&1&&0&&1&&0&&0&&\\
&3&&2&&1&&2&&1&&2&&2&&0&&1&&1&&1&&1&&0&&1&&1&&1&&1&&0&&1\\
4&2&3&1&3&2&2&0&2&2&3&1&2&1&2&1&1&0&1&1&2&1&1&0&1&1&1&0&1&1&2&1&1&0&1&1&1&0\\
&2&&3&&2&&2&&2&&2&&1&&2&&1&&1&&1&&1&&1&&1&&1&&1&&1&&1&&1\\
1&&2&&2&&2&&2&&1&&1&&1&&2&&1&&0&&1&&1&&1&&1&&0&&1&&1&&1\\
&1&&1&&2&&2&&1&&0&&1&&1&&2&&0&&0&&1&&1&&1&&0&&0&&1&&1&&1\\
1&&0&&1&&2&&1&&0&&0&&1&&1&&1&&0&&0&&1&&1&&0&&0&&0&&1&&1\\
&{{{}\drop\xycircle<4pt,4pt>{}R}}&&0&&1&&1&&0&&0&&{{{}\drop\xycircle<4pt,4pt>{}0}}&&1&&0&&1&&{{{}\drop\xycircle<4pt,4pt>{}0}}&&0&&1&\ar@{.}@<1ex>[-7,0]
&0&&0&&{{{}\drop\xycircle<4pt,4pt>{}0}}&&0&&1&\ar@{.}@<1ex>[-7,0]&{{{}\drop\xycircle<4pt,4pt>{}0}}
}
\end{array}
...
\begin{array}{c}
\xymatrix@C=-4pt@R=-4pt{
0&&0&&1&&0&\\
&0&&1&&1&&0\\
1&1&1&0&1&1&1&0&0&\\
&1&&1&&0&&1&&0\\
1&&1&&0&&0&&1&&0\\
&1&&0&&0&&0&&1&&0\\
1&&0&&0&&0&&0&&1&&0\\
&{{{}\drop\xycircle<4pt,4pt>{}R}}&&0&&0&&0&&0&&1&&{{{}\drop\xycircle<4pt,4pt>{}0}}&
}
\end{array}
\]}}
\end{proof}

\textbf{The case $m\equiv 19$.}
In this subfamily we have $m=30(b-2)+19$.
\begin{lemma}
For the group $\mathbb{I}_{19}$ (i.e. $b=2$) the following calculation determines the specials:
{\tiny{
\[
\xymatrix@C=-6.25pt@R=-4pt{
&.&&.&&.&&.&&1&&0&&0&&1&&0&&1&&1&&0&&1&&1&&1&&1&&1&&1&&1&&2&&1&&1&&1&&1&&2&&1&&1&&1&&1&&2&&1&&1&&1&&1&&2&&1&&1&&1&&1&&2&&1&&0&&1&&1&&1&&1&&0&&1&&1&&1&&1&&0&&1&&1&&1&&1&&0&\\
.&&.&&.&&.&&1&&1&&0&&1&&1&&1&&2&&1&&1&&2&&2&&2&&2&&2&&2&&3&&3&&2&&2&&2&&3&&3&&2&&2&&2&&3&&3&&2&&2&&2&&3&&3&&2&&2&&2&&3&&3&&1&&1&&2&&2&&2&&1&&1&&2&&2&&2&&1&&1&&2&&2&&2&&1&&1\\
.&.&.&.&.&.&.&1&1&1&0&1&1&1&0&1&1&2&1&2&1&2&1&2&1&2&1&3&2&3&1&3&2&3&1&3&2&4&2&4&2&4&2&3&1&3&2&4&2&4&2&4&2&3&1&3&2&4&2&4&2&4&2&3&1&3&2&4&2&4&2&4&2&3&1&3&2&4&2&4&2&3&1&2&1&2&1&3&2&3&1&2&1&2&1&2&1&3&2&3&1&2&1&2&1&2&1&3&2&3&1&2&1&2&1\\
.&&.&&.&&1&&0&&1&&1&&1&&1&&2&&1&&2&&2&&2&&2&&3&&2&&3&&3&&3&&3&&3&&3&&3&&3&&3&&3&&3&&3&&3&&3&&3&&3&&3&&3&&3&&3&&3&&3&&3&&2&&3&&2&&2&&2&&2&&2&&2&&2&&2&&2&&2&&2&&2&&2&&2&&2&&2\\
&.&&.&&1&&0&&0&&1&&1&&1&&1&&1&&1&&2&&2&&1&&2&&2&&2&&3&&2&&2&&2&&3&&3&&2&&2&&2&&3&&3&&2&&2&&2&&3&&3&&2&&2&&2&&3&&3&&2&&1&&2&&3&&2&&1&&1&&2&&2&&2&&1&&1&&2&&2&&2&&1&&1&&2&&2&\\
.&&.&&1&&0&&0&&0&&1&&1&&1&&0&&1&&1&&2&&1&&1&&1&&2&&2&&2&&1&&1&&2&&3&&2&&1&&1&&2&&3&&2&&1&&1&&2&&3&&2&&1&&1&&2&&3&&2&&0&&1&&2&&3&&1&&0&&1&&2&&2&&1&&0&&1&&2&&2&&1&&0&&1&&2&&2\\
&{{}\drop\xycircle<4pt,4pt>{}.}&&1&&0&&0&&0&&{{}\drop\xycircle<4pt,4pt>{}0}&&1&&1&&0&&0&&1&&1&&1&&1&&0&&1&&2&&1&&1&&0&&1&&2&&2&&1&&0&&1&&2&&2&&1&&0&&1&&2&&2&&1&&0&&1&&2&&2&&0&&0&&1&&2&&2&&0&&0&&1&&2&&1&&0&&0&&1&&2&&1&&0&&0&&1&&2\\
{{{}\drop\xycircle<4pt,4pt>{}R}}&&1&&0&&0&&0&&0&&{{}\drop\xycircle<4pt,4pt>{}0}&&1&&0&&0&&{{}\drop\xycircle<4pt,4pt>{}0}&&1&&0&&1&&0&&{{}\drop\xycircle<4pt,4pt>{}0}&&1&&1&&0&&R&&0&&1&&1&&1&&0&&0&&1&&1&&1&&0&&0&&1&&1&&1&&0&&0&&1&&1&&R&&0&&0&&1&&1&&1&&0&&0&&1&&1&&0&&0&&0&&1&&1&&0&&0&&0&&1&&R}
\]
 \begin{flushright}
$
\xymatrix@C=-6.25pt@R=-4pt{
&1&&1&&1&&0&&0&&1&&0&&1&&0&&0&&1&&0&&1&&0&&0&&1&&0&&1&&0&&0&&1&&0&&&&&&&\\
1&&2&&2&&1&&0&&1&&1&&1&&1&&0&&1&&1&&1&&1&&0&&1&&1&&1&&1&&0&&1&&1&&0&&&&&&\\
1&2&1&3&2&2&0&1&1&1&0&1&1&2&1&1&0&1&1&1&0&1&1&2&1&1&0&1&1&1&0&1&1&2&1&1&0&1&1&1&0&1&1&1&0&0&&&&&\\
2&&2&&1&&2&&1&&1&&1&&1&&1&&1&&1&&1&&1&&1&&1&&1&&1&&1&&1&&1&&1&&0&&1&&0&&\\
&2&&0&&1&&2&&1&&1&&0&&1&&1&&1&&1&&0&&1&&1&&1&&1&&0&&1&&1&&1&&0&&0&&1&&0&&&\\
2&&0&&0&&1&&2&&1&&0&&0&&1&&1&&1&&0&&0&&1&&1&&1&&0&&0&&1&&1&&0&&0&&0&&1&&0&&\\
&0&&0&&0&&1&&2&&0&&0&&0&&1&&1&&0&&0&&0&&1&&1&&0&&0&&0&&1&&0&&0&&0&&0&&1&&0&\\
R&&0&&0&&0&&1&&1&&0&&0&&0&&1&&0&&0&&0&&0&&1&&0&&0&&0&&0&&R&&0&&0&&0&&0&&1&&0}
$
 \end{flushright}
}}
\end{lemma}
\begin{lemma}
For $\mathbb{I}_{30(b-2)+19}$ with $b\geq 3$ the specials are precisely those CM modules circled below.
{\tiny{
\[
\xymatrix@C=-5pt@R=-4pt{
&.&&.&&.&&.&&1&&0&&0&&1&&0&&1&&1&&0&&1&&1&&1&&1&&1&&1&&1&&2&&1&&1&&2&&1&&2&&2&&1&&2&&2&&2&\\
&&.&&.&&.&&1&&1&&0&&1&&1&&1&&2&&1&&1&&2&&2&&2&&2&&2&&2&&3&&3&&2&&3&&3&&3&&4&&3&&3&&4&&4&&4\\
.&.&.&.&.&.&.&1&1&1&0&1&1&1&0&1&1&2&1&2&1&2&1&2&1&2&1&3&2&3&1&3&2&3&1&3&2&4&2&4&2&4&2&4&2&4&2&5&3&5&2&5&3&5&2&5&3&6&3&6&3\\
.&&.&&.&&1&&0&&1&&1&&1&&1&&2&&1&&2&&2&&2&&2&&3&&2&&3&&3&&3&&3&&4&&3&&4&&4&&4&&4&&5&&4&&5&&5\\
&.&&.&&1&&0&&0&&1&&1&&1&&1&&1&&1&&2&&2&&1&&2&&2&&2&&3&&2&&2&&3&&3&&3&&3&&3&&3&&4&&4&&3&&4&\\
.&&.&&1&&0&&0&&0&&1&&1&&1&&0&&1&&1&&2&&1&&1&&1&&2&&2&&2&&1&&2&&2&&3&&2&&2&&2&&3&&3&&3&&2&&3\\
&.&&1&&0&&0&&0&&0&&1&&1&&0&&0&&1&&1&&1&&1&&0&&1&&2&&1&&1&&1&&1&&2&&2&&1&&1&&2&&2&&2&&2&&1&\\
{{{}\drop\xycircle<4pt,4pt>{}R}}&&1&&0&&0&&0&&0&&{{{}\drop\xycircle<4pt,4pt>{}0}}&&1&&0&&0&&{{{}\drop\xycircle<4pt,4pt>{}0}}&&1&&0&&1&&0&&{{{}\drop\xycircle<4pt,4pt>{}0}}&&1&&1&&0&&1&&{{{}\drop\xycircle<4pt,4pt>{}0}}&&1&&1&&1&&0&&1&&1&&1&&1&&1&&{{{}\drop\xycircle<4pt,4pt>{}0}}\\
&&&&&&&&&&&&&&&&&&&&&&&&&&&&&&&&&&&&&&&&&&&&&&
}
\]}}
\end{lemma}
\begin{proof}
$R$ is a distance of $30(b-2)+36$ away from $\tau^{-1}R$, thus by
Lemma~\ref{freeI} we have {\tiny{
\[
\xymatrix@C=-6pt@R=-4pt{
2b\minus 3&&2b\minus 2&&2b\minus 3&&2b\minus 3&&b\minus 1&&2b\minus 3&&2b\minus 2&&b\minus 1&&2b\minus 3&&b\minus 1&&b\minus 1&&2b\minus 2&&b\minus 1\\
&4b\minus 5&&4b\minus 5&&4b\minus 6&&3b\minus 4&&3b\minus 5&&4b\minus 5&&3b\minus 3&&3b\minus 4&&3b\minus 4&&2b\minus 2&&3b\minus 3&&3b\minus 3&&2b\minus 2\\
6b\minus 8&3b\minus 4&6b\minus 8&3b\minus 4&6b\minus 8&3b\minus 4&5b\minus 7&2b\minus 3&5b\minus 7&3b\minus 4&5b\minus 6&2b\minus 2&5b\minus 6&3b\minus 4&5b\minus 6&2b\minus 2&4b\minus 5&2b\minus 3&4b\minus 5&2b\minus 2&4b\minus 4&2b\minus 2&4b\minus 4&2b\minus 2&4b\minus 4&2b\minus 2&\\
&5b\minus 7&&5b\minus 7&&4b\minus 5&&5b\minus 7&&4b\minus 5&&4b\minus 5&&4b\minus 5&&4b\minus 5&&3b\minus 3&&4b\minus 5&&3b\minus 3&&3b\minus 3&&3b\minus 3\\
4b\minus 6&&4b\minus 6&&3b\minus 4&&4b\minus 5&&4b\minus 5&&3b\minus 4&&3b\minus 4&&3b\minus 4&&3b\minus 3&&3b\minus 3&&3b\minus 4&&2b\minus 2&&2b\minus 2\\
&3b\minus 5&&2b\minus 3&&3b\minus 4&&3b\minus 3&&3b\minus 4&&2b\minus 3&&2b\minus 3&&2b\minus 2&&3b\minus 3&&2b\minus 2&&2b\minus 3&&b\minus 1&&2b\minus 2\\
2b\minus 3&&b\minus 2&&2b\minus 3&&2b\minus 2&&2b\minus 2&&2b\minus 3&&b\minus 2&&b\minus 1&&2b\minus 2&&2b\minus 2&&b\minus 1&&b\minus 2&&b\minus 1\\
&{{{}\drop\xycircle<4pt,4pt>{}R}}&&b\minus 2&&b\minus 1&&b\minus 1&&b\minus 1&&b\minus 2&&{{{}\drop\xycircle<4pt,4pt>{}0}}&&b\minus 1&&b\minus 1&&b\minus 1&&{{{}\drop\xycircle<4pt,4pt>{}0}}&&b\minus 2&&1
}
\]
\[
\xymatrix@C=-7pt@R=-4pt{
b\minus 1&&b\minus 1&&b\minus 1&&b&&b\minus 1&&b\minus 1&&1&&b\minus 1&&b&&1&&b\minus 1&&1&&1&&b&&1\\
&2b\minus 2&&2b\minus 2&&2b\minus 1&&2b\minus 1&&2b\minus 2&&b&&b&&2b\minus 1&&b\plus 1&&b&&b&&2&&b\plus 1&&b\plus 1&&2\\
3b\minus 3&b\minus 1&3b\minus 3&2b\minus 2&3b\minus 2&b&3b\minus 2&2b\minus 2&3b\minus 2&b&2b\minus 1&b\minus 1&2b\minus 1&b&2b&b&2b&b&2b&b&b\plus 1&1&b\plus 1&b&b\plus 2&2&b\plus 2&b&b\plus 2&2\\
&3b\minus 3&&2b\minus 1&&3b\minus 3&&2b\minus 1&&2b\minus 1&&2b\minus 1&&2b\minus 1&&b\plus 1&&2b\minus 1&&b\plus 1&&b\plus 1&&b\plus 1&&b\plus 1&&3&&b\plus 1\\
3b\minus 3&&2b\minus 1&&2b\minus 2&&2b\minus 2&&b&&2b\minus 1&&2b\minus 1&&b&&b&&b&&b\plus 1&&b\plus 1&&b&&2&&2&\\
&2b\minus 1&&2b\minus 2&&b\minus 1&&b\minus 1&&b&&2b\minus 1&&b&&b\minus 1&&1&&b&&b\plus 1&&b&&1&&1&&2&\\
b&&2b\minus 2&&b\minus 1&&0&&b\minus 1&&b&&b&&b\minus 1&&0&&1&&b&&b&&1&&0&&1&\\
&b\minus 1&&b\minus 1&&{{{}\drop\xycircle<4pt,4pt>{}0}}&&0&&b\minus 1&&1&&b\minus 1&&{{{}\drop\xycircle<4pt,4pt>{}0}}&&0&&1&&b\minus 1&&1&&0&&0&&1
}
\]
\[
\begin{array}{c}
\xymatrix@C=-5pt@R=-4pt{
1&&1&&1&&2&&1&&1&&1&&1\\
&2&&2&&3&&3&&2&&2&&2&&3\\
3&1&3&2&4&2&4&2&4&2&3&1&3&2&4&2\\
&3&&3&&3&&3&&3&&3&&3&&3\\
b\plus 1&&3&&2&&2&&2&&3&&3&&2\\
&b\plus 1&&2&&1&&1&&2&&3&&2&&1\\
2&&b&&1&&0&&1&&2&&2&&1\\
&1&&b\minus 1&\ar@{.}@<1ex>[-7,0]&{{{}\drop\xycircle<4pt,4pt>{}0}}&&0&&1&&1&&1&\ar@{.}@<1ex>[-7,0]&0
}
\end{array}
...
\begin{array}{c}
\xymatrix@C=-5pt@R=-4pt{
1&&1&&1&&1&&1&&1&&1&&0&&1&&1&&0&&1&&0&&0&&1&&0&\\
&2&&2&&2&&2&&2&&2&&1&&1&&2&&1&&1&&1&&0&&1&&1&&0 \\
4&2&3&1&3&2&3&1&3&2&3&1&2&1&2&1&2&1&2&1&2&1&1&0&1&1&1&0&1&1&1&0&0\\
&3&&3&&2&&3&&2&&2&&2&&2&&1&&2&&1&&1&&1&&1&&0&&1&&0\\
2&&3&&2&&2&&2&&1&&2&&2&&1&&1&&1&&1&&1&&1&&0&&0&&1&&0\\
&2&&2&&2&&1&&1&&1&&2&&1&&1&&0&&1&&1&&1&&0&&0&&0&&1&&0\\
1&&1&&2&&1&&0&&1&&1&&1&&1&&0&&0&&1&&1&&0&&0&&0&&0&&1&&0\\
&{{{}\drop\xycircle<4pt,4pt>{}R}}&&1&&1&&0&&0&&1&&{{{}\drop\xycircle<4pt,4pt>{}0}}&&1&&0&&0&&{{{}\drop\xycircle<4pt,4pt>{}0}}&&1&&0&&0&&0&&{{{}\drop\xycircle<4pt,4pt>{}0}}&&0&&1&&0}
\end{array}
\]}}
\end{proof}

\textbf{The case $m\equiv 23$.}
In this subfamily we have $m=30(b-2)+23$.
\begin{lemma}
For the group  $\mathbb{I}_{23}$ (i.e. $b=2$) the following calculation determines the specials:
{\tiny{
\[
\xymatrix@C=-6.25pt@R=-4pt{
&.&&.&&.&&{{}\drop\xycircle<4pt,4pt>{}.}&&1&&0&&0&&1&&0&&1&&1&&0&&1&&1&&1&&1&&1&&1&&1&&2&&1&&1&&2&&1&&2&&2&&0&&2&&2&&1&&2&&1&&1&&2&&2&&1&&1&&2&&1&&2&&2&&0&&2&&2&&1&&2&&1&&1&&2&&0&&1&&1&&0&&1&&0\\
.&&.&&.&&.&&1&&1&&0&&1&&1&&1&&2&&1&&1&&2&&2&&2&&2&&2&&2&&3&&3&&2&&3&&3&&3&&4&&2&&2&&4&&3&&3&&3&&2&&3&&4&&3&&2&&3&&3&&3&&4&&2&&2&&4&&3&&3&&3&&2&&3&&2&&1&&2&&1&&1&&1&&0&\\
.&.&.&.&.&.&.&1&1&1&0&1&1&1&0&1&1&2&1&2&1&2&1&2&1&2&1&3&2&3&1&3&2&3&1&3&2&4&2&4&2&4&2&4&2&4&2&5&3&5&2&4&2&4&2&4&2&5&3&5&2&4&2&4&2&4&2&5&3&5&2&4&2&4&2&4&2&5&3&5&2&4&2&4&2&4&2&5&3&5&2&4&2&4&2&4&2&3&1&3&2&2&0&2&2&2&0&1&1&1&0&0&\\
.&&.&&.&&1&&0&&1&&1&&1&&1&&2&&1&&2&&2&&2&&2&&3&&2&&3&&3&&3&&3&&4&&3&&4&&4&&3&&4&&4&&3&&4&&4&&3&&4&&4&&3&&4&&4&&3&&4&&4&&3&&4&&4&&3&&4&&4&&3&&4&&2&&3&&2&&2&&1&&2&&0&&1&&0\\
&.&&.&&1&&0&&0&&1&&1&&1&&1&&1&&1&&2&&2&&1&&2&&2&&2&&3&&2&&2&&3&&3&&3&&3&&2&&3&&4&&3&&2&&3&&3&&3&&4&&2&&2&&4&&3&&3&&3&&2&&3&&4&&3&&2&&3&&3&&3&&2&&2&&2&&2&&1&&1&&1&&0&&1&&0\\
.&&.&&1&&0&&0&&0&&1&&1&&1&&0&&1&&1&&2&&1&&1&&1&&2&&2&&2&&1&&2&&2&&3&&2&&1&&2&&3&&3&&2&&1&&2&&3&&3&&2&&1&&2&&3&&3&&2&&1&&2&&3&&3&&2&&1&&2&&3&&1&&2&&1&&2&&1&&1&&0&&1&&0&&1&&0&\\
&.&&1&&0&&0&&0&&0&&1&&1&&0&&0&&1&&1&&1&&1&&0&&1&&2&&1&&1&&1&&1&&2&&2&&0&&1&&2&&2&&2&&1&&0&&2&&3&&1&&1&&1&&1&&3&&2&&0&&1&&2&&2&&2&&1&&0&&2&&1&&1&&1&&1&&1&&1&&0&&0&&1&&0&&1&&0&\\
{{{}\drop\xycircle<4pt,4pt>{}R}}&&1&&0&&0&&0&&0&&{{}\drop\xycircle<4pt,4pt>{}0}&&1&&0&&0&&{{}\drop\xycircle<4pt,4pt>{}0}&&1&&{{}\drop\xycircle<4pt,4pt>{}0}&&1&&0&&{{}\drop\xycircle<4pt,4pt>{}0}&&1&&1&&0&&1&&0&&1&&1&&R&&0&&1&&1&&1&&1&&0&&0&&2&&1&&0&&1&&0&&1&&2&&0&&0&&1&&1&&1&&1&&0&&0&&R&&1&&0&&1&&0&&1&&0&&0&&0&&1&&0&&1&&0}
\]}}
\end{lemma}
\begin{lemma}
For $\mathbb{I}_{30(b-2)+23}$ with $b\geq 3$ the specials are precisely those CM modules circled below.
{\tiny{
\[
\xymatrix@C=-5pt@R=-4pt{
&.&&.&&.&&.&&1&&0&&0&&1&&0&&1&&1&&0&&1&&1&&1&&1&&1&&1&&1&&2&&1&&1&&2&&1&&2&&2&&1&&2&&2&&2&\\
&&.&&.&&.&&1&&1&&0&&1&&1&&1&&2&&1&&1&&2&&2&&2&&2&&2&&2&&3&&3&&2&&3&&3&&3&&4&&3&&3&&4&&4&&4\\
.&.&.&.&.&.&.&1&1&1&0&1&1&1&0&1&1&2&1&2&1&2&1&2&1&2&1&3&2&3&1&3&2&3&1&3&2&4&2&4&2&4&2&4&2&4&2&5&3&5&2&5&3&5&2&5&3&6&3&6&3\\
.&&.&&.&&1&&0&&1&&1&&1&&1&&2&&1&&2&&2&&2&&2&&3&&2&&3&&3&&3&&3&&4&&3&&4&&4&&4&&4&&5&&4&&5&&5\\
&.&&.&&1&&0&&0&&1&&1&&1&&1&&1&&1&&2&&2&&1&&2&&2&&2&&3&&2&&2&&3&&3&&3&&3&&3&&3&&4&&4&&3&&4&\\
.&&.&&1&&0&&0&&0&&1&&1&&1&&0&&1&&1&&2&&1&&1&&1&&2&&2&&2&&1&&2&&2&&3&&2&&2&&2&&3&&3&&3&&2&&3\\
&.&&1&&0&&0&&0&&0&&1&&1&&0&&0&&1&&1&&1&&1&&0&&1&&2&&1&&1&&1&&1&&2&&2&&1&&1&&2&&2&&2&&2&&1&\\
{{{}\drop\xycircle<4pt,4pt>{}R}}&&1&&0&&0&&0&&0&&{{{}\drop\xycircle<4pt,4pt>{}0}}&&1&&0&&0&&{{{}\drop\xycircle<4pt,4pt>{}0}}&&1&&{{{}\drop\xycircle<4pt,4pt>{}0}}&&1&&0&&{{{}\drop\xycircle<4pt,4pt>{}0}}&&1&&1&&0&&1&&0&&1&&1&&1&&0&&1&&1&&1&&1&&1&&{{{}\drop\xycircle<4pt,4pt>{}0}}\\
&&&&&&&&&&&&&&&&&&&&&&&&&&&&&&&&&&&&&&&&&&&&&&
}
\]}}
\end{lemma}
\begin{proof}
$R$ is a distance of $30(b-2)+44$ away from $\tau^{-1}R$, thus by
Lemma~\ref{freeI} we have {\tiny{
\[
\xymatrix@C=-6pt@R=-4pt{
2b\minus 2&&2b\minus 3&&2b\minus 2&&2b\minus 2&&b\minus 2&&2b\minus 2&&2b\minus 2&&b\minus 1&&2b\minus 2&&b\minus 1&&b\minus 1&&2b\minus 2&&b\\
&4b\minus 5&&4b\minus 5&&4b\minus 4&&3b\minus 4&&3b\minus 4&&4b\minus 4&&3b\minus 3&&3b\minus 3&&3b\minus 3&&2b\minus 2&&3b\minus 3&&3b\minus 2&&2b\minus 1\\
6b\minus 8&3b\minus 4&6b\minus 7&3b\minus 3&6b\minus 7&3b\minus 4&5b\minus 6&2b\minus 2&5b\minus 6&3b\minus 4&5b\minus 6&2b\minus 2&5b\minus 5&3b\minus 3&5b\minus 5&2b\minus 2&4b\minus 4&2b\minus 2&4b\minus 4&2b\minus 2&4b\minus 4&2b\minus 2&4b\minus 3&2b\minus 1&4b\minus 3&2b\minus 2&\\
&5b\minus 6&&5b\minus 6&&4b\minus 5&&5b\minus 6&&4b\minus 4&&4b\minus 5&&4b\minus 4&&4b\minus 4&&3b\minus 3&&4b\minus 4&&3b\minus 2&&3b\minus 3&&3b\minus 2\\
4b\minus 5&&4b\minus 5&&3b\minus 4&&4b\minus 5&&4b\minus 4&&3b\minus 3&&3b\minus 4&&3b\minus 3&&3b\minus 3&&3b\minus 3&&3b\minus 2&&2b\minus 2&&2b\minus 2\\
&3b\minus 4&&2b\minus 3&&3b\minus 4&&3b\minus 3&&3b\minus 3&&2b\minus 2&&2b\minus 3&&2b\minus 2&&3b\minus 3&&2b\minus 1&&2b\minus 2&&b\minus 1&&2b\minus 2\\
2b\minus 2&&b\minus 2&&2b\minus 3&&2b\minus 2&&2b\minus 2&&2b\minus 2&&b\minus 1&&b\minus 2&&2b\minus 2&&2b\minus 1&&b\minus 1&&b\minus 1&&b\minus 1\\
&{{{}\drop\xycircle<4pt,4pt>{}R}}&&b\minus 2&&b\minus 1&&b\minus 1&&b\minus 1&&b\minus 1&&{{{}\drop\xycircle<4pt,4pt>{}0}}&&b\minus 2&&b&&b\minus 1&&{{{}\drop\xycircle<4pt,4pt>{}0}}&&b\minus 1&&{{{}\drop\xycircle<4pt,4pt>{}0}}
}
\]
\[
\xymatrix@C=-6pt@R=-4pt{
b\minus 1&&b\minus 1&&b&&b\minus 1&&b&&b&&0&&b&&b&&1&&b&&1&&1&&b&\\
&2b\minus 2&&2b\minus 1&&2b\minus 1&&2b\minus 1&&2b&&b&&b&&2b&&b\plus 1&&b\plus 1&&b\plus 1&&2&&b\plus 1&&b\plus 2\\
3b\minus 2&b&3b\minus 2&2b\minus 2&3b\minus 2&b&3b\minus 1&2b\minus 1&3b\minus 1&b&2b&b&2b&b&2b&b&2b\plus 1&b\plus 1&2b\plus 1&b&b\plus 2&2&b\plus 2&b&b\plus 2&2&b\plus 3&b\plus 1\\
&3b\minus 2&&2b\minus 1&&3b\minus 2&&2b&&2b\minus 1&&2b&&2b&&b\plus 1&&2b&&b\plus 2&&b\plus 1&&b\plus 2&&b\plus 2&&3\\
3b\minus 2&&2b\minus 1&&2b\minus 1&&2b\minus 1&&b&&2b\minus 1&&2b&&b\plus 1&&b&&b\plus 1&&b\plus 1&&b\plus 1&&b\plus 2&&2&\\
&2b\minus 1&&2b\minus 1&&b&&b\minus 1&&b&&2b\minus 1&&b\plus 1&&b&&1&&b&&b\plus 1&&b\plus 1&&2&&1\\
b\minus 1&&2b\minus 1&&b&&0&&b\minus 1&&b&&b&&b&&1&&0&&b&&b\plus 1&&1&&1&\\
&b\minus 1&&b&&{{{}\drop\xycircle<4pt,4pt>{}0}}&&0&&b\minus 1&&1&&b\minus 1&&1&&0&&0&&b&&1&&0&&1
}
\]
\[
\xymatrix@C=-4pt@R=-4pt{
2&&1&&1&&1&&0&&1&&1&&\\
&3&&2&&2&&1&&1&&2&&0&\\
b\plus 3&2&4&2&3&1&2&1&2&1&2&1&1&0&0&\\
&b\plus 2&&3&&2&&2&&2&&0&&1&&0\\
2&&b\plus 1&&2&&2&&2&&0&&0&&1&&0\\
&1&&b&&2&&2&&0&&0&&0&&1&&0\\
1&&0&&b&&2&&0&&0&&0&&0&&1&&0\\
&0&&0&&b&&{{{}\drop\xycircle<4pt,4pt>{}0}}&&0&&0&&0&&0&&1&&0
}
\]
}}
\end{proof}

\textbf{The case $m\equiv 29$.}
In this subfamily we have $m=30(b-2)+29$.
\begin{lemma}
For the group $\mathbb{I}_{29}$ (i.e. $b=2$) the following calculation determines the specials:
{\tiny{
\[
\xymatrix@C=-6.25pt@R=-4pt{
&.&&.&&.&&.&&1&&0&&0&&1&&0&&1&&1&&0&&1&&1&&1&&1&&1&&1&&1&&2&&1&&1&&2&&1&&2&&2&&1&&2&&2&&2&&2&&2&&1&&2&&3&&1&&2&&2&&1&&3&&2&&1&&2&&2&&2&&2&&2&&1&&2&&3&&1&&2&&2&&1&&3&&2&&1&&2&\\
.&&.&&.&&.&&1&&1&&0&&1&&1&&1&&2&&1&&1&&2&&2&&2&&2&&2&&2&&3&&3&&2&&3&&3&&3&&4&&3&&3&&4&&4&&4&&4&&3&&3&&5&&4&&3&&4&&3&&4&&5&&3&&3&&4&&4&&4&&4&&3&&3&&5&&4&&3&&4&&3&&4&&5&&3&&3&&4\\
.&.&.&.&.&.&.&1&1&1&0&1&1&1&0&1&1&2&1&2&1&2&1&2&1&2&1&3&2&3&1&3&2&3&1&3&2&4&2&4&2&4&2&4&2&4&2&5&3&5&2&5&3&5&2&5&3&6&3&6&3&6&3&5&2&5&3&6&3&6&3&6&3&5&2&5&3&6&3&6&3&6&3&5&2&5&3&6&3&6&3&6&3&5&2&5&3&6&3&6&3&6&3&5&2&5&3&6&3&6&3&6&3&5&2&5&3\\
.&&.&&.&&1&&0&&1&&1&&1&&1&&2&&1&&2&&2&&2&&2&&3&&2&&3&&3&&3&&3&&4&&3&&4&&4&&4&&4&&5&&4&&5&&5&&4&&5&&5&&4&&5&&5&&4&&5&&5&&4&&5&&5&&4&&5&&5&&4&&5&&5&&4&&5&&5&&4&&5&&5&&4&&5&&5&&4\\
&.&&.&&1&&0&&0&&1&&1&&1&&1&&1&&1&&2&&2&&1&&2&&2&&2&&3&&2&&2&&3&&3&&3&&3&&3&&3&&4&&4&&3&&4&&3&&4&&5&&3&&3&&4&&4&&4&&4&&3&&3&&5&&4&&3&&4&&3&&4&&5&&3&&3&&4&&4&&4&&4&&3&&3&&5&&4&\\
.&&.&&1&&0&&0&&0&&1&&1&&1&&0&&1&&1&&2&&1&&1&&1&&2&&2&&2&&1&&2&&2&&3&&2&&2&&2&&3&&3&&3&&2&&2&&3&&4&&3&&2&&2&&3&&4&&3&&2&&2&&3&&4&&3&&2&&2&&3&&4&&3&&2&&2&&3&&4&&3&&2&&2&&3&&4&&3\\
&{{}\drop\xycircle<4pt,4pt>{}.}&&1&&0&&0&&0&&0&&1&&1&&0&&0&&1&&1&&1&&1&&0&&1&&2&&1&&1&&1&&1&&2&&2&&1&&1&&2&&2&&2&&2&&0&&2&&3&&2&&2&&1&&1&&3&&3&&1&&1&&2&&2&&3&&2&&0&&2&&3&&2&&2&&1&&1&&3&&3&&1&&1&&2&&2&&3& \\
{{{}\drop\xycircle<4pt,4pt>{}R}}&&1&&0&&0&&0&&0&&{{}\drop\xycircle<4pt,4pt>{}0}&&1&&0&&0&&{{}\drop\xycircle<4pt,4pt>{}0}&&1&&0&&1&&0&&{{}\drop\xycircle<4pt,4pt>{}0}&&1&&1&&0&&1&&0&&1&&1&&1&&0&&1&&1&&1&&1&&R&&0&&2&&1&&1&&1&&0&&1&&2&&1&&0&&1&&1&&1&&2&&0&&0&&2&&1&&1&&1&&0&&1&&2&&1&&0&&1&&1&&1&&R
}
\]}}
{\tiny{
\[
\xymatrix@C=-6.55pt@R=-4pt{
&2&&2&&2&&0&&1&&2&&1&&1&&0&&1&&1&&1&&1&&0&&1&&1&&1&&1&&0&&1&&1&&1&&1&&0&&1&&1&&1&&1&&0&&1&&1&&1&&0&&0&&1&&0&&1&&0&&0&&1&&0&&1&&0&&0&&1&&0&&1&&0&&0&&1&&0&&1&&0&&0&&1&&0&&1&&0&&0&&1&&0&&&&&&&&&&&\\
4&&4&&4&&2&&1&&3&&3&&2&&1&&1&&2&&2&&2&&1&&1&&2&&2&&2&&1&&1&&2&&2&&2&&1&&1&&2&&2&&2&&1&&1&&2&&2&&1&&0&&1&&1&&1&&1&&0&&1&&1&&1&&1&&0&&1&&1&&1&&1&&0&&1&&1&&1&&1&&0&&1&&1&&1&&1&&0&&1&&1&&0&\\
3&6&3&6&3&4&1&3&2&3&1&4&3&4&1&2&1&2&1&2&1&3&2&3&1&2&1&2&1&2&1&3&2&3&1&2&1&2&1&2&1&3&2&3&1&2&1&2&1&2&1&3&2&3&1&2&1&2&1&2&1&3&2&2&0&1&1&1&0&1&1&2&1&1&0&1&1&1&0&1&1&2&1&1&0&1&1&1&0&1&1&2&1&1&0&1&1&1&0&1&1&2&1&1&0&1&1&1&0&1&1&2&1&1&0&1&1&1&0&1&1&1&0&0&\\
4&&5&&3&&4&&3&&3&&2&&3&&2&&2&&2&&2&&2&&2&&2&&2&&2&&2&&2&&2&&2&&2&&2&&2&&2&&2&&2&&2&&2&&2&&2&&1&&2&&1&&1&&1&&1&&1&&1&&1&&1&&1&&1&&1&&1&&1&&1&&1&&1&&1&&1&&1&&1&&1&&1&&1&&1&&1&&1&&1&&0&&1&&0\\
&3&&2&&3&&4&&3&&1&&1&&3&&2&&2&&1&&1&&2&&2&&2&&1&&1&&2&&2&&2&&1&&1&&2&&2&&2&&1&&1&&2&&2&&2&&0&&1&&2&&1&&1&&0&&1&&1&&1&&1&&0&&1&&1&&1&&1&&0&&1&&1&&1&&1&&0&&1&&1&&1&&1&&0&&1&&1&&1&&0&&0&&1&&0\\
3&&0&&2&&3&&4&&1&&0&&1&&3&&2&&1&&0&&1&&2&&2&&1&&0&&1&&2&&2&&1&&0&&1&&2&&2&&1&&0&&1&&2&&2&&0&&0&&1&&2&&1&&0&&0&&1&&1&&1&&0&&0&&1&&1&&1&&0&&0&&1&&1&&1&&0&&0&&1&&1&&1&&0&&0&&1&&1&&0&&0&&0&&1&&0&&\\
&0&&0&&2&&3&&2&&0&&0&&1&&3&&1&&0&&0&&1&&2&&1&&0&&0&&1&&2&&1&&0&&0&&1&&2&&1&&0&&0&&1&&2&&0&&0&&0&&1&&2&&0&&0&&0&&1&&1&&0&&0&&0&&1&&1&&0&&0&&0&&1&&1&&0&&0&&0&&1&&1&&0&&0&&0&&1&&0&&0&&0&&0&&1&&0&&&&\\
R&&0&&0&&2&&1&&1&&0&&0&&1&&2&&0&&0&&0&&1&&1&&0&&0&&0&&1&&1&&0&&0&&0&&1&&1&&0&&0&&0&&1&&R&&0&&0&&0&&1&&1&&0&&0&&0&&1&&0&&0&&0&&0&&1&&0&&0&&0&&0&&1&&0&&0&&0&&0&&1&&0&&0&&0&&0&&R&&0&&0&&0&&0&&1&&0
}
\]}}
\end{lemma}
\begin{lemma}
For $\mathbb{I}_{30(b-2)+29}$ with $b\geq 3$ the specials are precisely those CM modules circled below.
{\tiny{
\[
\xymatrix@C=-5pt@R=-4pt{
&.&&.&&.&&.&&1&&0&&0&&1&&0&&1&&1&&0&&1&&1&&1&&1&&1&&1&&1&&2&&1&&1&&2&&1&&2&&2&&1&&2&&2&&2&\\
&&.&&.&&.&&1&&1&&0&&1&&1&&1&&2&&1&&1&&2&&2&&2&&2&&2&&2&&3&&3&&2&&3&&3&&3&&4&&3&&3&&4&&4&&4\\
.&.&.&.&.&.&.&1&1&1&0&1&1&1&0&1&1&2&1&2&1&2&1&2&1&2&1&3&2&3&1&3&2&3&1&3&2&4&2&4&2&4&2&4&2&4&2&5&3&5&2&5&3&5&2&5&3&6&3&6&3\\
.&&.&&.&&1&&0&&1&&1&&1&&1&&2&&1&&2&&2&&2&&2&&3&&2&&3&&3&&3&&3&&4&&3&&4&&4&&4&&4&&5&&4&&5&&5\\
&.&&.&&1&&0&&0&&1&&1&&1&&1&&1&&1&&2&&2&&1&&2&&2&&2&&3&&2&&2&&3&&3&&3&&3&&3&&3&&4&&4&&3&&4&\\
.&&.&&1&&0&&0&&0&&1&&1&&1&&0&&1&&1&&2&&1&&1&&1&&2&&2&&2&&1&&2&&2&&3&&2&&2&&2&&3&&3&&3&&2&&3\\
&.&&1&&0&&0&&0&&0&&1&&1&&0&&0&&1&&1&&1&&1&&0&&1&&2&&1&&1&&1&&1&&2&&2&&1&&1&&2&&2&&2&&2&&1&\\
{{{}\drop\xycircle<4pt,4pt>{}R}}&&1&&0&&0&&0&&0&&{{{}\drop\xycircle<4pt,4pt>{}0}}&&1&&0&&0&&{{{}\drop\xycircle<4pt,4pt>{}0}}&&1&&0&&1&&0&&{{{}\drop\xycircle<4pt,4pt>{}0}}&&1&&1&&0&&1&&0&&1&&1&&1&&0&&1&&1&&1&&1&&1&&{{{}\drop\xycircle<4pt,4pt>{}0}}\\
&&&&&&&&&&&&&&&&&&&&&&&&&&&&&&&&&&&&&&&&&&&&&&
}
\]}}
\end{lemma}
\begin{proof}
$R$ is a distance of $30(b-2)+56$ away from $\tau^{-1}R$, thus by
Lemma~\ref{freeI} we have {\tiny{
\[
\xymatrix@C=-6pt@R=-4pt{
2b\minus 2&&2b\minus 2&&2b\minus 2&&2b\minus 2&&b\minus 1&&2b\minus 2&&2b\minus 1&&b\minus 1&&2b\minus 2&&b&&b\minus 1&&2b\minus 1&&b\\
&4b\minus 4&&4b\minus 4&&4b\minus 4&&3b\minus 3&&3b\minus 3&&4b\minus 3&&3b\minus 2&&3b\minus 3&&3b\minus 2&&2b\minus 1&&3b\minus 2&&3b\minus 1&&2b\minus 1\\
6b\minus 6&3b\minus 3&6b\minus 6&3b\minus 3&6b\minus 6&3b\minus 3&5b\minus 5&2b\minus 2&5b\minus 5&3b\minus 3&5b\minus 4&2b\minus 1&5b\minus 4&3b\minus 3&5b\minus 4&2b\minus 1&4b\minus 3&2b\minus 2&4b\minus 3&2b\minus 1&4b\minus 2&2b\minus 1&4b\minus 2&2b\minus 1&4b\minus 2&2b\minus 1&\\
&5b\minus 5&&5b\minus 5&&4b\minus 4&&5b\minus 5&&4b\minus 3&&4b\minus 4&&4b\minus 3&&4b\minus 3&&3b\minus 2&&4b\minus 3&&3b\minus 1&&3b\minus 2&&3b\minus 1\\
4b\minus 5&&4b\minus 4&&3b\minus 3&&4b\minus 4&&4b\minus 3&&3b\minus 3&&3b\minus 3&&3b\minus 2&&3b\minus 2&&3b\minus 2&&3b\minus 2&&2b\minus 1&&2b\minus 1\\
&3b\minus 4&&2b\minus 2&&3b\minus 3&&3b\minus 2&&3b\minus 3&&2b\minus 2&&2b\minus 2&&2b\minus 1&&3b\minus 2&&2b\minus 1&&2b\minus 2&&b&&2b\minus 1\\
2b\minus 2&&b\minus 2&&2b\minus 2&&2b\minus 1&&2b\minus 2&&2b\minus 2&&b\minus 1&&b\minus 1&&2b\minus 1&&2b\minus 1&&b\minus 1&&b\minus 1&&b\\
&{{{}\drop\xycircle<4pt,4pt>{}R}}&&b\minus 2&&b&&b\minus 1&&b\minus 1&&b\minus 1&&{{{}\drop\xycircle<4pt,4pt>{}0}}&&b\minus 1&&b&&b\minus 1&&{{{}\drop\xycircle<4pt,4pt>{}0}}&&b\minus 1&&1
}
\]
\[
\xymatrix@C=-7pt@R=-4pt{
b\minus 1&&b&&b&&b&&b&&b&&1&&b&&b\plus 1&&1&&b&&2&&1&&b\plus 1\\
&2b\minus 1&&2b&&2b&&2b&&2b&&b\plus 1&&b\plus 1&&2b\plus 1&&b\plus 2&&b\plus 1&&b\plus 2&&3&&b\plus 2&\\
3b\minus 1&b&3b\minus 1&2b\minus 1&3b&b\plus 1&3b&2b\minus 1&3b&b\plus 1&2b\plus 1&b&2b\plus 1&b\plus 1&2b\plus 2&b\plus 1&2b\plus 2&b\plus 1&2b\plus 2&b\plus 1&b\plus 3&2&b\plus 3&b\plus 1&b\plus 4&3&b\plus 4\\
&3b\minus 1&&2b&&3b\minus 1&&2b\plus 1&&2b&&2b\plus 1&&2b\plus 1&&b\plus 2&&2b\plus 1&&b\plus 3&&b\plus 2&&b\plus 3&&b\plus 3&\\
3b\minus 1&&2b&&2b\minus 1&&2b&&b\plus 1&&2b&&2b\plus 1&&b\plus 1&&b\plus 1&&b\plus 2&&b\plus 2&&b\plus 2&&b\plus 2&&3\\
&2b&&2b\minus 1&&b&&b&&b\plus 1&&2b&&b\plus 1&&b&&2&&b\plus 1&&b\plus 2&&b\plus 1&&2&\\
b&&2b\minus 1&&b&&0&&b&&b\plus 1&&b&&b&&1&&1&&b\plus 1&&b\plus 1&&1&&1\\
&b\minus 1&&b&&{{{}\drop\xycircle<4pt,4pt>{}0}}&&0&&b&&1&&b\minus 1&&1&&0&&1&&b&&1&&0&
}
\]
\[
\begin{array}{c}
\xymatrix@C=-5pt@R=-4pt{
&2&&1&&2&&2&&2&&2&&2&&1&&2&&3&&1&&2&&2&&1&&3&&2&&1&&2&&2&\\
b\plus 3&&3&&3&&4&&4&&4&&4&&3&&3&&5&&4&&3&&4&&3&&4&&5&&3&&3&&4&&4\\
b\plus 1&b\plus 4&3&5&2&5&3&6&3&6&3&6&3&5&2&5&3&6&3&6&3&6&3&5&2&5&3&6&3&6&3&6&3&5&2&5&3&6&3\\
4&&b\plus 3&&5&&4&&5&&5&&4&&5&&5&&4&&5&&5&&4&&5&&5&&4&&5&&5&&4&&5\\
&3&&b\plus 3&&4&&3&&4&&3&&4&&5&&3&&3&&4&&4&&4&&4&&3&&3&&5&&4&&3&\\
2&&3&&b\plus 2&&3&&2&&2&&3&&4&&3&&2&&2&&3&&4&&3&&2&&2&&3&&4&&3&&2\\
&2&&2&&b\plus 1&&2&&0&&2&&3&&2&&2&&1&&1&&3&&3&&1&&1&&2&&2&&3&&2\\
1&&1&&1&&b&\ar@{.}@<1ex>[-7,0]&{{{}\drop\xycircle<4pt,4pt>{}0}}&&0&&2&&1&&1&&1&&0&&1&&2&&1&&0&&1&&1&&1&&2&\ar@{.}@<1ex>[-7,0]&0
}
\end{array}
...
\]
}}where the dotted segment repeats until it reaches $R$ as:
{\tiny{
\[
...
\begin{array}{c}
\xymatrix@C=-4pt@R=-4pt{
2&&2&&2&&2&&0&&1&&2&&1&&1&&0&&1&&1&&0&\\
&4&&4&&4&&2&&1&&3&&3&&2&&1&&1&&2&&1&&0  \\
5&3&6&3&6&3&4&1&3&2&3&1&4&3&4&1&2&1&2&1&2&1&2&1&1&0&0\\
&4&&5&&3&&4&&3&&3&&2&&3&&2&&2&&2&&1&&1&&0\\
4&&3&&2&&3&&4&&3&&1&&1&&3&&2&&2&&1&&1&&1&&0&\\
&3&&0&&2&&3&&4&&1&&0&&1&&3&&2&&1&&0&&1&&1&&0&\\
3&&0&&0&&2&&3&&2&&0&&0&&1&&3&&1&&0&&0&&1&&1&&0\\
&{{{}\drop\xycircle<4pt,4pt>{}R}}&&0&&0&&2&&1&&1&&{{{}\drop\xycircle<4pt,4pt>{}0}}&&0&&1&&2&&{{{}\drop\xycircle<4pt,4pt>{}0}}&&0&&0&&1&&0&&{{{}\drop\xycircle<4pt,4pt>{}0}}}
\end{array}
\]
}}completing the classification.
\end{proof}

\section{Summary of the Classification}
In type $\mathbb{O}$ denote
{\tiny{
\[
\begin{array}{rl}
\begin{array}{c}
\xymatrix@C=0pt@R=2pt{
\cdot&&\cdot&&\cdot&&{}\save[]*{D_1}\restore&&\cdot&&\cdot&&{}\save[]*{F}\restore&&\cdot&&\cdot&&{}\save[]*{D_3}\restore&&\cdot&&\cdot&&\cdot&\\
&\cdot&&\cdot&&\cdot&&\cdot&&\cdot&&\cdot&&\cdot&&\cdot&&\cdot&&\cdot&&\cdot&&\cdot&&\cdot\\
\cdot&&\cdot&&\cdot&&\cdot&&\cdot&&\cdot&&\cdot&&\cdot&&\cdot&&\cdot&&\cdot&&\cdot&&\cdot&\\
&\cdot&\cdot&\cdot&\cdot&\cdot&\cdot&\cdot&\cdot&\cdot&\cdot&\cdot&\cdot&\cdot&\cdot&\cdot&\cdot&\cdot&\cdot&\cdot&\cdot&\cdot&\cdot&\cdot&\cdot&\cdot\\
\cdot&&\cdot&&\cdot&&\cdot&&\cdot&&\cdot&&\cdot&&\cdot&&\cdot&&\cdot&&\cdot&&\cdot&&\cdot&\\
&\cdot&&\cdot&&\cdot&&\cdot&&\cdot&&\cdot&&\cdot&&\cdot&&\cdot&&\cdot&&\cdot&&\cdot&&\cdot&\\
{}\save[]*{R}\restore&&\cdot&&\cdot&&\cdot&&{}\save[]*{E_1}\restore&&\cdot&&{}\save[]*{D_2}\restore&&\cdot&&{}\save[]*{E_2}\restore&&\cdot&&\cdot&&\cdot&&{}\save[]*{N}\restore}
\end{array}
& \hdots
\end{array}
\]
}}and in type $\mathbb{I}$ denote
{\tiny{
\[
\begin{array}{rl}
\begin{array}{c}
\xymatrix@C=-3pt@R=1pt{
&\cdot&&\cdot&&\cdot&&\cdot&&\cdot&&\cdot&&\cdot&&\cdot&&\cdot&&\cdot&&\cdot&&\cdot&&\cdot&&\cdot&&\cdot&&\cdot&&\cdot&&\cdot&&\cdot&&\cdot&&\cdot&&\cdot&&\cdot&&\cdot&&\cdot&&\cdot&&\cdot&&\cdot&&\cdot&&\cdot \\
\cdot&&\cdot&&\cdot&&\cdot&&\cdot&&\cdot&&\cdot&&\cdot&&\cdot&&\cdot&&\cdot&&\cdot&&\cdot&&\cdot&&\cdot&&\cdot&&\cdot&&\cdot&&\cdot&&\cdot&&\cdot&&\cdot&&\cdot&&\cdot&&\cdot&&\cdot&&\cdot&&\cdot&&\cdot&&\cdot&&\cdot \\
\cdot&\cdot&\cdot&\cdot&\cdot&\cdot&\cdot&\cdot&\cdot&\cdot&\cdot&\cdot&\cdot&\cdot&\cdot&\cdot&\cdot&\cdot&\cdot&\cdot&\cdot&\cdot&\cdot&\cdot&\cdot&\cdot&\cdot&\cdot&\cdot&\cdot&\cdot&\cdot&\cdot&\cdot&\cdot&\cdot&\cdot&\cdot&\cdot&\cdot&\cdot&\cdot&\cdot&\cdot&\cdot&\cdot&\cdot&\cdot&\cdot&\cdot&\cdot&\cdot&\cdot&\cdot&\cdot&\cdot&\cdot&\cdot&\cdot&\cdot&\cdot\\
\cdot&&\cdot&&\cdot&&\cdot&&\cdot&&\cdot&&\cdot&&\cdot&&\cdot&&\cdot&&\cdot&&\cdot&&\cdot&&\cdot&&\cdot&&\cdot&&\cdot&&\cdot&&\cdot&&\cdot&&\cdot&&\cdot&&\cdot&&\cdot&&\cdot&&\cdot&&\cdot&&\cdot&&\cdot&&\cdot&&\cdot \\
&\cdot&&\cdot&&\cdot&&\cdot&&\cdot&&\cdot&&\cdot&&\cdot&&\cdot&&\cdot&&\cdot&&\cdot&&\cdot&&\cdot&&\cdot&&\cdot&&\cdot&&\cdot&&\cdot&&\cdot&&\cdot&&\cdot&&\cdot&&\cdot&&\cdot&&\cdot&&\cdot&&\cdot&&\cdot&&\cdot& \\
\cdot&&\cdot&&\cdot&&\cdot&&\cdot&&\cdot&&\cdot&&\cdot&&\cdot&&\cdot&&\cdot&&\cdot&&\cdot&&\cdot&&\cdot&&\cdot&&\cdot&&\cdot&&\cdot&&\cdot&&\cdot&&\cdot&&\cdot&&\cdot&&\cdot&&\cdot&&\cdot&&\cdot&&\cdot&&\cdot&&\cdot \\
&\cdot&&\cdot&&\cdot&&\cdot&&\cdot&&\cdot&&\cdot&&\cdot&&\cdot&&\cdot&&\cdot&&\cdot&&\cdot&&\cdot&&\cdot&&\cdot&&\cdot&&\cdot&&\cdot&&\cdot&&\cdot&&\cdot&&\cdot&&\cdot&&\cdot&&\cdot&&\cdot&&\cdot&&\cdot&&\cdot \\
{}\save[]*{R}\restore&&\cdot&&\cdot&&\cdot&&\cdot&&\cdot&&{}\save[]*{A_1}\restore&&\cdot&&\cdot&&\cdot&&{}\save[]*{B_1}\restore&&\cdot&&{}\save[]*{A_2}\restore&&\cdot&&\cdot&&{}\save[]*{C}\restore&&\cdot&&\cdot&&{}\save[]*{A_3}\restore&&\cdot&&{}\save[]*{B_2}\restore&&\cdot&&\cdot&&\cdot&&{}\save[]*{A_4}\restore&&\cdot&&\cdot&&\cdot&&\cdot&&\cdot&&{}\save[]*{M}\restore}\end{array} & \hdots
\end{array}
\]
}}In the following theorem we include the description of the dual graph of the minimal resolution for completeness; the classification of the dual graphs is due to Brieskorn \cite[2.11]{Brieskorn}.   We also include the fundamental cycle $Z_f$ since the rank of an indecomposable special CM module coincides with the co-efficient of the corresponding exceptional curve in $Z_f$.
\begin{thm}
Denote by $[\alpha_1,\hdots,\alpha_N]$ the continued fraction expansion of $\frac{n}{q}$.  Then with notation as before the specials for every small finite subgroup of $GL(2,\C{})$ are as follows:
{\scriptsize{
\[
\begin{array}{cccc}
\t{group} & \t{specials} &\t{dual graph} & Z_f\\ &&&\\
\mathbb{A}_{n,q} & S_{i_1},S_{i_2},\hdots,S_{i_N} & \begin{array}{c}\tiny{\xymatrix@C=-2pt@R=-2pt{\alpha_1&\alpha_2&...&\alpha_N}}\end{array} & \begin{array}{c}\tiny{\xymatrix@C=-2pt@R=-2pt{1&1&...&1}}\end{array} \\
&&&\\
\mathbb{D}_{n,q}\,\, {}_{(n>2q)} & W_+,W_-,W_{i_1},\hdots,W_{i_N} & \begin{array}{c}\tiny{\xymatrix@C=-2pt@R=-2pt{&2&&\\ 2&\alpha_1&\alpha_2&...&\alpha_N}}\end{array} &  \begin{array}{c}\tiny{\xymatrix@C=-2pt@R=-2pt{&1&&\\ 1&1&...&1}}\end{array} \\ &&&\\
\mathbb{D}_{n,q}\,\, {}_{(n<2q)}&\begin{array}{c}W_+,W_-,W_{i_{\nu+1}},\hdots,W_{i_N}\\ V_{i_{\nu+1}+s(n-q)}\,\,\t{for all }\,\,0\leq s\leq \nu-1 \end{array}& \begin{array}{c}\tiny{\xymatrix@C=-2pt@R=-2pt{&2&&\\ 2&\alpha_1&\alpha_2&...&\alpha_N}}\end{array}&\begin{array}{c}\tiny{\xymatrix@C=-2pt@R=-2pt{&1&&\\ 1&2&..&2&1&..&1}}\end{array}\\
\mathbb{T}_1=E_6 & \t{all CM modules} & \begin{array}{c}\tiny{\xymatrix@C=-2pt@R=-2pt{&&2&&\\2&2&2&2&2}}\end{array} & \begin{array}{c}\tiny{\xymatrix@C=-2pt@R=-2pt{&&2&&\\1&2&3&2&1}}\end{array} \\ &&&\\
\mathbb{T}_{6(b-2)+1} & {\tiny{
\begin{array}{c}
\xymatrix@C=-2pt@R=0pt{
.&&.&&{{{}\drop\xycircle<4pt,4pt>{}.}}&&.&&{{{}\drop\xycircle<4pt,4pt>{}.}}&&.&&. \\
&.&&.&&.&&.&&.&&.\\
.&.&.&.&.&.&.&.&.&.&.&.&.\\
R&.&.&.&.&.&{{{}\drop\xycircle<4pt,4pt>{}.}}&.&.&.&.&.&{{{}\drop\xycircle<4pt,4pt>{}.}}\\
.&&.&&{{{}\drop\xycircle<4pt,4pt>{}.}}&&.&&{{{}\drop\xycircle<4pt,4pt>{}.}}&&.&&.}
\end{array}}}
 &\begin{array}{c}\tiny{\xymatrix@C=-2pt@R=-2pt{&&2&&\\2&2&b&2&2}}\end{array}& \begin{array}{c}\tiny{\xymatrix@C=-2pt@R=-2pt{&&1&&\\1&1&1&1&1}}\end{array}\\ &&&\\
 \mathbb{T}_3& {\tiny{
\begin{array}{c}
\xymatrix@C=-2pt@R=0pt{
{{}\drop\xycircle<4pt,4pt>{}.}&&.&&{{}\drop\xycircle<4pt,4pt>{}.}&&.\\
&.&&.&&.&\\
.&{{}\drop\xycircle<4pt,4pt>{}.}&.&.&.&.&.\\
R&.&.&{{}\drop\xycircle<4pt,4pt>{}.}&.&.&{{}\drop\xycircle<4pt,4pt>{}.}\\
.&&.&&{{}\drop\xycircle<4pt,4pt>{}.}&&R}
\end{array}}}&\begin{array}{c}\tiny{\xymatrix@C=-2pt@R=-2pt{&2&&\\3&2&2&2}}\end{array}&\begin{array}{c}\tiny{\xymatrix@C=-2pt@R=-2pt{&1&&\\1&2&2&1}}\end{array}\\ &&&\\
 \mathbb{T}_{6(b-2)+3}&{\tiny{
\begin{array}{c}
\xymatrix@C=-2pt@R=0pt{
.&&.&&{{{}\drop\xycircle<4pt,4pt>{}.}}&&.&&.&&.&&. \\
&.&&.&&.&&.&&.&&.\\
.&.&.&.&.&.&.&.&.&.&.&.&.\\
R&.&.&.&.&.&{{{}\drop\xycircle<4pt,4pt>{}.}}&.&.&.&.&.&{{{}\drop\xycircle<4pt,4pt>{}.}}\\
.&&.&&{{{}\drop\xycircle<4pt,4pt>{}.}}&&.&&{{{}\drop\xycircle<4pt,4pt>{}.}}&&.&&.}
\end{array}}}&\begin{array}{c}\tiny{\xymatrix@C=-2pt@R=-2pt{&2&&\\3&b&2&2}}\end{array}&\begin{array}{c}\tiny{\xymatrix@C=-2pt@R=-2pt{&1&&\\1&1&1&1}}\end{array}
\end{array}
\]
\[
\begin{array}{cccc}
  \mathbb{T}_5&{\tiny{
\begin{array}{c}
\xymatrix@C=-2pt@R=0pt{
.&&.&&{{}\drop\xycircle<4pt,4pt>{}.}&&.&&.&&. \\
&.&&.&&.&&.&&.\\
.&{{}\drop\xycircle<4pt,4pt>{}.}&.&.&.&.&.&.&.&.&.&\\
R&.&.&.&.&.&{{}\drop\xycircle<4pt,4pt>{}.}&.&.&.&R\\
.&&.&&{{}\drop\xycircle<4pt,4pt>{}.}&&.&&.&&.}
\end{array}
}}&\begin{array}{c}\tiny{\xymatrix@C=-2pt@R=-2pt{&2&\\3&2&3}}\end{array}&\begin{array}{c}\tiny{\xymatrix@C=-2pt@R=-2pt{&1&\\1&2&1}}\end{array}\\ &&&\\
 \mathbb{T}_{6(b-2)+5}&{\tiny{
\begin{array}{c}
\xymatrix@C=-2pt@R=0pt{
.&&.&&{{{}\drop\xycircle<4pt,4pt>{}.}}&&.&&.&&.&&. \\
&.&&.&&.&&.&&.&&.\\
.&.&.&.&.&.&.&.&.&.&.&.&.\\
R&.&.&.&.&.&{{{}\drop\xycircle<4pt,4pt>{}.}}&.&.&.&.&.&{{{}\drop\xycircle<4pt,4pt>{}.}}\\
.&&.&&{{{}\drop\xycircle<4pt,4pt>{}.}}&&.&&.&&.&&.}
\end{array}}} &\begin{array}{c}\tiny{\xymatrix@C=-2pt@R=-2pt{&2&\\3&b&3}}\end{array}&\begin{array}{c}\tiny{\xymatrix@C=-2pt@R=-2pt{&1&\\1&1&1}}\end{array}\\
&&&\\
\mathbb{O}_1=E_7&\t{all CM modules}&\begin{array}{c}\tiny{\xymatrix@C=-2pt@R=-2pt{&&2&&&\\2&2&2&2&2&2}}\end{array}&\begin{array}{c}\tiny{\xymatrix@C=-2pt@R=-2pt{&&2&&&\\ 2&3&4&3&2&1}}\end{array}\\
&&&\\ 
\mathbb{O}_{12(b-2)+1}&D_1,D_2,D_3,E_1,E_2,F,N&\begin{array}{c}\tiny{\xymatrix@C=-2pt@R=-2pt{&&2&&&\\2&2&b&2&2&2}}\end{array}
&\begin{array}{c}\tiny{\xymatrix@C=-2pt@R=-2pt{&&1&&&\\1&1&1&1&1&1}}\end{array}\\
&&&\\ 
\mathbb{O}_5&{\tiny{\begin{array}{c}
\xymatrix@C=-2pt@R=0pt{
.&&{{}\drop\xycircle<4pt,4pt>{}.}&&.&&{{}\drop\xycircle<4pt,4pt>{}.}&&.&&.\\
&.&&.&&.&&.&&.\\
.&&.&&.&&.&&.&&.\\
.&.&.&.&{{}\drop\xycircle<4pt,4pt>{}.}&.&.&.&.&.&.\\
.&&.&&.&&.&&.&&.\\
&{{}\drop\xycircle<4pt,4pt>{}.}&&.&&.&&{{}\drop\xycircle<4pt,4pt>{}.}&&. \\
R&&.&&.&&.&&{{}\drop\xycircle<4pt,4pt>{}.}&&R}\end{array}}}&\begin{array}{c}\tiny{\xymatrix@C=-2pt@R=-2pt{&2&&&\\3&2&2&2&2}}\end{array} &\begin{array}{c}\tiny{\xymatrix@C=-2pt@R=-2pt{&1&&&\\1&2&2&2&1}}\end{array}\\
&&&\\ 
\mathbb{O}_{12(b-2)+5}&D_1,D_2,D_3,E_1,F,N&\begin{array}{c}\tiny{\xymatrix@C=-2pt@R=-2pt{&2&&&\\3&b&2&2&2}}\end{array}&\begin{array}{c}\tiny{\xymatrix@C=-2pt@R=-2pt{&1&&&\\1&1&1&1&1}}\end{array}\\
&&&\\
\mathbb{O}_7&{\tiny{\begin{array}{c}
\xymatrix@C=-3pt@R=-0pt{
.&&.&&.&&{{}\drop\xycircle<4pt,4pt>{}.}&&.&&.&&{{}\drop\xycircle<4pt,4pt>{}.}&&.&\\
&.&&.&&{{}\drop\xycircle<4pt,4pt>{}.}&&.&&.&&.&&.&\\
.&&.&&.&&.&&.&&.&&.&&.&\\
.&.&.&.&.&.&.&.&.&.&.&.&.&.&.\\
.&&.&&.&&.&&.&&.&&.&&.&\\
&{{}\drop\xycircle<4pt,4pt>{}.}&&.&&.&&.&&.&&.&&. \\
R&&.&&.&&.&&{{}\drop\xycircle<4pt,4pt>{}.}&&.&&.&&R&}\end{array}}}&\begin{array}{c}\tiny{\xymatrix@C=-2pt@R=-2pt{&2&&\\4&2&2&2}}\end{array}&\begin{array}{c}\tiny{\xymatrix@C=-2pt@R=-2pt{&1&&\\1&2&2&1}}\end{array}\\
&&&\\ 
\mathbb{O}_{12(b-2)+7}&D_1,E_1,E_2,F,N&\begin{array}{c}\tiny{\xymatrix@C=-2pt@R=-2pt{&2&&\\4&b&2&2}}\end{array}&\begin{array}{c}\tiny{\xymatrix@C=-2pt@R=-2pt{&1&&\\1&1&1&1}}\end{array}\\
&&&\\
\mathbb{O}_{11}&{\tiny{\begin{array}{c}
\xymatrix@C=-3pt@R=0pt{
.&&.&&.&&{{}\drop\xycircle<4pt,4pt>{}.}&&.&&.&&{{}\drop\xycircle<4pt,4pt>{}.}&&.&&.&&.&&.&&.\\
&.&&.&&.&&.&&.&&.&&.&&.&&.&&.&&.&\\
.&&.&&.&&.&&.&&.&&.&&.&&.&&.&&.&&.\\
.&.&.&.&.&.&.&.&.&.&.&.&.&.&.&.&.&.&.&.&.&.&.\\
.&&.&&.&&.&&.&&.&&.&&.&&.&&.&&.&&.\\
&{{}\drop\xycircle<4pt,4pt>{}.}&&.&&.&&.&&.&&.&&.&&.&&.&&.&&.&\\
R&&.&&.&&.&&{{}\drop\xycircle<4pt,4pt>{}.}&&.&&.&&.&&.&&.&&.&&R}\end{array}}}&\begin{array}{c}\tiny{\xymatrix@C=-2pt@R=-2pt{&2&\\3&2&4}}\end{array} &\begin{array}{c}\tiny{\xymatrix@C=-2pt@R=-2pt{&1&\\1&2&1}}\end{array}\\
&&&\\ 
\mathbb{O}_{12(b-2)+11}&D_1,E_1,F,N&\begin{array}{c}\tiny{\xymatrix@C=-2pt@R=-2pt{&2&\\3&b&4}}\end{array}&\begin{array}{c}\tiny{\xymatrix@C=-2pt@R=-2pt{&1&\\1&1&1}}\end{array}\\
&&&\\ 
\mathbb{I}_1=E_8&\t{all CM modules}&\begin{array}{c}\tiny{\xymatrix@C=-2pt@R=-2pt{&&2&&&&\\2&2&2&2&2&2&2}}\end{array}&\begin{array}{c}\tiny{\xymatrix@C=-2pt@R=-2pt{&&3&&&\\ 2&4&6&5&4&3&2}}\end{array}\\
&&&\\
\mathbb{I}_{30(b-2)+1}&A_1,A_2,A_3,A_4,B_1,B_2,C,M&\begin{array}{c}\tiny{\xymatrix@C=-2pt@R=-2pt{&&2&&&&\\2&2&b&2&2&2&2}}\end{array}&\begin{array}{c}\tiny{\xymatrix@C=-2pt@R=-2pt{&&1&&&\\1&1&1&1&1&1&1}}\end{array}\\
&&&\\
\mathbb{I}_{7}&{\tiny{\begin{array}{c}
\xymatrix@C=-2pt@R=0pt{
&.&&.&&.&&.&&.&&.&&{{}\drop\xycircle<4pt,4pt>{}.}& \\
.&&.&&.&&.&&.&&.&&.&&.&\\
.&.&.&.&.&.&{{}\drop\xycircle<4pt,4pt>{}.}&.&.&.&.&.&.&.&.\\
.&&.&&.&&.&&.&&.&&.&&.& \\
&.&&.&&.&&.&&.&&.&&.&\\
.&&.&&.&&.&&.&&.&&.&&.\\
&{{}\drop\xycircle<4pt,4pt>{}.}&&.&&.&&.&&.&&{{}\drop\xycircle<4pt,4pt>{}.}&&.&  \\
R&&.&&.&&{{}\drop\xycircle<4pt,4pt>{}.}&&.&&.&&{{}\drop\xycircle<4pt,4pt>{}.}&&R}\end{array}}}&\begin{array}{c}\tiny{\xymatrix@C=-2pt@R=-2pt{&&2&&\\2&2&2&2&3}}\end{array}&\begin{array}{c}\tiny{\xymatrix@C=-2pt@R=-2pt{&&2&&\\1&2&3&2&1}}\end{array}\\
&&&\\
\mathbb{I}_{30(b-2)+7}&A_1,A_3,B_1,B_2,C,M&\begin{array}{c}\tiny{\xymatrix@C=-2pt@R=-2pt{&&2&&\\2&2&b&2&3}}\end{array}&\begin{array}{c}\tiny{\xymatrix@C=-2pt@R=-2pt{&&1&&\\1&1&1&1&1}}\end{array}\\
&&&\\
\mathbb{I}_{11}&{\tiny{
\begin{array}{c}
\xymatrix@C=-4pt@R=0pt{
&.&&.&&.&&{{}\drop\xycircle<4pt,4pt>{}.}&&.&&.&&{{}\drop\xycircle<4pt,4pt>{}.}&&.&&.&&.&&.&\\
.&&.&&.&&.&&.&&.&&.&&.&&.&&.&&.&&.\\
.&.&.&.&.&.&.&.&.&.&.&.&.&.&.&.&.&.&.&.&.&.&.\\
.&&.&&.&&.&&.&&.&&.&&.&&.&&.&&.&&.&&\\
&.&&.&&.&&.&&.&&.&&.&&.&&.&&.&&.&\\
.&&.&&.&&.&&.&&.&&.&&.&&.&&.&&.&&.\\
&{{}\drop\xycircle<4pt,4pt>{}.}&&.&&.&&.&&.&&.&&.&&.&&.&&{{}\drop\xycircle<4pt,4pt>{}.}&&.\\
R&&.&&.&&.&&{{}\drop\xycircle<4pt,4pt>{}.}&&.&&{{}\drop\xycircle<4pt,4pt>{}.}&&.&&.&&.&&{{}\drop\xycircle<4pt,4pt>{}.}&&R}
\end{array}}}
&\begin{array}{c}\tiny{\xymatrix@C=-2pt@R=-2pt{&2&&&\\3&2&2&2&2&2}}\end{array}&\begin{array}{c}\tiny{\xymatrix@C=-2pt@R=-2pt{&1&&&\\1&2&2&2&2&1}}\end{array}\\
&&&\\
\mathbb{I}_{30(b-2)+11}&A_1,A_2,A_3,A_4,B_1,C,M&\begin{array}{c}\tiny{\xymatrix@C=-2pt@R=-2pt{&2&&&\\3&b&2&2&2&2}}\end{array}&\begin{array}{c}\tiny{\xymatrix@C=-2pt@R=-2pt{&1&&&\\1&1&1&1&1&1}}\end{array}
\end{array}
\]
\[
\begin{array}{cccc}
\mathbb{I}_{13}&{\tiny{
\begin{array}{c}
\xymatrix@C=-4pt@R=0pt{
&.&&.&&.&&{{}\drop\xycircle<4pt,4pt>{}.}&&.&&.&&.&&.&&{{}\drop\xycircle<4pt,4pt>{}.}&&.&&.&&.&&.&\\
.&&.&&.&&.&&.&&.&&.&&.&&.&&.&&.&&.&&.&&.\\
.&.&.&.&.&.&.&.&.&.&.&.&.&.&.&.&.&.&.&.&.&.&.&.&.&.&.\\
.&&.&&.&&.&&.&&.&&.&&.&&.&&.&&.&&.&&.&&.\\
&.&&.&&.&&.&&.&&.&&.&&.&&.&&.&&.&&.&&.&\\
.&&.&&.&&.&&.&&.&&.&&.&&.&&.&&.&&.&&.&&.\\
&.&&.&&.&&.&&.&&.&&.&&.&&.&&.&&.&&.&&.&\\
R&&.&&{{}\drop\xycircle<4pt,4pt>{}.}&&.&&.&&.&&{{}\drop\xycircle<4pt,4pt>{}.}&&.&&.&&.&&{{}\drop\xycircle<4pt,4pt>{}.}&&.&&{{}\drop\xycircle<4pt,4pt>{}.}&&R}
\end{array}}}&\begin{array}{c}\tiny{\xymatrix@C=-2pt@R=-2pt{&&2&&\\2&2&2&3&2}}\end{array}&\begin{array}{c}\tiny{\xymatrix@C=-2pt@R=-2pt{&&1&&\\1&2&2&1&1}}\end{array}\\
&&&\\
\mathbb{I}_{30(b-2)+13}&A_1,A_2,B_1,B_2,C,M&\begin{array}{c}\tiny{\xymatrix@C=-2pt@R=-2pt{&&2&&\\2&2&b&3&2}}\end{array}&\begin{array}{c}\tiny{\xymatrix@C=-2pt@R=-2pt{&&1&&\\1&1&1&1&1}}\end{array}\\
&&&\\
\mathbb{I}_{17}&{\tiny{
\begin{array}{c}
\xymatrix@C=-4pt@R=0pt{
&.&&.&&.&&.&&.&&.&&{{}\drop\xycircle<4pt,4pt>{}.}&&.&&.&&.&&.&&.&&.&&.&&.&&.&&.\\
.&&.&&.&&.&&.&&.&&.&&.&&.&&.&&.&&.&&.&&.&&.&&.&&.&&.\\
.&.&.&.&.&.&.&.&.&.&.&.&.&.&.&.&.&.&.&.&.&.&.&.&.&.&.&.&.&.&.&.&.&.&.\\
.&&.&&.&&.&&.&&.&&.&&.&&.&&.&&.&&.&&.&&.&&.&&.&&.&&.\\
&.&&.&&.&&.&&.&&.&&.&&.&&.&&.&&.&&.&&.&&.&&.&&.&&.&\\
.&&.&&.&&.&&.&&.&&.&&.&&.&&.&&.&&.&&.&&.&&.&&.&&.&&.\\
&{{}\drop\xycircle<4pt,4pt>{}.}&&.&&.&&.&&.&&.&&.&&.&&.&&.&&.&&.&&.&&.&&.&&.&&.&\\
R&&.&&.&&.&&.&&.&&{{}\drop\xycircle<4pt,4pt>{}.}&&.&&.&&.&&{{}\drop\xycircle<4pt,4pt>{}.}&&.&&.&&.&&.&&{{}\drop\xycircle<4pt,4pt>{}.}&&.&&R}
\end{array}}}&\begin{array}{c}\tiny{\xymatrix@C=-2pt@R=-2pt{&2&&\\3&2&2&3}}\end{array}&\begin{array}{c}\tiny{\xymatrix@C=-2pt@R=-2pt{&1&&\\1&2&2&1}}\end{array}\\
&&&\\
\mathbb{I}_{30(b-2)+17}&A_1,A_3,B_1,C,M&\begin{array}{c}\tiny{\xymatrix@C=-2pt@R=-2pt{&2&&\\3&b&2&3}}\end{array}&\begin{array}{c}\tiny{\xymatrix@C=-2pt@R=-2pt{&1&&\\1&1&1&1}}\end{array}\\ 
&&&\\
\mathbb{I}_{19}&{\tiny{
\begin{array}{c}
\xymatrix@C=-4.5pt@R=0pt{
&.&&.&&.&&.&&.&&.&&.&&.&&.&&.&&.&&.&&.&&.&&.&&.&&.&&.&&.&\\
.&&.&&.&&.&&.&&.&&.&&.&&.&&.&&.&&.&&.&&.&&.&&.&&.&&.&&.&&.\\
.&.&.&.&.&.&.&.&.&.&.&.&.&.&.&.&.&.&.&.&.&.&.&.&.&.&.&.&.&.&.&.&.&.&.&.&.&.&.\\
.&&.&&.&&.&&.&&.&&.&&.&&.&&.&&.&&.&&.&&.&&.&&.&&.&&.&&.&&.\\
&.&&.&&.&&.&&.&&.&&.&&.&&.&&.&&.&&.&&.&&.&&.&&.&&.&&.&&.&\\
.&&.&&.&&.&&.&&.&&.&&.&&.&&.&&.&&.&&.&&.&&.&&.&&.&&.&&.&&.\\
&{{}\drop\xycircle<4pt,4pt>{}.}&&.&&.&&.&&.&&{{}\drop\xycircle<4pt,4pt>{}.}&&.&&.&&.&&.&&.&&.&&.&&.&&.&&.&&.&&.&&.&\\
R&&.&&.&&.&&.&&.&&{{}\drop\xycircle<4pt,4pt>{}.}&&.&&.&&.&&{{}\drop\xycircle<4pt,4pt>{}.}&&.&&.&&.&&.&&{{}\drop\xycircle<4pt,4pt>{}.}&&.&&.&&.&&R}
\end{array}}}
&\begin{array}{c}\tiny{\xymatrix@C=-2pt@R=-2pt{&2&&\\5&2&2&2}}\end{array}&\begin{array}{c}\tiny{\xymatrix@C=-2pt@R=-2pt{&1&&\\1&2&2&1}}\end{array}\\
&&&\\
\mathbb{I}_{30(b-2)+19}&A_1,B_1,B_2,C,M&\begin{array}{c}\tiny{\xymatrix@C=-2pt@R=-2pt{&2&&\\5&b&2&2}}\end{array}&\begin{array}{c}\tiny{\xymatrix@C=-2pt@R=-2pt{&1&&\\1&1&1&1}}\end{array}
\\
&&&\\
\mathbb{I}_{23}&{\tiny{
\begin{array}{c}
\xymatrix@C=-5pt@R=0pt{
&.&&.&&.&&{{}\drop\xycircle<4pt,4pt>{}.}&&.&&.&&.&&.&&.&&.&&.&&.&&.&&.&&.&&.&&.&&.&&.&&.&&.&&.&&.\\
.&&.&&.&&.&&.&&.&&.&&.&&.&&.&&.&&.&&.&&.&&.&&.&&.&&.&&.&&.&&.&&.&&.&&.\\
.&.&.&.&.&.&.&.&.&.&.&.&.&.&.&.&.&.&.&.&.&.&.&.&.&.&.&.&.&.&.&.&.&.&.&.&.&.&.&.&.&.&.&.&.&.&.\\
.&&.&&.&&.&&.&&.&&.&&.&&.&&.&&.&&.&&.&&.&&.&&.&&.&&.&&.&&.&&.&&.&&.&&.\\
&.&&.&&.&&.&&.&&.&&.&&.&&.&&.&&.&&.&&.&&.&&.&&.&&.&&.&&.&&.&&.&&.&&.&\\
.&&.&&.&&.&&.&&.&&.&&.&&.&&.&&.&&.&&.&&.&&.&&.&&.&&.&&.&&.&&.&&.&&.&&.\\
&.&&.&&.&&.&&.&&.&&.&&.&&.&&.&&.&&.&&.&&.&&.&&.&&.&&.&&.&&.&&.&&.&&.\\
R&&.&&.&&.&&.&&.&&{{}\drop\xycircle<4pt,4pt>{}.}&&.&&.&&.&&{{}\drop\xycircle<4pt,4pt>{}.}&&.&&{{}\drop\xycircle<4pt,4pt>{}.}&&.&&.&&{{}\drop\xycircle<4pt,4pt>{}.}&&.&&.&&.&&.&&.&&.&&.&&R}
\end{array}}}&\begin{array}{c}\tiny{\xymatrix@C=-2pt@R=-2pt{&2&&\\3&2&3&2}}\end{array}&\begin{array}{c}\tiny{\xymatrix@C=-2pt@R=-2pt{&1&&\\1&2&1&1}}\end{array}
\\
&&&\\
\mathbb{I}_{30(b-2)+23}&A_1,A_2,B_1,C,M&\begin{array}{c}\tiny{\xymatrix@C=-2pt@R=-2pt{&2&&\\3&b&3&2}}\end{array}&\begin{array}{c}\tiny{\xymatrix@C=-2pt@R=-2pt{&1&&\\1&1&1&1}}\end{array}
\\
&&&\\
\mathbb{I}_{29}&{\tiny{
\begin{array}{c}
\xymatrix@C=-6pt@R=0pt{
&.&&.&&.&&.&&.&&.&&.&&.&&.&&.&&.&&.&&.&&.&&.&&.&&.&&.&&.&&.&&.&&.&&.&&.&&.&&.&&.&&.&&.&\\
.&&.&&.&&.&&.&&.&&.&&.&&.&&.&&.&&.&&.&&.&&.&&.&&.&&.&&.&&.&&.&&.&&.&&.&&.&&.&&.&&.&&.&&.\\
.&.&.&.&.&.&.&.&.&.&.&.&.&.&.&.&.&.&.&.&.&.&.&.&.&.&.&.&.&.&.&.&.&.&.&.&.&.&.&.&.&.&.&.&.&.&.&.&.&.&.&.&.&.&.&.&.&.&.\\
.&&.&&.&&.&&.&&.&&.&&.&&.&&.&&.&&.&&.&&.&&.&&.&&.&&.&&.&&.&&.&&.&&.&&.&&.&&.&&.&&.&&.&&.&\\
&.&&.&&.&&.&&.&&.&&.&&.&&.&&.&&.&&.&&.&&.&&.&&.&&.&&.&&.&&.&&.&&.&&.&&.&&.&&.&&.&&.&&.&\\
.&&.&&.&&.&&.&&.&&.&&.&&.&&.&&.&&.&&.&&.&&.&&.&&.&&.&&.&&.&&.&&.&&.&&.&&.&&.&&.&&.&&.&&.&\\
&{{}\drop\xycircle<4pt,4pt>{}.}&&.&&.&&.&&.&&.&&.&&.&&.&&.&&.&&.&&.&&.&&.&&.&&.&&.&&.&&.&&.&&.&&.&&.&&.&&.&&.&&.&&. \\
R&&.&&.&&.&&.&&.&&{{}\drop\xycircle<4pt,4pt>{}.}&&.&&.&&.&&{{}\drop\xycircle<4pt,4pt>{}.}&&.&&.&&.&&.&&{{}\drop\xycircle<4pt,4pt>{}.}&&.&&.&&.&&.&&.&&.&&.&&.&&.&&.&&.&&.&&.&&R
}
\end{array}}}
&\begin{array}{c}\tiny{\xymatrix@C=-2pt@R=-2pt{&2&\\3&2&5}}\end{array}&\begin{array}{c}\tiny{\xymatrix@C=-2pt@R=-2pt{&1&\\1&2&1}}\end{array}\\
&&&\\
\mathbb{I}_{30(b-2)+29}&A_1,B_1,C,M&\begin{array}{c}\tiny{\xymatrix@C=-2pt@R=-2pt{&2&\\3&b&5}}\end{array}&\begin{array}{c}\tiny{\xymatrix@C=-2pt@R=-2pt{&1&\\1&1&1}}\end{array}
\\
&&&\\
\end{array}
\]
}}
\end{thm}
It is possible to use this classification to assign to each indecomposable special CM module the corresponding exceptional curve in the minimal resolution.  Type $\mathbb{A}$ is well understood, for type $\mathbb{D}$ see \cite{Wemyss_reconstruct_D(i)} and \cite{Wemyss_reconstruct_D(ii)}, and for the remaining cases see version 1 of the paper \cite{Wemyss_GL2i}.


\begin{thebibliography}{Wem08b}
\bibitem[Art66]{Art}
M. Artin,
\emph{On isolated rational singularities of surfaces.}  Amer. J. Math.  \textbf{88} (1966) 129--136.

\bibitem[AV85]{ArtVer}
M. Artin and J-L Verdier
\emph{Reflexive modules over rational double points}. Math. Ann.  \textbf{270} (1985) 79--82.

\bibitem[Aus71]{A}
M. Auslander,
\emph{Representation dimension of Artin algebras}. Lecture notes,
Queen Mary College, London, 1971.

\bibitem[Aus78]{A2}
M. Auslander,
\emph{Functors and morphisms determined by objects}.  Representation
theory of algebras (Proc. Conf., Temple Univ., Philadelphia, Pa.,
1976), pp. 1--244. Lecture Notes in Pure Appl. Math., Vol. 37,
Dekker, New York, 1978.

\bibitem[Aus86]{Auslander_rational}
M. Auslander
\emph{Rational singularities and almost split sequences}. Trans. Amer. Math. Soc. \textbf{293} (1986), no.~2, 511--531.

\bibitem[AB69]{AB}
M. Auslander, M. Bridger,
\emph{Stable module theory. Memoirs of the American Mathematical
Society}, No. 94 American Mathematical Society, Providence, R.I.
1969 146 pp.

\bibitem[AR86]{AR_McKayGraphs}
M.~Auslander and I.~Reiten, \emph{{M}c{K}ay quivers and extended
{D}ynkin
  diagrams}, Trans. Amer. Math. Soc. \textbf{293} (1986), no.~1, 293--301.

\bibitem[Bri68]{Brieskorn}
E.~Brieskorn, \emph{Rationale singularit{\"{a}}ten komplexer
fl{\"{a}}chen},
  Invent. Math. \textbf{4} (1968), 336--358.



\bibitem[EHIS]{EHIS}
K. Erdmann, T. Holm, O. Iyama and J. Schr\"oer,
\emph{Radical embeddings and representation dimension}. Adv. Math.
\textbf{185} (2004), no. 1, 159--177.

\bibitem[EG85]{EG}
E. G. Evans and P. Griffith, \emph{Syzygies}.
London Mathematical Society Lecture Note Series, 106. Cambridge
University Press, Cambridge, 1985.


\bibitem[Gab80]{Gabriel}
P. Gabriel, \emph{Auslander-Reiten sequences and representation-finite
algebras}. Representation theory, I (Proc. Workshop, Carleton Univ., Ottawa, Ont., 1979), pp. 1--71,
Lecture Notes in Math., 831, Springer, Berlin, 1980. 

\bibitem[Ish02]{Ishii}
A.~Ishii, \emph{On the {M}c{K}ay correspondence for a finite small
subgroup of
  $\textnormal{{GL}}(2,\mathbb{C})$}, Journal fur die Reine und Angewandte
  Mathematik \textbf{549} (2002), 221--233.

\bibitem[Ito02]{Ito_special}
Y.~Ito, \emph{Special {M}c{K}ay correspondence.}, S\'{e}min. Congr.
\textbf{6}
  (2002), 213--225.

\bibitem[IT84]{IT1}
K. Igusa and G. Todorov, \emph{Radical layers of representable functors},
J. Algebra \textbf{89} (1984), no. 1, 105--147.

\bibitem[Iy05a]{I1}
O. Iyama, \emph{$\tau$-categories I: Ladders},
Algebr. Represent. Theory \textbf{8} (2005), no. 3, 297--321.

\bibitem[Iy05b]{I2}
O. Iyama, \emph{$\tau$-categories II: Nakayama pairs and Rejective subcategories},
Algebr. Represent. Theory \textbf{8} (2005), no. 4, 449--477.

\bibitem[Lau72]{Lau72}
H. Laufer, \emph{On rational singularities},
Amer. J. of Math., \textbf{94} (1972), 597--608.

\bibitem[MS04]{MS}
A. Martsinkovsky and J.R. Strooker, \emph{Linkage of modules}.  J. Algebra
\textbf{271} (2004), no. 2, 587--626.

\bibitem[McK80]{McKay_original}
J.~McKay, \emph{Graphs, singularities, and finite groups}, Proc.
Sympos. Pure
  Math. \textbf{37} (1980), 183--186.

\bibitem[NdC08]{Alvaro}
A.~Nolla de Celis \emph{Dihedral groups and $G$-Hilbert Schemes}, Warwick PhD thesis (Sep. 2008).


\bibitem[Rie77]{Riemenschneider_invarianten}
O.~Riemenschneider, \emph{Invarianten endlicher {U}ntergruppen},
Math. Z
  \textbf{153} (1977), 37--50.

\bibitem[Wem07]{Wemyss_reconstruct_A}
M.~Wemyss, \emph{Reconstruction algebras of type ${A}$},
arXiv:0704.3693
  (2007).

\bibitem[Wem08]{Wemyss_GL2i}
\bysame, \emph{The ${GL}(2)$ {M}c{K}ay correspondence},
arXiv:0809.1973 (version 1) (2008).

\bibitem[Wem09a]{Wemyss_reconstruct_D(i)}
\bysame, \emph{Reconstruction algebras of type ${D}$ ({I})}, in
preparation
  (2008).

\bibitem[Wem09b]{Wemyss_reconstruct_D(ii)}
\bysame, \emph{Reconstruction algebras of type ${D}$ ({II})}, in
preparation
  (2008).

\bibitem[Wun87]{Wunram_cyclicBook}
J.~Wunram, \emph{Reflexive modules on cyclic quotient surface
singularities},
  Lecture Notes in Mathematics, Springer-Verlag \textbf{1273} (1987), 221--231.

\bibitem[Wun88]{Wunram_generalpaper}
\bysame, \emph{Reflexive modules on quotient surface singularities},
  Mathematische Annalen \textbf{279} (1988), no.~4, 583--598.

\bibitem[Yos90]{Y}
Y. Yoshino, \emph{Cohen-Macaulay modules over Cohen-Macaulay rings},
London Mathematical Society Lecture Note Series, 146. Cambridge University Press, Cambridge, 1990.

\end{thebibliography}
\end{document}